\documentclass[11pt,twoside, leqno]{article}

\usepackage{amssymb}
\usepackage{amsmath}
\usepackage{mathrsfs}
\usepackage{amsthm}
\usepackage{txfonts}

\usepackage{color}
\usepackage{indentfirst}

\allowdisplaybreaks

\def\red{\color{red}}

\def\ls{\lesssim}

\def\fz{\infty}

\renewcommand{\r}{\right}
\newcommand{\lf}{\left}

\def\ls{\lesssim}

\def\supp{{\mathop\mathrm{\,supp\,}}}

\def\aa{{\mathbb A}}
\def\rr{{\mathbb R}}
\def\rh{{\mathbb R}{\mathbb H}}
\def\rn{{{\rr}^n}}
\def\zz{{\mathbb Z}}
\def\nn{{\mathbb N}}
\def\cc{{\mathbb C}}

\newcommand{\wz}{\widetilde}
\newcommand{\oz}{\overline}

\newcommand{\ca}{{\mathcal A}}

\newcommand{\cd}{{\mathcal D}}

\newcommand{\cm}{{\mathcal M}}

\newcommand{\cp}{{\mathcal P}}

\newcommand{\cs}{{\mathcal S}}

\newcommand{\cx}{{\mathcal X}}
\def\az{\alpha}
\def\lz{\lambda}
\def\blz{\Lambda}

\def\bfai{\Phi}
\def\dz{\delta}

\def\epz{\epsilon}

\def\bz{\beta}

\def\fai{\varphi}
\def\gz{{\gamma}}
\def\bgz{{\Gamma}}
\def\vz{\varphi}
\def\tz{\theta}
\def\sz{\sigma}

\def\wz{\widetilde}

\def\ls{\lesssim}

\def\ol{\overline}

\def\boz{\Omega}
\def\oz{\omega}

\def\esup{\mathop\mathrm{\,ess\,sup\,}}
\def\einf{\mathop\mathrm{\,ess\,inf\,}}

\def\divz{{{\mathop\mathrm {div}}}}

\def\hs{\hspace{0.3cm}}

\def\dint{\displaystyle\int}

\def\dsup{\displaystyle\sup}

\newtheorem{theorem}{Theorem}[section]
\newtheorem{lemma}[theorem]{Lemma}
\newtheorem{corollary}[theorem]{Corollary}
\newtheorem{proposition}[theorem]{Proposition}
\theoremstyle{definition}
\newtheorem{remark}[theorem]{Remark}
\newtheorem{definition}[theorem]{Definition}
\newtheorem{assumption}[theorem]{Assumption}

\def\supp{{\mathop\mathrm{\,supp\,}}}
\def\diam{{\mathop\mathrm{diam\,}}}
\def\dist{{\mathop\mathrm{\,dist\,}}}
\def\loc{{\mathop\mathrm{loc}}}

\def\lfz{\lfloor}
\def\rfz{\rfloor}
\numberwithin{equation}{section}

\begin{document}

\arraycolsep=1pt

\title{\Large\bf
Atomic and Maximal Function Characterizations of Musielak-Orlicz-Hardy Spaces Associated to Non-negative Self-adjoint
Operators on Spaces of Homogeneous Type
\footnotetext{\hspace{-0.35cm} 2010 {\it Mathematics Subject
Classification}. {Primary 42B25; Secondary 42B35, 46E30, 30L99.}
\endgraf{\it Key words and phrases}. Musielak-Orlicz-Hardy space,
atom, maximal function, non-negative self-adjoint operator, Gaussian upper bound estimate, space of
homogeneous type, strongly Lipschitz domain.
\endgraf This project is supported by the
National Natural Science Foundation  of China (Grant Nos. 11871254, 11571289, 11571039,  11761131002 and 11726621)
and the Fundamental Research Funds for the Central Universities (Grant No. lzujbky-2018-111).}}
\author{Sibei Yang and Dachun Yang\,\footnote{Corresponding author/{\red August 29, 2018}/Final version.}}
\date{ }
\maketitle

\vspace{-0.8cm}

\begin{center}
\begin{minipage}{13.5cm}\small
{{\bf Abstract.} Let $\mathcal{X}$ be a metric space
with doubling measure and $L$ a non-negative self-adjoint operator
on $L^2(\mathcal{X})$ whose heat kernels
satisfy the Gaussian upper bound estimates.
Assume that the growth function $\varphi:\ \mathcal{X}\times[0,\infty)
\to[0,\infty)$ satisfies that $\varphi(x,\cdot)$ is an Orlicz function and
$\varphi(\cdot,t)\in {\mathbb A}_{\infty}(\mathcal{X})$ (the class of
uniformly Muckenhoupt weights).  Let $H_{\varphi,\,L}(\mathcal{X})$
be the Musielak-Orlicz-Hardy space defined via the Lusin area
function associated with the heat semigroup of $L$.
In this article, the authors characterize
the space $H_{\varphi,\,L}(\mathcal{X})$ by means of atoms, non-tangential and radial maximal
functions associated with $L$. In particular, when $\mu(\mathcal{X})<\fz$,
the local non-tangential and radial maximal function characterizations of $H_{\varphi,\,L}(\mathcal{X})$
are obtained. As applications, the authors obtain various maximal
function and the atomic characterizations of the ``geometric" Musielak-Orlicz-Hardy spaces $H_{\varphi,\,r}(\Omega)$
and $H_{\varphi,\,z}(\Omega)$ on the strongly Lipschitz domain $\Omega$ in $\mathbb{R}^n$
associated with second-order self-adjoint elliptic operators
with the Dirichlet and the Neumann boundary conditions; even when $\varphi(x,t):=t$ for any $x\in\mathbb{R}^n$
and $t\in[0,\infty)$, the equivalent characterizations of
$H_{\varphi,\,z}(\Omega)$ given in this article improve the known results
via removing the assumption that $\Omega$ is unbounded.}
\end{minipage}
\end{center}

\vspace{0.1cm}

\section{Introduction\label{s1}}

 Let $\cx$ be a \emph{set}, $d$ a \emph{metric} on $\cx$ and $\mu$ a \emph{non-negative Borel regular
measure} on $\cx$. For any $x\in\cx$ and $r\in(0,\fz)$, let
$B(x,r):=\{y\in\cx:\ d(x,y)<r\}$
and $V(x,r):=\mu(B(x,r))$.
Moreover, we assume that there exists a constant
$C_1\in[1,\fz)$ such that, for any $x\in\cx$ and $r\in(0,\fz)$,
\begin{equation}\label{1.1}
V(x,2r)\le C_1V(x,r)<\fz.
\end{equation}

Observe that $(\cx,\,d,\,\mu)$ is a \emph{space of homogeneous type} in
the sense of Coifman and Weiss \cite{cw71,cw77}. Recall that, in the
definition of spaces of homogeneous type in \cite[Chapter 3]{cw71}, $d$ is assumed to be a
quasi-metric. But, to simplify the presentation of this article, we restrict that $d$
is a metric in the whole article. Notice that the doubling property
\eqref{1.1} implies that the following strong homogeneity property
that, for some positive constants $C$ and $n$,
\begin{equation}\label{1.2}
V(x,\lz r)\le C\lz^n V(x,r)
\end{equation}
uniformly holds true for any $\lz\in[1,\fz)$, $x\in\cx$ and $r\in(0,\fz)$. There also
exist constants $C\in(0,\fz)$ and $n_0\in[0,n]$ such that, for any
$x,\,y\in\cx$ and $r\in(0,\fz)$,
$$V(x,r)\le C\lf[1+\frac{d(x,y)}{r}\r]^{n_0}V(y,r).$$
It is worth pointing out that, in the cases of Euclidean spaces $\rn$,
strongly Lipschitz domains in $\rn$ and Lie groups of polynomial growth,
$n_0$ can be chosen to be 0.

The main purposes of this article are to obtain the
atomic characterization and several maximal function characterizations
of Musielak-Orlicz-Hardy spaces associated with  non-negative self-adjoint
operators satisfying the Gaussian upper bound estimates on the space of homogeneous type $\cx$,
particularly for the case that $\mu(\cx)<\fz$, and give some applications.
We point out that, even when going back to the special case of Hardy spaces $H^p_L(\cx)\ (p\in(0,1])$,
the results obtained in this article improve the known results via removing
the \emph{restriction that $\mu(\cx)=\fz$} (see Remark \ref{r1.2} below).
It is worth pointing out that, by using the atomic decomposition of the tent space
on $\cx$ when $\mu(\cx)<\fz$, obtained in this article (see Lemma \ref{l2.4} below for the details),
the atomic characterization of the (Musielak-)(Orlicz-)Hardy space
associated with the non-negative self-adjoint operator $L$ on $L^2(\cx)$
satisfying Davies-Gaffney estimates (or reinforced off-diagonal estimates),
established in \cite{bckyy13b,hlmmy,jy11,yys4},
could be improved via removing the \emph{assumption that $\mu(\cx)=\fz$} (see Remark \ref{r1.1} below).

Recall that the real-variable theory of Hardy spaces on the $n$-dimensional
Euclidean space $\rn$, initiated by Stein and Weiss \cite{sw60} and then systematically developed
by Fefferman and Stein \cite{fs72}, plays important roles in various fields of analysis
and partial differential equations (see, for example, \cite{fs72,m94,sw60}).
It is well known that the Hardy space $H^p(\rn)$, with $p\in(0,1]$,
is a suitable substitute of the Lebesgue space $L^p(\rn)$; for example, the
classical Riesz transform is bounded on $H^p(\rn)$, but not on $L^p(\rn)$
when $p\in(0,1]$. Moreover, it is worth pointing out that the classical
Hardy space $H^p(\rn)$ is essentially
related to the Laplace operator on $\rn$. However, in many settings,
these classical function spaces are not applicable;
for example, the Riesz transforms $\nabla L^{-1/2}$ may not be bounded
from the Hardy space $H^1(\rn)$ to $L^1(\rn)$ when $L$ is a second-order divergence form
elliptic operator with complex bounded measurable coefficients on $\rn$ (see, for example,
\cite{hm09,hmm11}). Motivated by this, the study for the real-variable
theory of various function spaces on $\rn$ or domains in $\rn$, especially,
the Hardy-type spaces, associated with different differential operators,
has inspired great interests in recent years; see, for example, \cite{adm,ar03,dl13,dy04,dz02,hm09,hll18,jyy12,sy16,y08}
for the case of Hardy spaces, \cite{bckyy,ls13,sy10} for the case of weighted Hardy spaces, \cite{abdr17,yzz18,yz16,zy16} for the case of variable exponent Hardy spaces
and \cite{al11,jy10,jy11,yy14,yys2,yys3,yys16} for the case of (Musielak-)Orlicz Hardy spaces. Recall that the Musielak-Orlicz space was originated by Nakano \cite{n50} and
developed by Musielak and Orlicz \cite{m83,mo59}, which is a natural generalization of
many important spaces such as (weighted) Lebesgue spaces, variable Lebesgue spaces
and Orlicz spaces and not only has its own interest, but is also very useful in partial differential equations
\cite{ap16,bcg15na,hhk16,GSZ}, in calculus of variations \cite{cm15arma}, in
image restoration \cite{hhlt13,ana14} and in fluid dynamics \cite{sg14,mw}.
The Musielak-Orlicz Hardy space on $\rn$ has proved useful in harmonic analysis
(see, for example, \cite{k14,CCYY16tams,LY15,ylk17})
and, especially, naturally appears in the endpoint estimate for both the div-curl lemma and the commutator of
Caldr\'{o}n-Zygmund operators (see \cite{bgk12,bijz07,ylk17}).
Moreover, the real-variable theory of Hardy-type spaces associated with
operators on the space of homogeneous type was also widely studied (see, for example,
\cite{bckyy13b,bd18,bdl1,bdl,dkkp17,hhlt18,hlmmy,sy18,yys4,jy11,yz18}).

Let $\cx$ be a space of homogeneous type and $L$ a non-negative self-adjoint operator
on $L^2(\cx)$ whose heat semigroup satisfies the so-called Davies-Gaffney estimates
(see, for example, \cite{hlmmy}). \emph{Under the assumption that $\mu(\cx)=\fz$},
the equivalent characterizations of the Hardy space
$H^1_L(\cx)$ associated with $L$, including the atom, the molecule and the
Lusin area function associated with $L$, were established in \cite{hlmmy}, which were extended
to the (Musielak-)Orlicz-Hardy space in \cite{jy11,yys4}. As a special case of those operators,
when $\cx:=\rn$ and
$L:=-\Delta+V$ is the Schr\"odinger operator with $0\le V\in L^1_{\loc}(\rn)$,
the non-tangential maximal function and the radial maximal function characterizations, associated with $L$,
of the Hardy space $H^1_L(\rn)$, the Orlicz-Hardy space
$H_{\Phi,\,L}(\rn)$ and the Musielak-Orlicz-Hardy space $H_{\fai,\,L}(\rn)$
were obtained, respectively, in \cite{hlmmy},
\cite{jy11} and \cite{bckyy13b,yys4}. Recall that, when $L:=-\Delta+V$ with $0\le V\in L^1_{\loc}(\rn)$,
its heat semigroup satisfies the  Gaussian upper bound estimates (see, for example, \cite[(8.4)]{hlmmy}).

Assume further that $L$ is a non-negative self-adjoint operator
on $L^2(\cx)$ whose heat kernels satisfy the Gaussian upper bound estimates.
When $\cx:=\rn$, via borrowing some ideas from \cite{c77},
the non-tangential maximal function characterization associated with $L$ of
the Hardy space $H^p_L(\rn)$, with $p\in(0,1]$, was established in \cite{sy16},
which was extended to the Musielak-Orlicz-Hardy space $H_{\fai,\,L}(\rn)$ in \cite{yys16}.
Moreover, under some additional assumptions for $L$, the radial maximal
function characterization of $H_{\fai,\,L}(\rn)$ was also obtained in \cite{yys16}.
Recently,
$$\emph{under the assumption that $\cx$ is a space of homogeneous type with
$\mu(\cx)=\fz$},$$
the non-tangential maximal function and the radial maximal function characterizations,
associated with $L$, of the Hardy space $H^p_L(\cx)$ ($p\in(0,1]$) were established in \cite{sy18},
which improves the results obtained in \cite{sy16} even when $\cx:=\rn$.
Moreover, the equivalent characterizations of the Hardy space
$H^p_L(\cx)$ ($p\in(0,1]$) associated with $L$, including the atom, the non-tangential maximal function
and the radial maximal function associated with $L$, were established in \cite{bdl1}
via a different method from that in \cite{sy18}, which improve the results in \cite{sy18}
by removing the assumption that $\mu(\cx)=\fz$. Recently, the atomic and the maximal function
characterizations of the weighted Hardy space $H^p_{L,\,\omega}(\cx)$, with any $p\in(0,1]$,
associated with $L$, were obtained in \cite{bd18} via applying the method used in
\cite{bdl1}, without assuming $\mu(\cx)=\fz$, where $\omega\in A_\fz(\cx)$.
Here and hereafter, $A_q(\cx)$ with $q\in[1,\fz]$ denotes the
\emph{class of Muckenhoupt weights} (see, for example, \cite{gra1}).

Let $\cx$ be a space of homogeneous type and $L$ a non-negative self-adjoint operator
on $L^2(\cx)$ whose heat kernels satisfy the Gaussian upper bound estimates. Motivated
by \cite{bdl1,sy16,sy18,yys16}, in this article, we establish
the equivalent characterizations of the Musielak-Orlicz-Hardy space $H_{\varphi,\,L}(\mathcal{X})$,
in terms of the atom, the non-tangential maximal function and the radial
maximal function associated with $L$,
which improve the corresponding results for the space $H_{\varphi,\,L}(\mathcal{X})$ established in
\cite{bckyy13b,jy11,yys4} via removing the assumption that $\mu(\cx)=\fz$.
In particular, when $\mu(\mathcal{X})<\fz$,
the local non-tangential and radial maximal function characterizations of $H_{\varphi,\,L}(\mathcal{X})$
are obtained, which imply that the global Musielak-Orlicz-Hardy space
$H_{\varphi,\,L}(\mathcal{X})$ and the local Musielak-Orlicz-Hardy space
$h_{\fai,\,L}(\cx)$ are equivalent when $\mu(\cx)<\fz$.
We point out that even in the special case of Hardy spaces $H^p_L(\cx)\ (p\in(0,1])$,
the results obtained in this article improve the known results via removing
the \emph{assumption that $\mu(\cx)=\fz$} (see Remark \ref{r1.2} below).
As applications, the ``geometric" Musielak-Orlicz-Hardy spaces $H_{\varphi,\,r}(\Omega)$
and $H_{\varphi,\,z}(\Omega)$ on the strongly Lipschitz domain $\Omega$ in $\mathbb{R}^n$ are characterized
via several maximal functions and atoms associated with second-order self-adjoint elliptic operators
with the Dirichlet and the Neumann boundary conditions. Some results obtained in this article are new
even when $\mu(\cx)<\fz$ and $\varphi(x,t):=t^p$, with $p\in(0,1]$,
for any $x\in\mathcal{X}$ and $t\in[0,\infty)$
(see Remark \ref{r1.2}(i) below for the details).
Moreover, it is worth pointing out that the results obtained for the space $H_{\fai,\,z}(\boz)$
improve the corresponding results established in \cite{ar03,yys2,yys3} by removing the additional assumption
that $\boz$ is unbounded.

To state our main results, in the remainder of the whole article,
we \emph{always assume} that $L$ is a densely defined linear operator on $L^2(\cx)$
satisfying the following two assumptions:

\begin{assumption}\label{a1}
$L$ is non-negative and self-adjoint.
\end{assumption}

\begin{assumption}\label{a2}
The kernels of the semigroup $\{e^{-tL}\}_{t>0}$, denoted by $\{K_t\}_{t>0}$,
  are measurable functions on $\cx\times\cx$ and satisfy the Gaussian upper bound estimates, namely,
  there exist positive constants $C_2$ and $c_2$ such that, for any $t\in(0,\fz)$ and $x,\,y\in\cx$,
  \begin{equation}\label{1.3}
  \lf|K_t(x,y)\r|\le\frac{C_2}{V(x,\sqrt{t})}\exp\lf\{-\frac{[d(x,y)]^2}{c_2t}\r\}.
  \end{equation}
\end{assumption}

Examples of operators satisfying Assumptions \ref{a1} and \ref{a2} include both second-order self-adjoint
elliptic operators in divergence form with bounded measurable coefficients and
(degenerate) Schr\"odinger operators with non-negative potentials or with magnetic fields
on $\rn$ or strongly Lipschitz domains, and Laplace-Beltrami operators on Heisenberg groups, connected and simply connected
nilpotent Lie groups or complete Riemannian manifolds
(see, for example, \cite{ar03,bdl1,bdl,d89,s95,sy18}).

Now we describe the Musielak-Orlicz function considered in this article as follows.
Recall that a function $\Phi:\ [0,\fz)\to[0,\fz)$ is called an \emph{Orlicz function}
if it is non-decreasing, $\Phi(0)=0$, $\Phi(t)>0$ for any $t\in(0,\fz)$ and
$\lim_{t\to\fz}\Phi(t)=\fz$ (see, for example,
\cite{m83,rr91,ylk17}). We point out that, differently from the classical definition of
Orlicz functions, the \emph{Orlicz functions in this article may not be convex}.
Moreover, $\Phi$ is said to be of \emph{upper} (resp. \emph{lower})
\emph{type $p$} for some $p\in(0,\fz)$ if
there exists a positive constant $C$ such that, for any
$s\in[1,\fz)$ (resp. $s\in[0,1]$) and $t\in[0,\fz)$,
$\Phi(st)\le Cs^p \Phi(t).$

For a given function $\fai:\ \cx\times[0,\fz)\to[0,\fz)$ such that, for
any $x\in\cx$, $\fai(x,\cdot)$ is an Orlicz function,
$\fai$ is said to be of \emph{uniformly upper} (resp.
\emph{lower}) \emph{type $p$}  for some $p\in(0,\fz)$ if there
exists a positive constant $C$ such that, for any $x\in\cx$,
$t\in[0,\fz)$ and $s\in[1,\fz)$ (resp. $s\in[0,1]$),
$\fai(x,st)\le Cs^p\fai(x,t)$.
Let
\begin{equation}\label{1.4}
I(\fai):=\inf\lf\{p\in(0,\fz):\ \fai\ \text{is of uniformly upper
type}\ p\r\}
\end{equation}
and
\begin{equation}\label{1.5}
i(\fai):=\sup\lf\{p\in(0,\fz):\ \fai\ \text{is of uniformly lower
type}\ p\r\}.
\end{equation}
In what follows, $I(\fai)$ and $i(\fai)$ are
called, respectively, the \emph{uniformly critical
upper type index} and the \emph{uniformly critical lower type index} of $\fai$.
Observe that $I(\fai)$ and $i(\fai)$ may not be attainable, namely, $\fai$ may not
be of uniformly upper type $I(\fai)$ or uniformly lower type $i(\fai)$
(see, for example, \cite{bckyy13b,hyy,k14,ylk17,yys1,yys4} for some examples).
Moreover, it is easy to see that, if $\fai$ is of uniformly upper type $p_0\in(0,\fz)$
and lower type $p_1\in(0,\fz)$, then $p_0\ge p_1$ and hence $I(\fai)\ge i(\fai)$.

\begin{definition}\label{d1.1}
Assume that $\cx$  is a space of homogeneous type.
Let $\fai:\ \cx\times[0,\fz)\to[0,\fz)$ satisfy that
$\fai(\cdot,t)$ is measurable for any $t\in[0,\fz)$.
Then $\fai$ is said to satisfy the
\emph{uniformly Muckenhoupt condition for some $q\in[1,\fz)$},
denoted by $\fai\in\aa_q(\cx)$, if, when $q\in (1,\fz)$,
\begin{equation*}
\aa_q (\fai):=\sup_{t\in
(0,\fz)}\sup_{B\subset\cx}\lf\{\frac{1}{V(B)}\int_B
\fai(x,t)\,d\mu(x)\r\}\lf\{\frac{1}{V(B)}\int_B
[\fai(y,t)]^{1-q}\,d\mu(y)\r\}^{q-1}<\fz
\end{equation*}
or
\begin{equation*}
\aa_1 (\fai):=\sup_{t\in (0,\fz)}
\sup_{B\subset\cx}\lf\{\frac{1}{V(B)}\int_B \fai(x,t)\,d\mu(x)\r\}
\lf\{\esup_{y\in B}[\fai(y,t)]^{-1}\r\}<\fz,
\end{equation*}
where the first suprema are taken over all $t\in(0,\fz)$ and the
second ones over all balls $B\subset\cx$.

The function $\fai$ is said to satisfy the
\emph{uniformly reverse H\"older condition for some
$q\in(1,\fz]$}, denoted by $\fai\in \rh_q(\cx)$, if, when $q\in(1,\fz)$,
\begin{align*}
\rh_q (\fai):=\sup_{t\in(0,\fz)}\sup_{B\subset\cx}\lf\{\frac{1}
{V(B)}\int_B [\fai(x,t)]^q\,d\mu(x)\r\}^{1/q}\lf\{\frac{1}{V(B)}\int_B
\fai(x,t)\,d\mu(x)\r\}^{-1}<\fz
\end{align*}
or
\begin{equation*}
\rh_{\fz} (\fai):=\sup_{t\in(0,\fz)}\sup_{B\subset\cx}\lf\{\esup_{y\in
B}\fai(y,t)\r\}\lf\{\frac{1}{V(B)}\int_B
\fai(x,t)\,d\mu(x)\r\}^{-1} <\fz,
\end{equation*}
where the first suprema are taken over all $t\in(0,\fz)$ and the
second ones over all balls $B\subset\cx$.
\end{definition}

Recall that, in Definition \ref{d1.1},
$\aa_p(\cx)$,  with $p\in[1,\fz)$, and $\rh_q(\cx)$, with $q\in(1,\fz]$,
were introduced in \cite{yys4} (see \cite{k14} or \cite{ylk17} for the Euclidean case $\rn$).
Let $\aa_{\fz}(\cx):=\cup_{q\in[1,\fz)}\aa_{q}(\cx)$.
We now recall the notions of the \emph{critical indices} for $\fai\in\aa_{\fz}(\cx)$ as follows:
\begin{equation}\label{1.6}
q(\fai):=\inf\lf\{q\in[1,\fz):\ \fai\in\aa_{q}(\cx)\r\}
\end{equation}
and
\begin{equation}\label{1.7}
r(\fai):=\sup\lf\{q\in(1,\fz]:\ \fai\in\rh_{q}(\cx)\r\}.
\end{equation}
Recall also that, if $q(\fai)\in(1,\fz)$, then, by \cite[Lemma 2.4(iii)]{hyy},
we know that $\fai\not\in\aa_{q(\fai)}(\cx)$ and there exists
$\fai\not\in \aa_1(\cx)$ such that $q(\fai)=1$
(see, for example, \cite{jn87}).
Similarly, if $r(\fai)\in(1,\fz)$, then, by \cite[Lemma 2.3(iv)]{yy14},
we know that $\fai\not\in\rh_{r(\fai)}(\cx)$
and there exists $\fai\not\in\rh_\fz(\cx)$
such that $r(\fai)=\fz$ (see, for example, \cite{cn95}).

Now we recall the notion of growth functions from Ky \cite{k14} (see also \cite{ylk17}).

\begin{definition}\label{d1.2}
Let $\cx$ be a space of homogeneous type. A function $\fai:\ \cx\times[0,\fz)\rightarrow[0,\fz)$ is called
 a \emph{growth function} if the following hold true:
 \vspace{-0.25cm}
\begin{enumerate}
\item[(i)] $\fai$ is a \emph{Musielak-Orlicz function}, namely,
\vspace{-0.2cm}
\begin{enumerate}
    \item[(a)] $\fai(x,\cdot)$ is an
    Orlicz function for any $x\in\cx$;
    \vspace{-0.2cm}
    \item [(b)] $\fai(\cdot,t)$ is a measurable
    function for any $t\in[0,\fz)$.
\end{enumerate}
\vspace{-0.25cm} \item[(ii)] $\fai\in \aa_{\fz}(\cx)$.
\vspace{-0.25cm} \item[(iii)] The function $\fai$ is of
uniformly lower type $p$ for some $p\in(0,1]$ and of uniformly
upper type 1.
\end{enumerate}
\end{definition}

For a Musielak-Orlicz function $\fai$ as in Definition \ref{d1.2},
a measurable function $f$ on $\cx$ is said to be in the \emph{Musielak-Orlicz space}
$L^{\fai}(\cx)$ if $\int_{\cx}\fai(x,|f(x)|)\,d\mu(x)<\fz$. Moreover,
for any $f\in L^{\fai}(\cx)$, the \emph{Luxemburg} (also called the
\emph{Luxemburg-Nakano}) \emph{quasi-norm} $\|f\|_{L^\fai(\cx)}$ of $f$
is defined by setting
$$\|f\|_{L^{\fai}(\cx)}:=\inf\lf\{\lz\in(0,\fz):\
\int_{\cx}\fai\lf(x,\frac{|f(x)|}{\lz}\r)\,d\mu(x)\le1\r\}.$$

Clearly, for any $x\in\cx$ and $t\in[0,\fz)$,
\begin{equation}\label{1.8}
\fai(x,t):=\oz(x)\Phi(t)
\end{equation} is a growth function if
$\oz\in A_{\fz}(\cx)$ and $\Phi$ is an Orlicz function of lower
type $p$ for some $p\in(0,1]$ and upper type 1.
A typical example of such functions $\Phi$ is $\Phi(t):=t^p$,
with $p\in(0,1]$, for any $t\in [0,\fz)$
(see, for example, \cite{hyy,k13,k14,ylk17,yys4} for more examples of such $\Phi$).
Let $x_0\in\cx$. Another typical example $\fai$ of growth functions is given by setting
\begin{equation}\label{1.9}
\fai(x,t):=\frac{t}{\ln(e+d(x,x_0))+\ln(e+t)}
\end{equation}
for any $x\in\cx$ and $t\in[0,\fz)$ (see, for example, \cite{bgk12,k13,k14}). More precisely,
$\fai\in \aa_1(\cx)$, $\fai$ is of uniformly upper type 1 (indeed, $I(\fai)=1$,
which is attainable) and $i(\fai)=1$ which is not attainable (see \cite{k14,yys4} for the details).
It is worth pointing out that some special Musielak-Orlicz-Hardy spaces appear naturally
in the study of the products of functions in $H^1(\cx)$ and $\mathrm{BMO}(\cx)$
(see, for example, \cite{bgk12,bijz07,fyl17,fcy17,fy17,lyy}), and the endpoint estimates
for both the div-curl lemma (see, for example, \cite{bfg10,bijz07}) and the commutators
of Calder\'on-Zygmund operators (see, for example, \cite{k13,lcfy17,lcfy18}).

Now we recall the notion of the Musielak-Orlicz-Hardy space $H_{\fai,\,L}(\cx)$
introduced in \cite{bckyy13b,yys4}.

\begin{definition}\label{d1.3}
Assume that $\cx$  is a space of homogeneous type.
Let $L$ be an operator on $L^2(\cx)$ satisfying Assumptions \ref{a1} and \ref{a2},
and $\fai$ as in Definition \ref{d1.2}.
For any $f\in L^2(\cx)$ and $x\in\cx$, the \emph{Lusin area
function, $S_{L}(f)$, associated with $L$} is defined by setting
\begin{equation*}
S_{L}(f)(x):=\lf\{\int_{\bgz(x)}\lf|t^2 Le^{-t^2L}(f)(y)\r|^2\frac{d\mu(y)\,dt}{V(x,t)t}\r\}^{1/2}.
\end{equation*}

A function $f\in L^2 (\cx)$ is said to be in the set $\mathbb{H}_{\fai,\,L}(\cx)$ if $S_L(f)\in
L^{\fai}(\cx)$; moreover, define
$\|f\|_{H_{\fai,\,L}(\cx)}:=\|S_{L}(f)\|_{L^{\fai}(\cx)}$. Then the \emph{Musielak-Orlicz-Hardy space}
$H_{\fai,\,L}(\cx)$ is defined to be the completion of $\mathbb{H}_{\fai,\,L}(\cx)$ with respect to the
quasi-norm $\|\cdot\|_{H_{\fai,\,L}(\cx)}$.
\end{definition}

Moreover, we introduce the notions of both $(\fai,\,q,\,M)_L$-atoms and
atomic Musielak-Orlicz-Hardy spaces $H^{M,\,q}_{\fai,\,L,\,\mathrm{at}}(\cx)$ as follows.
In what follows, for any subset $E$ of $\cx$, we use $\mathbf{1}_E$ to denote its
\emph{characteristic function}.

\begin{definition}\label{d1.4}
Let $\cx$, $L$ and $\fai$ be as in Definition \ref{d1.3}.
Assume that $q\in(1,\fz]$, $M\in\nn$ and $B\subset\cx$ is a ball.
\begin{itemize}
  \item[{\rm(I)}] Let $\cd(L^M)$ be the domain of $L^M$.
A function $\az\in L^q(\cx)$ is called a $(\fai,\,q,\,M)_L$-\emph{atom}
associated with the ball $B$ if there exists a function $b\in\cd(L^M)$
such that
\begin{itemize}
  \item[(i)] $\az=L^M b$;
  \item[(ii)] for any $j\in\{0,\,1,\,\ldots,\,M\}$, $\supp(L^j b)\subset B$;
  \item[(iii)] for any $j\in\{0,\,1,\,\ldots,\,M\}$,
  $\|(r^2_BL)^jb\|_{L^q(\cx)}\le r^{2M}_B[V(B)]^{1/q}\|\mathbf{1}_B\|^{-1}_{L^\fai(\cx)}$,
  where $r_B$ denotes the radius of $B$.
\end{itemize}

In particular, if $\mu(\cx)<\fz$, a function $\az$ on $\cx$
is called a \emph{$(\fai,\,q)$-single-atom} if
$$\|\az\|_{L^q(\cx)}\le[\mu(\cx)]^{1/q}\|\mathbf{1}_\cx\|^{-1}_{L^\fai(\cx)}.$$
  \item[{\rm(II)}] When $\mu(\cx)=\fz$, for any $f\in L^2(\cx)$,
\begin{align}\label{1.10}
f=\sum_{j=1}^\fz\lz_j\az_j
\end{align}
is called an \emph{atomic} $(\fai,\,q,\,M)_L$-\emph{representation}
of the function $f$ if, for any $j\in\nn$, $\az_j$ is a
$(\fai,\,q,\,M)_L$-atom associated with the ball $B_j\subset\cx$,
the summation \eqref{1.10} converges in $L^2(\cx)$ and $\{\lz_j\}_j\subset\cc$ satisfies
that $\sum_{j=1}^\fz\fai(B_j,|\lz_j|\|\mathbf{1}_{B_j}\|^{-1}_{L^\fai(\cx)})<\fz$.

When $\mu(\cx)<\fz$, for any $f\in L^2(\cx)$,
\begin{align}\label{1.11}
f=\lz_0\az_0+\sum_{j=1}^\fz\lz_j\az_j
\end{align}
is called an \emph{atomic} $(\fai,\,q,\,M)_L$-\emph{representation}
of $f$ if $\az_0$ is a $(\fai,\,q)$-single-atom, $\az_j$ for any $j\in\nn$ is a
$(\fai,\,q,\,M)_L$-atom associated with the ball $B_j\subset\cx$,
the summation \eqref{1.11} converges in $L^2(\cx)$ and $\{\lz_0\}\cup\{\lz_j\}_{j=1}^\fz\subset\cc$ satisfies
that
$$\fai\lf(\cx,|\lz_0|\|\mathbf{1}_{\cx}\|^{-1}_{L^\fai(\cx)}\r)+\sum_{j=1}^\fz
\fai\lf(B_j,|\lz_j|\|\mathbf{1}_{B_j}\|^{-1}_{L^\fai(\cx)}\r)<\fz.$$

Let
$$\mathbb{H}^{M,\,q}_{\fai,\,L,\,\mathrm{at}}(\cx):=\lf\{f\in L^2(\cx):\ f\ \text{has an atomic}\
(\fai,\,q,\,M)_L\text{-representation}\r\}
$$
equipped with the \emph{quasi-norm}
\begin{align*}
\|f\|_{H^{M,\,q}_{\fai,\,L,\,\mathrm{at}}(\cx)}
:=\inf\lf\{
\blz\lf(\lf\{\lz_j\az_j\r\}_{j=0}^\fz\r):\ \sum_{j=0}^\fz\lz_j\az_j\
\text{is a}\ (\fai,\,q,\,M)_L\text{-representation of}\ f\r\},
\end{align*}
where
\begin{align*}
\blz\lf(\lf\{\lz_j\az_j\r\}_{j=0}^\fz\r)
:=\inf\lf\{\lz\in(0,\fz):\ \fai\lf(\cx,\frac{|\lz_0|}
{\lz\|\mathbf{1}_{\cx}\|_{L^\fai(\cx)}}\r)+\sum_{j=1}^\fz\fai\lf(B_j,\frac{|\lz_j|}
{\lz\|\mathbf{1}_{B_j}\|_{L^\fai(\cx)}}\r)\le1\r\}
\end{align*}
(when $\mu(\cx)=\fz$, $\lz_0:=0$) and the infimum is taken over all the atomic
$(\fai,\,q,\,M)_L$-representations of $f$ as above.
The \emph{atomic Musielak-Orlicz-Hardy space} $H^{M,\,q}_{\fai,\,L,\,\mathrm{at}}(\cx)$
is then defined as the completion of the set $\mathbb{H}^{M,\,q}_{\fai,\,L,\,\mathrm{at}}(\cx)$
with respect to the quasi-norm $\|\cdot\|_{H^{M,\,q}_{\fai,\,L,\,\mathrm{at}}(\cx)}$.
\end{itemize}
\end{definition}

We point out that, if $\mu(\cx)=\fz$, the atomic Musielak-Orlicz-Hardy
space $H^{M,\,q}_{\fai,\,L,\,\mathrm{at}}(\cx)$ in Definition \ref{d1.4} coincides
with the atomic Musielak-Orlicz-Hardy space introduced in \cite[Definition 5.3]{bckyy13b}.

Now we introduce the notions of Musielak-Orlicz-Hardy spaces via maximal
functions associated with the operator $L$.

\begin{definition}\label{d1.5}
Let $\cx$, $L$ and $\fai$ be as in Definition \ref{d1.3}.
Assume that $\phi\in\cs(\rr)$ is an even function with $\phi(0)=1$ and
$\az\in(0,\fz)$.
\begin{itemize}
  \item[{\rm(i)}] For any $f\in L^2(\cx)$ and $x\in\cx$, the \emph{non-tangential maximal function}
$\phi^{\ast}_{L,\,\az}(f)$ is defined by setting
\begin{align*}
\phi^{\ast}_{L,\,\az}(f)(x):=\sup_{d(x,y)<\az t,\,t\in(0,\fz)}\lf|\phi(t\sqrt{L})(f)(y)\r|.
\end{align*}
A function $f\in L^2 (\cx)$ is said to be in the set $\mathbb{H}^{\phi,\,\az}_{\fai,\,L,\,\mathrm{max}}(\cx)$ if $\phi^{\ast}_{L,\,\az}(f)\in L^{\fai}(\cx)$; moreover, define
$\|f\|_{H^{\phi,\,\az}_{\fai,\,L,\,\mathrm{max}}(\cx)}:=\|\phi^{\ast}_{L,\,\az}(f)\|_{L^{\fai}(\cx)}$.
Then the \emph{Musielak-Orlicz-Hardy space} $H^{\phi,\,\az}_{\fai,\,L,\,\mathrm{max}}(\cx)$ is defined to
be the completion of $\mathbb{H}^{\phi,\,\az}_{\fai,\,L,\,\mathrm{max}}(\cx)$ with respect to the
quasi-norm $\|\cdot\|_{H^{\phi,\,\az}_{\fai,\,L,\,\mathrm{max}}(\cx)}$.

Specially, if $\phi(x):=e^{-x^2}$ for any $x\in\rr$ and $\az:=1$, denote $\phi^{\ast}_{L,\,\az}(f)$
simply by $f^\ast_L$ and, in this case, denote the space
$H^{\phi,\,\az}_{\fai,\,L,\,\mathrm{max}}(\cx)$ simply by $H_{\fai,\,L,\,\mathrm{max}}(\cx)$.
  \item[{\rm(ii)}] For any $f\in L^2(\cx)$ and $x\in\cx$, the \emph{radial maximal function}
$\phi^+_{L}(f)$ is defined by setting
\begin{align*}
\phi^{+}_{L}(f)(x):=\sup_{t\in(0,\fz)}\lf|\phi(t\sqrt{L})(f)(x)\r|.
\end{align*}
A function $f\in L^2 (\cx)$ is said to be in the set $\mathbb{H}^{\phi}_{\fai,\,L,\,\mathrm{rad}}(\cx)$
if $\phi^+_{L}(f)\in L^{\fai}(\cx)$; moreover, define
$\|f\|_{H^{\phi}_{\fai,\,L,\,\mathrm{rad}}(\cx)}:=\|\phi^+_{L}(f)\|_{L^{\fai}(\cx)}$.
Then the \emph{Musielak-Orlicz-Hardy space} $H^{\phi}_{\fai,\,L,\,\mathrm{rad}}(\cx)$
is defined via replacing $\phi^{\ast}_{L,\,\az}(f)$ by $\phi^+_{L}(f)$ in the definition of
the space $H^{\phi,\,\az}_{\fai,\,L,\,\mathrm{max}}(\cx)$.

If $\phi(x):=e^{-x^2}$ for any $x\in\rr$, denote $\phi^+_{L}(f)$
simply by $f^+_L$ and, in this case, denote the space
$H^{\phi}_{\fai,\,L,\,\mathrm{rad}}(\cx)$ simply by $H_{\fai,\,L,\,\mathrm{rad}}(\cx)$.
\end{itemize}
\end{definition}

Then the first main result of this article reads as follows.

\begin{theorem}\label{t1.1}
Assume that $\cx$ is a space of homogeneous type.
Let $L$ be an operator on $L^2(\cx)$ satisfying Assumptions \ref{a1} and \ref{a2},
and $\fai$ as in Definition \ref{d1.2}. Assume that $r(\fai)$, $I(\fai)$, $q(\fai)$
and $i(\fai)$ are, respectively, as in \eqref{1.7}, \eqref{1.4}, \eqref{1.6} and \eqref{1.5},
and $[r(\fai)]'$ denotes the conjugate exponent of $r(\fai)$.
For any $q\in([r(\fai)]'I(\fai),\fz]\cap(1,\fz]$,
$M\in\nn\cap(\frac{nq(\fai)}{2i(\fai)},\fz)$ and $\az\in(0,\fz)$,
the Musielak-Orlicz-Hardy spaces $H^{M,\,q}_{\fai,\,L,\,\mathrm{at}}(\cx)$,
$H^{\phi,\,\az}_{\fai,\,L,\,\mathrm{max}}(\cx)$, $H^{\phi}_{\fai,\,L,\,\mathrm{rad}}(\cx)$
and $H_{\fai,\,L}(\cx)$ coincide with equivalent quasi-norms.
\end{theorem}

The strategy of the proof of Theorem \ref{t1.1} is as follows.
When $\mu(\cx)=\fz$, the equivalence of the spaces $H_{\fai,\,L}(\cx)$ and
$H^{M,\,q}_{\fai,\,L,\,\mathrm{at}}(\cx)$ was established in \cite[Theorem 5.4]{bckyy13b}.
When $\mu(\cx)<\fz$, for any $k\in\zz$, let $O_k:=\{x\in\cx:\ \ca(f)(x)>2^k\}$,
where $\ca(f)$ is as in \eqref{2.9} below.
If $O_{k}=\cx$, then $\widehat{O}_k=\cx\times(0,\fz)$,
where the tent $\widehat{O}_k$ of $O_k$
is defined by setting
$$\widehat{O}_k:=\lf\{(x,t)\in\cx\times(0,\fz):\ d(x,O_k^\complement)\ge t\r\},$$
$d(x,O_k^\complement):=\inf\{d(x,y):\ y\in O_k^\complement\}$
and $O_k^\complement:=\cx\backslash O_k$. Using this fact and repeating the proof
of \cite[Theorem 5.4]{bckyy13b}, we obtain the equivalence of $H_{\fai,\,L}(\cx)$ and
$H^{M,\,q}_{\fai,\,L,\,\mathrm{at}}(\cx)$ in the case that $\mu(\cx)<\fz$.
Moreover, via borrowing some ideas from the proof of \cite[Theorem 3.1]{sy18} and subtly
applying the properties of Musielak-Orlicz functions, we prove the equivalence of
$H^{\phi,\,\az}_{\fai,\,L,\,\mathrm{max}}(\cx)$ and $H^{\phi}_{\fai,\,L,\,\mathrm{rad}}(\cx)$.
Furthermore, by borrowing some ideas from the proof of \cite[Theorem 1.4]{bdl1},
we obtain the inclusion relation $H^{\phi,\,\az}_{\fai,\,L,\,\mathrm{max}}(\cx)\subset
H^{M,\,q}_{\fai,\,L,\,\mathrm{at}}(\cx)$. The proof of
the inclusion $H^{M,\,q}_{\fai,\,L,\,\mathrm{at}}(\cx)\subset
H^{\phi}_{\fai,\,L,\,\mathrm{rad}}(\cx)$ is standard. It is worth pointing out that
new ingredients appeared in the proof of Theorem \ref{t1.1}
are both the introduction of \emph{$(T_\fai,\,p)$-single-atoms} for tent spaces and the observation on
the existence of $(T_\fai,\,p)$-single-atoms in the decomposition of functions in
the tent space $T_\fai(\cx\times(0,\fz))$ when $\mu(\cx)<\fz$ (see Lemma \ref{l2.4} below
for the details).

\begin{remark}\label{r1.1}
Based on the atomic decomposition of the tent space on $\cx$ given in
Lemma \ref{l2.4} below, we know that the atomic decomposition theorem
of the tent space on $\cx$  in \cite{bckyy13b,hlmmy,jy11,yys4} holds true only when $\mu(\cx)=\fz$,
which further implies that the atomic characterization of the
(Musielak-)(Orlicz-)Hardy space, associated with the non-negative self-adjoint operator $L$ on $L^2(\cx)$
satisfying Davies-Gaffney estimates (or reinforced off-diagonal estimates),
obtained in \cite{bckyy13b,hlmmy,jy11,yys4}, holds true only under the assumption that $\mu(\cx)=\fz$.
Moreover, via replacing the atomic decomposition
of the tent space on $\cx$ used in \cite{bckyy13b,hlmmy,jy11,yys4} by Lemma \ref{l2.4} below
and then repeating the proof of the atomic characterization of the
(Musielak-)(Orlicz-)Hardy space in \cite{bckyy13b,hlmmy,jy11,yys4}, we obtain the
atomic characterization of the (Musielak-)(Orlicz-)Hardy space, associated with the operator $L$
satisfying Assumptions \ref{a1} and \ref{a2}, in Theorem \ref{t1.1}
without the \emph{assumption that $\mu(\cx)=\fz$}.
\end{remark}

When $\mu(\cx)<\fz$, by \eqref{1.2}, we conclude that $\diam(\cx)<\fz$,
here and hereafter, $\diam(\cx):=\sup\{d(x,y):\ x,\,y\in\cx\}$. Indeed, when $\mu(\cx)<\fz$,
$\cx$ is a bounded ball (see, for example, \cite[Lemma 5.1]{ny97}).

\begin{definition}\label{d1.6}
Let $\cx$, $L$ and $\fai$ be as in Definition \ref{d1.3}.
Assume that $\mu(\cx)<\fz$ and $\phi\in\cs(\rr)$ is an even function with $\phi(0)=1$ and
$\az\in(0,\fz)$.
\begin{itemize}
  \item[{\rm(i)}] For any $f\in L^2(\cx)$ and $x\in\cx$, the \emph{local non-tangential maximal function}
$\phi^{\ast}_{L,\,\az,\,\loc}(f)$ is defined by setting
\begin{align*}
\phi^{\ast}_{L,\,\az,\,\loc}(f)(x):=\sup_{d(x,y)<\az t,\,t\in(0,\diam(\cx))}\lf|\phi(t\sqrt{L})(f)(y)\r|.
\end{align*}

Then the \emph{local Musielak-Orlicz-Hardy space} $h^{\phi,\,\az}_{\fai,\,L,\,\mathrm{max}}(\cx)$ is defined
via replacing $\phi^{\ast}_{L,\,\az}(f)$ by $\phi^{\ast}_{L,\,\az,\,\loc}(f)$ in Definition \ref{d1.5}.

Specially, if $\phi(x):=e^{-x^2}$ for any $x\in\rr$ and $\az:=1$, denote $\phi^{\ast}_{L,\,\az,\,\loc}(f)$
simply by $f^\ast_{L,\,\loc}$ and, in this case, denote the space
$h^{\phi,\,\az}_{\fai,\,L,\,\mathrm{max}}(\cx)$ simply by $h_{\fai,\,L,\,\mathrm{max}}(\cx)$.
  \item[{\rm(ii)}] For any $f\in L^2(\cx)$ and $x\in\cx$, the \emph{local radial maximal function}
$\phi^+_{L,\,\loc}(f)$ is defined by setting
\begin{align*}
\phi^{+}_{L,\,\loc}(f)(x):=\sup_{t\in(0,\diam(\cx))}\lf|\phi(t\sqrt{L})(f)(x)\r|.
\end{align*}
Then the \emph{local Musielak-Orlicz-Hardy space} $h^{\phi}_{\fai,\,L,\,\mathrm{rad}}(\cx)$
is defined via replacing $\phi^{\ast}_{L,\,\az,\,\loc}(f)$ by $\phi^+_{L,\,\loc}(f)$ in the definition of
the space $h^{\phi,\,\az}_{\fai,\,L,\,\mathrm{max}}(\cx)$.
If $\phi(x):=e^{-x^2}$ for any $x\in\rr$, denote $\phi^+_{L,\,\loc}(f)$
simply by $f^+_{L,\,\loc}$ and, in this case, denote the space
$h^{\phi}_{\fai,\,L,\,\mathrm{rad}}(\cx)$ simply by $h_{\fai,\,L,\,\mathrm{rad}}(\cx)$.
\end{itemize}
\end{definition}

In the case that $\mu(\cx)<\fz$, we have the following further conclusion.

\begin{theorem}\label{t1.2}
Assume that $\cx$ is a space of homogeneous type with $\mu(\cx)<\fz$.
Let $L$ be an operator on $L^2(\cx)$ satisfying Assumptions \ref{a1} and \ref{a2},
and $\fai$ as in Definition \ref{d1.2}. Assume that $r(\fai)$, $I(\fai)$, $q(\fai)$
and $i(\fai)$ are, respectively, as in \eqref{1.7}, \eqref{1.4}, \eqref{1.6} and \eqref{1.5}.
For any $q\in([r(\fai)]'I(\fai),\fz]\cap(1,\fz]$,
$M\in\nn\cap(\frac{nq(\fai)}{2i(\fai)},\fz)$ and $\az\in(0,\fz)$,
the Musielak-Orlicz-Hardy spaces
$$H^{M,\,q}_{\fai,\,L,\,\mathrm{at}}(\cx),\ H^{\phi,\,\az}_{\fai,\,L,\,\mathrm{max}}(\cx), \ H^{\phi}_{\fai,\,L,\,\mathrm{rad}}(\cx)\ \mathrm{and}\ H_{\fai,\,L}(\cx)$$
and the local Musielak-Orlicz-Hardy spaces
$h^{\phi,\,\az}_{\fai,\,L,\,\mathrm{max}}(\cx)$ and
$h^{\phi}_{\fai,\,L,\,\mathrm{rad}}(\cx)$ coincide with equivalent quasi-norms.
\end{theorem}

Similarly to the proof of Theorem \ref{t2.1} below, we find that, for any $\az\in(0,\fz)$,
the local Musielak-Orlicz-Hardy spaces
$h^{\phi,\,\az}_{\fai,\,L,\,\mathrm{max}}(\cx)$ and
$h^{\phi}_{\fai,\,L,\,\mathrm{rad}}(\cx)$ coincide with equivalent quasi-norms (see \cite[Proposition 4.4]{bdl}
for the case that $\fai(x,t):=t^p$, with $p\in(0,1]$, for any $x\in\cx$ and $t\in[0,\fz)$).
Moreover, replacing the maximal function $\phi^{\ast}_{L,\,\az}(f)$ by $\phi^{\ast}_{L,\,\az,\,\loc}(f)$
and repeating the proof of Theorem \ref{t1.1}, we know that, for any $q\in([r(\fai)]'I(\fai),\fz]\cap(1,\fz]$,
$M\in\nn\cap(\frac{nq(\fai)}{2i(\fai)},\fz)$ and $\az\in(0,\fz)$,
the spaces $H^{M,\,q}_{\fai,\,L,\,\mathrm{at}}(\cx)$ and $h^{\phi,\,\az}_{\fai,\,L,\,\mathrm{max}}(\cx)$
coincide with equivalent quasi-norms. Using this and Theorem \ref{t1.1}, we then prove Theorem \ref{t1.2},
the details being omitted.

When $\fai(x,t):=t^p$, with $p\in(0,1]$, for any $x\in\cx$ and $t\in[0,\fz)$,
we denote the spaces $H^{M,\,q}_{\fai,\,L,\,\mathrm{at}}(\cx)$,
$H^{\phi,\,\az}_{\fai,\,L,\,\mathrm{max}}(\cx)$, $H^{\phi}_{\fai,\,L,\,\mathrm{rad}}(\cx)$,
$H_{\fai,\,L}(\cx)$, $h^{\phi,\,\az}_{\fai,\,L,\,\mathrm{max}}(\cx)$ and
$h^{\phi}_{\fai,\,L,\,\mathrm{rad}}(\cx)$, respectively, simply by $H^{p,\,M,\,q}_{L,\,\mathrm{at}}(\cx)$,
$H^{p,\,\phi,\,\az}_{L,\,\mathrm{max}}(\cx)$, $H^{p,\,\phi}_{L,\,\mathrm{rad}}(\cx)$,
$H^p_{L}(\cx)$, $h^{p,\,\phi,\,\az}_{L,\,\mathrm{max}}(\cx)$ and
$h^{p,\,\phi}_{L,\,\mathrm{rad}}(\cx)$.
Then, as the corollaries of Theorems \ref{t1.1} and \ref{t1.2}, we have the following conclusions.

\begin{corollary}\label{c1.1}
Let $\cx$ be a space of homogeneous type, $p\in(0,1]$ and
$L$ an operator on $L^2(\cx)$ satisfying Assumptions \ref{a1} and \ref{a2}.
For any $q\in(1,\fz]$, $M\in\nn\cap(\frac{n}{2p},\fz)$ and $\az\in(0,\fz)$,
the Hardy spaces $H^{p,\,M,\,q}_{L,\,\mathrm{at}}(\cx)$,
$H^{p,\,\phi,\,\az}_{L,\,\mathrm{max}}(\cx)$, $H^{p,\,\phi}_{L,\,\mathrm{rad}}(\cx)$
and $H^p_{L}(\cx)$ coincide with equivalent quasi-norms.

Assume further that $\mu(\cx)<\fz$. Then, for any $q\in(1,\fz]$, $M\in\nn\cap(\frac{n}{2p},\fz)$
and $\az\in(0,\fz)$, the Hardy spaces $H^{p,\,M,\,q}_{L,\,\mathrm{at}}(\cx)$,
$H^{p,\,\phi,\,\az}_{L,\,\mathrm{max}}(\cx)$, $H^{p,\,\phi}_{L,\,\mathrm{rad}}(\cx)$
and $H^p_{L}(\cx)$ and the local Hardy spaces $h^{p,\,\phi,\,\az}_{L,\,\mathrm{max}}(\cx)$ and
$h^{p,\,\phi}_{L,\,\mathrm{rad}}(\cx)$ coincide with equivalent quasi-norms.
\end{corollary}

\begin{remark}\label{r1.2}
The equivalences of the spaces
$H^{p,\,M,\,q}_{L,\,\mathrm{at}}(\cx)$, $H^{p,\,\phi,\,\az}_{L,\,\mathrm{max}}(\cx)$ and
$H^{p,\,\phi}_{L,\,\mathrm{rad}}(\cx)$ and
the equivalences of the spaces $H^{p,\,M,\,q}_{L,\,\mathrm{at}}(\cx)$,
$h^{p,\,\phi,\,\az}_{L,\,\mathrm{max}}(\cx)$ and
$h^{p,\,\phi}_{L,\,\mathrm{rad}}(\cx)$ in Corollary \ref{c1.1}
were obtained in \cite[Theorem 1.3]{sy18} (which requires $\mu(\cx)=\fz$) and
\cite[Theorem 1.4 and Remark 3.2]{bdl1}, respectively.
In particular, if $\cx:=\rn$, the equivalences of $H^p_{L}(\rn)$, $H^{p,\,M,\,q}_{L,\,\mathrm{at}}(\rn)$ and $H^{p,\,\phi,\,\az}_{L,\,\mathrm{max}}(\rn)$ were established in \cite[Theorem 1.4]{sy16}.
When $\mu(\cx)<\fz$, the equivalence of the spaces $H^p_{L}(\cx)$
and $H^{p,\,M,\,q}_{L,\,\mathrm{at}}(\cx)$ in Corollary \ref{c1.1}
is new in this case. Thus, when $\mu(\cx)<\fz$,
Corollary \ref{c1.1} improves the known results.
\end{remark}

\begin{remark}\label{r1.3}
\begin{itemize}
  \item[(i)] Let $\fai$ be as in Definition \ref{d1.2} or \eqref{1.8} with $\omega\equiv1$.
Under the \emph{assumption that} $\mu(\cx)=\fz$, the equivalence of the spaces $H_{\fai,\,L}(\cx)$
and $H^{M,\,q}_{\fai,\,L,\,\mathrm{at}}(\cx)$ was obtained in \cite{bckyy13b,jy11,yys4}.
Thus, Theorem \ref{t1.1} improves the corresponding results in \cite{bckyy13b,jy11,yys4}
via removing the assumption that $\mu(\cx)=\fz$.

In particular, if $\cx:=\rn$, the equivalence of $H^{M,\,q}_{\fai,\,L,\,\mathrm{at}}(\rn)$ and
$H^{\phi,\,\az}_{\fai,\,L,\,\mathrm{max}}(\rn)$ was obtained in \cite[Theorem 1.8]{yys16}.
Under the additional assumption that the heat semigroup of $L$ satisfies the H\"older continuity
(see \cite[Assumption 1.11]{yys16}), the equivalence of $H^{\phi,\,\az}_{\fai,\,L,\,\mathrm{max}}(\rn)$ and $H^{\phi}_{\fai,\,L,\,\mathrm{rad}}(\rn)$ was also obtained in \cite[Theorem 1.12]{yys16}.
Thus, Theorem \ref{t1.1} improves \cite[Theorem 1.12]{yys16} even when $\cx:=\rn$ by removing this
additional assumption for $L$.

\item[(ii)] Theorem \ref{t1.2} is new even when $\fai$ is as in \eqref{1.8} or \eqref{1.9}.
\end{itemize}
\end{remark}

\begin{remark}\label{r1.4}
\begin{itemize}
\item[(i)] Recall that a measurable function $p(\cdot):\ \rn\to(0,\fz)$ is called a
\emph{variable exponent}. For any variable exponent $p(\cdot)$, let
$$p^-:= \einf_{x\in\rn}p(x)\ \ \text{and}\ \ p^+:= \esup_{x\in\rn}p(x).$$
Denote by $\cp(\rn)$ the collection of all variable exponents $p(\cdot)$ satisfying $0<p^-\le p^+<\fz$.
It is known that, if
\begin{equation}\label{1.12}
\fai(x,t):=t^{p(x)}\ \ \text{for any}\  x\in\rn\ \text{and}\  t\in[0,\fz),
\end{equation}
where $p(\cdot)\in\cp(\rn)$, then the Musielak-Orlicz space $L^\fai(\rn)$ is just the variable exponent
space $L^{p(\cdot)}(\rn)$.

Moreover, the variable exponent Hardy space $H^{p(\cdot)}_L(\rn)$ associated with $L$ is defined
in the same way as $H_{\fai,\,L}(\rn)$ with $\|\cdot\|_{L^\fai(\rn)}$ replaced by $\|\cdot\|_{L^{p(\cdot)}(\rn)}$.
Under the assumptions that $0<p^-\le p^+\le1$ and
$p(\cdot)$ satisfies the so-called \emph{globally log-H\"older continuous condition},
which is denoted by $p(\cdot)\in C^{\log}(\rn)$ (see, for example, \cite{yyz14,yzz18,yz16,zy16} for the details),
the real-variable characterizations of $H^{p(\cdot)}_L(\rn)$ or its local version $h^{p(\cdot)}_{L}(\rn)$,
including atoms, Lusin area functions, non-tangential and radial maximal functions associated with $L$,
were established in \cite{abdr17,yzz18,yz16,zy16}. It is worth pointing out that a general Musielak-Orlicz
function $\fai$ as in Definition \ref{d1.2} may not have the form as in \eqref{1.12} (see, for example,
\cite{k14,ylk17,yys1}). On the other hand, it was proved in \cite[Remark 2.23(iii)]{yyz14}
that there exists a variable exponent function $p(\cdot)\in C^{\log}(\rn)$, but $t^{p(\cdot)}$ is not a
uniformly Muckenhoupt weight, which was required in Definition \ref{d1.2}.
Thus, variable exponent Hardy spaces associated with operators in \cite{abdr17,yzz18,yz16,zy16}
and Musielak-Orlicz-Hardy spaces associated with operators in this article do not
cover each other even when $\cx:=\rn$.

\item[(ii)] Let $\mathcal{L}$ be a degenerate Schr\"odinger operator on $\rn$,
which is defined by setting
$$\mathcal{L}:=-\frac{1}{\omega}\mathrm{div}(A\nabla)+V,
$$
where $\omega\in A_2(\rn)$, $A:=\{a_{ij}\}_{i,\,j=1}^n$ is a real symmetric matrix satisfying that
there exists a positive constant $C$ such that, for any $x,\,\xi\in\rn$,
$$C^{-1}\omega(x)|\xi|^2\le\sum_{i,\,j=1}^na_{ij}(x)\xi_i\xi_j\le C\omega(x)|\xi|^2
$$
with $|\xi|:=(\xi_1^2+\cdots+\xi_n^2)^{1/2}$, and $V\ge0$. Assume that $(\cx_0,d,\mu):=(\rn,|\cdot|,\omega dx)$,
where $|\cdot|$ denotes the usual Euclidean distance on $\rn$ and $dx$ the Lebesgue measure on $\rn$. Then it is easy to see that $\mathcal{L}$ is non-negative self-adjoint on $L^2(\cx_0)$ and hence $\mathcal{L}$ satisfies Assumption \ref{a1}. Moreover,
it is well known that $\mathcal{L}$ satisfies Assumption \ref{a2} (see, for example, \cite[Theorem 2.2]{d05}).

Under some additional assumptions for $\omega$ and $V$ (see \cite{d05,hll18}
for the details), the real-variable characterizations of the Hardy space $H^1_{\mathcal{L}}(\cx_0)$,
including the radial maximal function and the Lusin area function associated with $\mathcal{L}$,
were established in \cite{d05,hll18}.
Thus, even when $\cx:=\cx_0$ and $L:=\mathcal{L}$, Corollary \ref{c1.1}
improves the main results obtained in \cite{d05,hll18} by extending the range
$p=1$ of the exponent $p$ into $p\in(0,1]$.
\end{itemize}
\end{remark}

To obtain more equivalent characterizations of $H_{\fai,\,L}(\cx)$, we further introduce
the following assumptions for $L$.

\begin{assumption}\label{a3}
The kernels of the semigroup $\{e^{-tL}\}_{t>0}$, denoted by $\{K_t\}_{t>0}$,
satisfy the H\"older continuous condition, namely,
there exist positive constants $C_3$, $c_3$ and $\dz_0\in(0,1]$ such that,
for any $t\in(0,[\diam(\cx)]^2)$ and $x,\,y,\,z\in\cx$ with $d(y,z)\le\sqrt{t}$,
$$|K_t(x,y)-K_t(x,z)|\le\frac{C_3}{V(x,\sqrt{t})}\lf[\frac{d(y,z)}{\sqrt{t}}\r]^{\dz_0}\exp\lf\{-\frac{[d(x,y)]^2}{c_3t}\r\}.
$$
\end{assumption}

\begin{assumption}\label{a4}
The semigroup $\{e^{-tL}\}_{t>0}$ is conservative, namely, for any $t\in(0,\fz)$ and $x\in\cx$,
\begin{equation*}
\int_{\cx}K_t(x,y)\,d\mu(y)=1.
\end{equation*}
\end{assumption}

Now we introduce the following Musielak-Orlicz-Hardy spaces on $\cx$ independent of the operator $L$.

\begin{definition}\label{d1.7}
Let $\cx$ and $\fai$ be as in Definition \ref{d1.3}.
Assume that $q\in(1,\fz]$, $M\in\nn$ and $B\subset\cx$ is a ball.
A function $\az\in L^q(\cx)$ is called a $(\fai,\,q,\,0)$-\emph{atom}
associated with the ball $B$ if
\begin{itemize}
  \item[(i)] $\supp(\az)\subset B$;
  \item[(ii)] $\|\az\|_{L^q(\cx)}\le [\mu(B)]^{1/q}\|\mathbf{1}_B\|^{-1}_{L^\fai(\cx)}$;
  \item[(iii)] $\int_B\az(x)\,d\mu(x)=0$.
\end{itemize}
If $\mu(\cx)<\fz$, a function $\az$ on $\cx$ is called a \emph{$(\fai,\,q)$-single-atom} if
$$\|\az\|_{L^q(\cx)}\le[\mu(\cx)]^{1/q}\|\mathbf{1}_\cx\|^{-1}_{L^\fai(\cx)}.$$
Then the \emph{atomic Musielak-Orlicz-Hardy type space} $H^q_{\fai,\,\mathrm{at}}(\cx)$
is defined via replacing $(\fai,\,q,\,M)_L$-atoms by $(\fai,\,q,\,0)$-atoms in the definition of
the space $H^{M,\,q}_{\fai,\,L,\,\mathrm{at}}(\cx)$.
\end{definition}

We point out that, when $\fai(x,t):=t^p$, with $p\in(\frac{n}{n+1},1]$,
for any $x\in\cx$ and $t\in[0,\fz)$, the space $H^q_{\fai,\,\mathrm{at}}(\cx)$
coincides with the Hardy space $H^p_{CW}(\cx)$ introduced by Coifman and Weiss \cite{cw77}.

Assume further that the operator satisfies Assumptions \ref{a3} and \ref{a4}.
Using Theorem \ref{t1.1} and Assumptions \ref{a3} and \ref{a4}, we have the following conclusion.

\begin{theorem}\label{t1.3}
Assume that $\cx$  is a space of homogeneous type.
Let $L$ be an operator on $L^2(\cx)$ satisfying Assumptions \ref{a1}, \ref{a2}, \ref{a3} and \ref{a4},
and $\fai$ as in Definition \ref{d1.2}. Assume that $\dz_0\in(0,1]$, $r(\fai)$, $I(\fai)$, $q(\fai)$
and $i(\fai)$ are, respectively, as in Assumption \ref{a3}, \eqref{1.7}, \eqref{1.4}, \eqref{1.6} and \eqref{1.5},
and $[r(\fai)]'$ denotes the conjugate exponent of $r(\fai)$. If $nq(\fai)/i(\fai)<n+\dz_0$,
then, for any $q\in([r(\fai)]'I(\fai),\fz]\cap(1,\fz]$,
$M\in\nn\cap(\frac{nq(\fai)}{2i(\fai)},\fz)$ and $\az\in(0,\fz)$,
the spaces $H_{\fai,\,L}(\cx)$, $H^{M,\,q}_{\fai,\,L,\,\mathrm{at}}(\cx)$,
$H^{\phi,\,\az}_{\fai,\,L,\,\mathrm{max}}(\cx)$, $H^{\phi}_{\fai,\,L,\,\mathrm{rad}}(\cx)$ and
$H^{q}_{\fai,\,\mathrm{at}}(\cx)$ coincide with equivalent quasi-norms. In particular, if
$\mu(\cx)<\fz$, the spaces $H_{\fai,\,L}(\cx)$, $H^{M,\,q}_{\fai,\,L,\,\mathrm{at}}(\cx)$,
$H^{\phi,\,\az}_{\fai,\,L,\,\mathrm{max}}(\cx)$, $H^{\phi}_{\fai,\,L,\,\mathrm{rad}}(\cx)$,
$H^{q}_{\fai,\,\mathrm{at}}(\cx)$, $h^{\phi,\,\az}_{\fai,\,L,\,\mathrm{max}}(\cx)$ and
$h^{\phi}_{\fai,\,L,\,\mathrm{rad}}(\cx)$ coincide with equivalent quasi-norms.
\end{theorem}

When $\fai(x,t):=t^p$, with $p\in(0,1]$, for any $x\in\cx$ and $t\in[0,\fz)$,
we denote the space $H^{q}_{\fai,\,\mathrm{at}}(\cx)$ simply by $H^{p,\,q}_{\mathrm{at}}(\cx)$.
As the corollary of Theorem \ref{t1.3}, we have the following conclusion.

\begin{corollary}\label{c1.2}
Let $\cx$  be a space of homogeneous type and
$L$ an operator on $L^2(\cx)$ satisfying Assumptions \ref{a1}, \ref{a2}, \ref{a3} and \ref{a4}.
Assume that $\dz_0\in(0,1]$ is as in Assumption \ref{a3} and $p\in(\frac{n}{n+\dz_0},1]$.
Then, for any $q\in(1,\fz]$, $M\in\nn\cap(\frac{n}{2p},\fz)$ and $\az\in(0,\fz)$,
the spaces $H^p_{L}(\cx)$, $H^{p,\,M,\,q}_{L,\,\mathrm{at}}(\cx)$,
$H^{p,\,\phi,\,\az}_{L,\,\mathrm{max}}(\cx)$, $H^{p,\,\phi}_{L,\,\mathrm{rad}}(\cx)$ and
$H^{p,\,q}_{\mathrm{at}}(\cx)$ coincide with equivalent quasi-norms. In particular, if
$\mu(\cx)<\fz$, the spaces $H^p_{L}(\cx)$, $H^{p,\,M,\,q}_{L,\,\mathrm{at}}(\cx)$,
$H^{p,\,\phi,\,\az}_{L,\,\mathrm{max}}(\cx)$, $H^{p,\,\phi}_{L,\,\mathrm{rad}}(\cx)$,
$H^{p,\,q}_{\mathrm{at}}(\cx)$, $h^{p,\,\phi,\,\az}_{L,\,\mathrm{max}}(\cx)$ and
$h^{p,\,\phi}_{L,\,\mathrm{rad}}(\cx)$ coincide with equivalent quasi-norms.
\end{corollary}

\begin{remark}\label{r1.5}
\begin{itemize}
\item[(i)] We point out that the equivalences of the Hardy spaces
$$H^{p,\,M,\,q}_{L,\,\mathrm{at}}(\cx),\ H^{p,\,\phi,\,\az}_{L,\,\mathrm{max}}(\cx),
\ H^{p,\,\phi}_{L,\,\mathrm{rad}}(\cx),\ H^{p,\,q}_{\mathrm{at}}(\cx),\
h^{p,\,\phi,\,\az}_{L,\,\mathrm{max}}(\cx)\ \mathrm{and}\ h^{p,\,\phi}_{L,\,\mathrm{rad}}(\cx)$$
in Corollary \ref{c1.2} were obtained in \cite[Theorems 1.4 and 1.8 and Remark 3.2]{bdl1}.

\item[(ii)] Theorem \ref{t1.3} is new even when $\fai$ is as in \eqref{1.8} or \eqref{1.9}.
\end{itemize}
\end{remark}

The layout of this article is as follows. Section \ref{s2}
is devoted to the proof of Theorem \ref{t1.1}. In Section \ref{s3},
we give the proof of Theorem \ref{t1.3}.

Moreover, in Section \ref{s4}, as applications of
Theorems \ref{t1.1}, \ref{t1.2} and \ref{t1.3},  we characterize  the
``geometric" Musielak-Orlicz-Hardy spaces $H_{\fai,\,r}(\boz)$ and $H_{\fai,\,z}(\boz)$
on strongly Lipschitz domains in $\rn$ by means of atoms, non-tangential and radial maximal
functions associated with second-order self-adjoint elliptic operators
with Dirichlet and Neumann boundary conditions.
We point out that the results obtained for the spaces $H_{\fai,\,r}(\boz)$
and $H_{\fai,\,z}(\boz)$  in Section \ref{s4} below improve
the corresponding results established in \cite{ar03,bdl1,yys2,yys3}
(see Remarks \ref{r4.1}, \ref{r4.2} and \ref{r4.3} below for the details) by
removing the additional assumption that $\boz$ is unbounded.
Some further applications of Theorems \ref{t1.1} and \ref{t1.2} to Hardy type
spaces on domains will be given in a forthcoming article.

Finally, we make some conventions on notation. Throughout the whole
article, we always denote by $C$ a \emph{positive constant} which is
independent of the main parameters, but it may vary from line to
line. We also use $C_{(\gz,\,\bz,\,\ldots)}$ to denote a  \emph{positive
constant} depending on the indicated parameters $\gz,$ $\bz$,
$\ldots$. The \emph{symbol} $A\ls B$ means that $A\le CB$. If $A\ls
B$ and $B\ls A$, then we write $A\sim B$. The  \emph{symbol}
$\lfz s\rfz$ for any $s\in\rr$ denotes the largest integer not greater
than $s$. For any ball $B:=B(x_B,r_B)\subset\cx$, with some $x_B\in\cx$ and
$r_B\in (0,\fz)$, and $\az\in (0,\fz)$, let $\az B:=B(x_B,\az r_B)$.
For any subset $E$ of $\cx$, we denote by $E^\complement$
the \emph{set} $\cx\setminus E$ and by $\mathbf{1}_{E}$ its \emph{characteristic function}. We also let $\nn:=\{1,\, 2,\, \ldots\}$ and
$\zz_+:=\nn\cup\{0\}$. For any ball $B$ in $\cx$ and $j\in\zz_+$,
let $S_j(B):=(2^{j+1}B)\setminus(2^{j}B)$ with $j\in\nn$ and $S_0(B):=2B$.
Finally, for $q\in[1,\fz]$, we denote by $q'$ its \emph{conjugate exponent}, namely, $1/q + 1/q'= 1$.

\section{Proof of Theorem \ref{t1.1}\label{s2}}

In this section, we show Theorem \ref{t1.1}.
We begin with some auxiliary conclusions.

Let $\cx$  be a space of homogeneous type.
For a non-negative self-adjoint operator $L$ on $L^2(\cx)$,
denote by $E_L$ the spectral measure associated with $L$.
Then, for any bounded Borel function $F:\ [0,\fz)\rightarrow\cc$,
the operator $F(L):\ L^2(\cx)\rightarrow L^2(\cx)$ is defined by the formula
\begin{align*}
F(L):=\int_0^\fz F(\lz)\,dE_L(\lz).
\end{align*}

Then we have the following lemma, which was just \cite[Lemma 3.5]{hlmmy}.
In what follows, $C_c^\fz(\rr)$ denotes the set of all infinitely differentiable
functions on $\rr$ with compact supports and,
for any $\phi\in C_c^\fz(\rr)$, we use $\Phi$ to denote its Fourier transform,
namely, for any $\xi\in\rr$,
$$\Phi(\xi):=\widehat{\phi}(\xi):=\int_{\rr}\phi(x)e^{-2\pi x\xi}\,dx.$$

\begin{lemma}\label{l2.1}
Let $\cx$  be a space of homogeneous type.
Assume that the operator $L$ satisfies Assumptions \ref{a1} and \ref{a2}.
Let $\phi\in C^\fz_c(\rr)$ be even and $\supp(\phi)\subset(-1,1)$. Denote by
$\Phi$ the Fourier transform of $\phi$. Then, for any $k\in\zz_+$,
the kernels $\{K_{(t^2L)^k\Phi(t\sqrt{L})}\}_{t>0}$ of the operators $\{(t^2L)^k\Phi(t\sqrt{L})\}_{t>0}$ satisfy
that, for any $t\in(0,\fz)$ and $x,\,y\in\cx$,
\begin{align*}
\supp\lf(K_{(t^2L)^k\Phi(t\sqrt{L})}\r)\subset\{(x,y)\in\cx\times\cx:\ d(x,y)\le t\}.
\end{align*}
Moreover, for any $k\in\zz_+$, there exists a positive constant $C_{(k)}$,
depending on $k$, such that, for any $t\in(0,\fz)$ and $x,\,y\in\cx$,
\begin{align*}
\lf|K_{(t^2L)^k\Phi(t\sqrt{L})}(x,y)\r|\le\frac{C_{(k)}}{\mu(B(x,t))}.
\end{align*}
\end{lemma}

Let $L^1_{\loc}(\cx)$ denote the set of all locally integral functions on $\cx$
and $\cm$ the \emph{Hardy-Littlewood maximal operator} on $\cx$, namely,
for any $f\in L^1_{\loc}(\cx)$ and $x\in\cx$,
$$\cm(f)(x):=\sup_{B\ni x}\frac{1}{V(B)}\int_B|f(y)|\,d\mu(y),$$
where the supremum is taken over all balls $B\ni x$.

Moreover, we have the following properties of growth functions,
which were obtained in \cite[Lemma 4.1]{k14} and \cite[Lemma 2.8]{yys4}.

\begin{lemma}\label{l2.2}
Let $\fai$ be as in Definition \ref{d1.2}.
\begin{itemize}
  \item[{\rm(i)}] There exists a positive constant $C$ such that, for any
$(x,t_j)\in\cx\times[0,\fz)$ with $j\in\nn$,
$\fai(x,\sum_{j=1}^{\fz}t_j)\le C\sum_{j=1}^{\fz}\fai(x,t_j)$.

  \item[{\rm(ii)}] Let $\wz{\fai}(x,t):=\int_0^t\frac{\fai(x,s)}{s}\,ds$ for any
$(x,t)\in\cx\times[0,\fz)$. Then $\wz{\fai}$ is equivalent to $\fai$, namely,
there exists a positive constant $C$ such that, for any $(x,t)\in\cx\times[0,\fz)$,
$C^{-1}\fai(x,t)\le\wz{\fai}(x,t)\le C\fai(x,t)$.
  \item[{\rm(iii)}] If $p\in(1,\fz)$ and $\fai\in \aa_{p}(\cx)$, then
there exists a positive constant $C$, depending on $p$ and $\aa_{p}(\vz)$, such that, for any measurable
functions $f$ on $\cx$ and $t\in[0,\fz)$,
$$\int_{\cx}\lf[\cm(f)(x)\r]^p\fai(x,t)\,d\mu(x)\le
C\int_{\cx}|f(x)|^p\fai(x,t)\,d\mu(x).$$
  \item[{\rm(iv)}] If $\fai\in \aa_{p}(\cx)$ with $p\in[1,\fz)$, then
there exists a positive constant $C$, depending on $p$ and $\aa_{p}(\vz)$, such that, for any $t\in(0,\fz)$ and balls $B_1,\,B_2\subset\cx$ with $B_1\subset B_2$,
$\frac{\fai(B_2,t)}{\fai(B_1,t)}\le
C[\frac{V(B_2)}{V(B_1)}]^p$.
\end{itemize}
\end{lemma}

Let $\cx$  be a space of homogeneous type and $F$ a $\mu$-measurable function on $\cx\times(0,\fz)$.
For any $\az\in(0,\fz)$ and $x\in\cx$, let
\begin{equation}\label{2.1}
F^\ast_\az(x):=\sup_{d(x,y)<\az t,\,t\in(0,\fz)}|F(y,t)|.
\end{equation}

Then we have the following conclusion.

\begin{lemma}\label{l2.3}
Let $F$ be a $\mu$-measurable function on $\cx\times(0,\fz)$ and
$\fai$ as in Definition \ref{d1.2}. Assume further that $\fai\in\aa_p(\cx)$ with $p\in(1,\fz)$
and $\fai$ is of uniformly lower type $\wz{p}$ with $\wz{p}\in(0,1]$.
Then there exists a positive  constant $C$, depending on $n$, $\fai$ and $p$, such that,
for any $\az_1,\,\az_2\in(0,\fz)$ with $\az_2\le\az_1$,
\begin{equation}\label{2.2}
\lf\|F^\ast_{\az_1}\r\|_{L^\fai(\cx)}\le C\lf[\frac{\az_1}{\az_2}\r]^{\frac{np}{\wz{p}}}
\lf\|F^\ast_{\az_2}\r\|_{L^\fai(\cx)},
\end{equation}
where $F^\ast_{\az_1}$ and $F^\ast_{\az_2}$ are as in \eqref{2.1}, respectively, with
$\az:=\az_1$ and $\az:=\az_2$.
\end{lemma}

\begin{proof}
For any $\lz\in(0,\fz)$, let $E_\lz:=\{x\in\cx:\ F^\ast_{\az_2}(x)>\lz\}$.
Then it was proved in \cite[(2.8)]{sy18} that
there exists a positive constant $\wz{C}$ such that, for any $\lz\in(0,\fz)$,
\begin{equation}\label{2.3}
\lf\{x\in\cx:\ F^\ast_{\az_1}(x)>\lz\r\}\subset E_{\lz}^{\ast}:=\lf\{x\in\cx:\
\cm(\mathbf{1}_{E_\lz})(x)>\frac{\wz{C}}{(\az_1/\az_2)^n}\r\}.
\end{equation}
Moreover, from Lemma \ref{l2.2}(iii), it follows that, for any $t\in(0,\fz)$ and $\lz\in(0,\fz)$,
$$\int_{ E_{\lz}^{\ast}}\fai(x,t)\,d\mu(x)\ls\frac{[\az_1/\az_2]^{np}}{\wz{C}^p}
\int_{E_\lz}\fai(x,t)\,d\mu(x),
$$
which, together with Lemma \ref{l2.2}(ii), the Fubini theorem and \eqref{2.3}, implies that
\begin{align*}
\int_\cx\fai\lf(x,F^\ast_{\az_1}(x)\r)\,d\mu(x)&\sim\int_\cx\int_0^{F^\ast_{\az_1}(x)}
\frac{\fai(x,t)}{t}\,dt\,d\mu(x)\sim\int_0^\fz\int_{\{x\in\cx:\
F^\ast_{\az_1}(x)>t\}}\frac{\fai(x,t)}{t}\,d\mu(x)\,dt\\
&\ls\int_0^\fz
\int_{E^\ast_t}\frac{\fai(x,t)}{t}\,d\mu(x)\,dt
\ls\lf[\frac{\az_1}{\az_2}\r]^{np}\int_0^\fz
\int_{E_t}\frac{\fai(x,t)}{t}\,d\mu(x)\,dt\\
&\sim\lf[\frac{\az_1}{\az_2}\r]^{np}\int_\cx\fai\lf(x,F^\ast_{\az_2}(x)\r)\,d\mu(x).
\end{align*}
By this and the facts that $\fai$ is of uniformly lower type $\wz{p}$ and $\az_1\ge\az_2$,
we conclude that, for any $\lz\in(0,\fz)$,
\begin{align*}
\int_\cx\fai\lf(x,\frac{F^\ast_{\az_1}(x)}{[\az_1/\az_2]^{np/\wz{p}}\lz}\r)\,d\mu(x)
\ls\lf[\frac{\az_1}{\az_2}\r]^{-np}\int_\cx\fai\lf(x,\frac{F^\ast_{\az_1}(x)}{\lz}\r)\,d\mu(x)
\ls\int_\cx\fai\lf(x,\frac{F^\ast_{\az_2}(x)}{\lz}\r)\,d\mu(x),
\end{align*}
which further implies that
$$\lf\|F^\ast_{\az_1}\r\|_{L^\fai(\cx)}\ls\lf[\frac{\az_1}{\az_2}\r]^{\frac{np}{\wz{p}}}
\lf\|F^\ast_{\az_2}\r\|_{L^\fai(\cx)}.$$
This finishes the proof of Lemma \ref{l2.3}.
\end{proof}

Moreover, to show Theorem \ref{t1.1}, we need to establish the following conclusion.

\begin{proposition}\label{p2.1}
Assume that $\cx$  is a space of homogeneous type.
Let $L$ satisfy Assumptions \ref{a1} and \ref{a2}, and $\fai$ be as in Definition \ref{d1.2}.
Assume that $\psi_1,\,\psi_2\in\cs(\rr)$ are even functions with $\psi_1(0)=\psi_2(0)=1$,
and $\az_1,\,\az_2\in(0,\fz)$. For any $i\in\{0,\,1\}$,
$f\in L^2(\cx)$ and $x\in\cx$, let
$$(\psi_i)^{\ast}_{L,\,\az_i}(f)(x):=\sup_{d(x,y)<\az_i t,\,t\in(0,\fz)}\lf|\psi_i(t\sqrt{L})(f)(y)\r|.$$
Then there exists a positive constant $C$, depending on $n$, $\fai$, $\psi_1$, $\psi_2$,
$\az_1$ and $\az_2$, such that, for any $f\in L^2(\cx)$,
\begin{align}\label{2.4}
\lf\|(\psi_1)^{\ast}_{L,\,\az_1}(f)\r\|_{L^\fai(\cx)}\le C\lf\|(\psi_2)^{\ast}_{L,\,\az_2}(f)\r\|_{L^\fai(\cx)}.
\end{align}
Specially, for any given even function $\phi\in\cs(\rr)$
with $\phi(0)=1$ and $\az\in(0,\fz)$,
there exists a positive constant $C$, depending on $n$, $\fai$, $\phi$ and $\az$, such that, for any $f\in L^2(\cx)$,
\begin{align*}
C^{-1}\|f^\ast_L\|_{L^\fai(\cx)}\le\lf\|\phi^\ast_{L,\,\az}(f)\r\|_{L^\fai(\cx)}\le C\|f^\ast_L\|_{L^\fai(\cx)}.
\end{align*}
\end{proposition}

\begin{proof}
By the fact that $\fai$ is as in Definition \ref{d1.2},
we know that there exist $p\in(1,\fz)$ and $\wz{p}\in(0,1]$ such that $\fai\in\aa_p(\cx)$
and $\fai$ is of uniformly lower type $\wz{p}\in(0,1]$.
Let $\psi:=\psi_1-\psi_2$. Via \eqref{2.2}, to prove \eqref{2.4}, it suffices to show that
\begin{align}\label{2.5}
\lf\|\psi^{\ast}_{L,\,1}(f)\r\|_{L^\fai(\cx)}\ls\lf\|(\psi_2)^{\ast}_{L,\,1}(f)\r\|_{L^\fai(\cx)},
\end{align}
where the implicit positive constant depends on $n$, $\psi_1$, $\psi_2$ and $\fai$.
Now we show \eqref{2.5}.
Let $\Psi(x):=x^{2k}\Phi(x)$ for any $x\in\rr$, where  $k\in\nn$ with $k>nq(\fai)/[2i(\fai)]$
and $\Phi$ is as in Lemma \ref{l2.1}.
From the spectral calculus, we deduce that there exists a constant $C_{(\Psi,\,\psi_2)}$,
depending on $\Psi$ and $\psi_2$, such that
$$f=C_{(\Psi,\,\psi_2)}\int_0^\fz\Psi(s\sqrt{L})\psi_2(s\sqrt{L})(f)\,\frac{ds}{s},
$$
which further implies that, for any $t\in(0,\fz)$,
\begin{align*}
\psi(t\sqrt{L})(f)=C_{(\Psi,\,\psi_2)}\int_0^\fz\lf[\psi(t\sqrt{L})
\Psi(s\sqrt{L})\r]\psi_2(s\sqrt{L})(f)\,\frac{ds}{s}.
\end{align*}
Let $K_{\psi(t\sqrt{L})\Psi(s\sqrt{L})}$ be the kernel of $\psi(t\sqrt{L})\Psi(s\sqrt{L})$.
Then, for any $\lz\in(0,\fz)$ and $x\in\cx$, we have
\begin{align}\label{2.6}
&\sup_{d(x,y)<t,\,t\in(0,\fz)}\lf|\psi(t\sqrt{L})(f)(y)\r|\\ \nonumber
&\hs\sim\sup_{d(x,y)<t,\,t\in(0,\fz)}\lf|\int_0^\fz
K_{\psi(t\sqrt{L})\Psi(s\sqrt{L})}(y,z)\psi_2(s\sqrt{L})(f)(z)\frac{d\mu(z)\,ds}{s}\r|\\ \nonumber
&\hs\ls\sup_{d(x,y)<t,\,t\in(0,\fz)}\int_{\cx\times(0,\fz)}
\lf|K_{\psi(t\sqrt{L})\Psi(s\sqrt{L})}(y,z)\r|\lf[1+\frac{d(x,z)}{s}\r]^{\lz}\\ \nonumber
&\hs\hs\times\lf|\psi_2(s\sqrt{L})(f)(z)\r|\lf[1+\frac{d(x,z)}{s}\r]^{-\lz}\frac{d\mu(z)\,ds}{s}\\ \nonumber
&\hs\ls\sup_{z\in\cx,\,s\in(0,\fz)}\lf|\psi_2(s\sqrt{L})(f)(z)\r|\lf[1+\frac{d(x,z)}{s}\r]^{-\lz}\\ \nonumber
&\hs\hs\times\sup_{d(x,y)<t,\,t\in(0,\fz)}\int_{\cx\times(0,\fz)}\lf|K_{\psi(t\sqrt{L})\Psi(s\sqrt{L})}(y,z)\r|
\lf[1+\frac{d(x,z)}{s}\r]^{\lz}\frac{d\mu(z)\,ds}{s}.
\end{align}
Moreover, from the proof of \cite[Proposition 2.4]{sy18}, it follows that, for any $\lz\in(0,2k)$ and $x\in\cx$,
$$\sup_{d(x,y)<t,\,t\in(0,\fz)}\int_{\cx\times(0,\fz)}\lf|K_{\psi(t\sqrt{L})
\Psi(s\sqrt{L})}(y,z)\r|\lf[1+\frac{d(x,z)}{s}\r]^{\lz}\frac{d\mu(z)\,ds}{s}\ls1,
$$
where the implicit positive constant depends on $n$, $\Psi$, $\psi$ and $\lz$,
which, together with \eqref{2.6}, implies that
\begin{align}\label{2.7}
\sup_{d(x,y)<t,\,t\in(0,\fz)}\lf|\psi(t\sqrt{L})(f)(y)\r|\ls\sup_{z\in\cx,\,s\in(0,\fz)}\lf|\psi_2(s\sqrt{L})
(f)(z)\r|\lf[1+\frac{d(x,z)}{s}\r]^{-\lz},
\end{align}
where the implicit positive constant depends on $n$, $\Psi$, $\psi$ and $\lz$.
Furthermore, let $\mathbf{1}_{[0,1]}$ be the characteristic function of $[0,1]$. Then, for any $\lz,\,s\in(0,\fz)$,
we have
\begin{align*}
(1+s)^{-\lz}\le\sum_{k=1}^\fz2^{-k}\mathbf{1}_{[0,1]}\lf(\frac{s}{2^{k/\lz}}\r),
\end{align*}
which further implies that
\begin{align}\label{2.8}
&\sup_{z\in\cx,\,s\in(0,\fz)}\lf|\psi_2(s\sqrt{L})(f)(z)\r|\lf[1+\frac{d(x,z)}{s}\r]^{-\lz}\\ \nonumber
&\hs\le\sum_{k=1}^\fz2^{-k}\sup_{z\in\cx,\,s\in(0,\fz)}\lf|\psi_2(s\sqrt{L})(f)(z)\r|
\mathbf{1}_{[0,1]}\lf(\frac{d(x,z)}{s2^{k/\lz}}\r)\\ \nonumber
&\hs=\sum_{k=1}^\fz2^{-k}\sup_{d(x,z)<2^{k/\lz}s,\,s\in(0,\fz)}\lf|\psi_2(s\sqrt{L})(f)(z)\r|
=\sum_{k=1}^\fz2^{-k}(\psi_2)^\ast_{L,\,2^{k/\lz}}(f)(x).
\end{align}
Let $\lz\in(nq(\fai)/i(\fai),2k)$. Then, by $\lz>nq(\fai)/i(\fai)$ and the definitions of
$q(\fai)$ and $i(\fai)$, we find that there exist $p_0\in(0,i(\fai))$
and $\wz{q}\in(q(\fai),\fz)$ such that $\lz>n\wz{q}/p_0$, $\fai$ is of
uniformly lower type $p_0$ and $\fai\in\aa_{\wz{q}}(\cx)$,
which, combined with \eqref{2.8}, Lemma \ref{l2.2}(i) and \eqref{2.2},
implies that, for any $t\in(0,\fz)$,
\begin{align*}
&\int_\cx\fai\lf(x,\sup_{z\in\cx,\,s\in(0,\fz)}
\lf|\psi_2(s\sqrt{L})(f)(z)\r|\lf[1+\frac{d(x,z)}{s}\r]^{-\lz}/t\r)\,d\mu(x)\\
&\hs\le\int_\cx\fai\lf(x,\sum_{k=1}^\fz2^{-k}(\psi_2)^\ast_{L,\,2^{k/\lz}}(f)(x)/t\r)\,d\mu(x)
\ls\sum_{k=1}^\fz2^{-kp_0}\int_\cx\fai\lf(x,(\psi_2)^\ast_{L,\,2^{k/\lz}}(f)(x)/t\r)\,d\mu(x)\\
&\hs\ls\sum_{k=1}^\fz2^{-k[p_0-n\wz{q}/\lz]}\int_\cx\fai\lf(x,(\psi_2)^\ast_{L}(f)(x)/t\r)\,d\mu(x)
\sim\int_\cx\fai\lf(x,(\psi_2)^\ast_{L}(f)(x)/t\r)\,d\mu(x),
\end{align*}
where the implicit positive constants depend on $n$, $\psi$ and $\fai$.
From this, we deduce that
$$\lf\|\sup_{z\in\cx,\,s\in(0,\fz)}\lf|\psi_2(s\sqrt{L})(f)(z)\r|
\lf[1+\frac{d(\cdot,z)}{s}\r]^{-\lz}\r\|_{L^\fai(\cx)}
\ls\lf\|(\psi_2)^\ast_{L}(f)\r\|_{L^\fai(\cx)},$$
which, together with \eqref{2.7}, further implies that
\begin{align*}
\lf\|\psi^\ast_{L,\,1}(f)\r\|_{L^\fai(\cx)}&=\lf\|\sup_{d(\cdot,y)<t,\,t\in(0,\fz)}
\lf|\psi(t\sqrt{L})(f)(y)\r|\r\|_{L^\fai(\cx)}\\ \nonumber
&\ls\lf\|\sup_{z\in\cx,\,s\in(0,\fz)}\lf|\psi_2(s\sqrt{L})(f)(z)\r|
\lf[1+\frac{d(\cdot,z)}{s}\r]^{-\lz}\r\|_{L^\fai(\cx)}
\ls\lf\|(\psi_2)^\ast_{L}(f)\r\|_{L^\fai(\cx)},
\end{align*}
where the implicit positive constants depend on $n$, $\psi$, $\Psi$, $\lz$ and $\fai$.
This finishes the proof of \eqref{2.5} and hence of Proposition \ref{p2.1}.
\end{proof}

\begin{remark}\label{r2.1}
In the proof of Proposition \ref{p2.1}, to show \eqref{2.4}, we
introduce the new function $\psi$ and then convert \eqref{2.4} into \eqref{2.5}.
This is because of $\psi(0)=0$ and $\Psi(0)=0$, which imply that, for any $t,\,s\in(0,\fz)$, $K_{\psi(t\sqrt{L})\Psi(s\sqrt{L})}$ has a better decay
estimate than $K_{\psi_1(t\sqrt{L})\Psi(s\sqrt{L})}$ itself,
where $K_{\psi(t\sqrt{L})\Psi(s\sqrt{L})}$ and $K_{\psi_1(t\sqrt{L})\Psi(s\sqrt{L})}$ denote, respectively, the kernels of $\psi(t\sqrt{L})\Psi(s\sqrt{L})$ and $\psi_1(t\sqrt{L})\Psi(s\sqrt{L})$ (see \cite[Lemma 2.2(ii)]{sy18} for the details).
By this, we know that it is easier to estimate
\eqref{2.5} than \eqref{2.4} itself. These ideas originate from the proof of \cite[Proposition 3.1]{sy16} (see also the proof of \cite[Proposition 2.4]{sy18}).
\end{remark}

Moreover, to prove Theorem \ref{t1.1}, we also need the following atomic characterization of the
Musielak-Orlicz-Hardy space $H_{\fai,\,L}(\cx)$.

\begin{proposition}\label{p2.2}
Assume that $\cx$  is a space of homogeneous type.
Let $L$ satisfy Assumptions \ref{a1} and \ref{a2} and $\fai$ be as in Definition \ref{d1.2}.
Assume that $M\in\nn\cap(\frac{nq(\fai)}{2i(\fai)},\fz)$ and
$q\in([r(\fai)]'I(\fai),\fz)\cap(1,\fz)$,
where $q(\fai)$, $i(\fai)$, $r(\fai)$ and
$I(\fai)$ are, respectively, as in \eqref{1.6}, \eqref{1.5},
\eqref{1.7} and \eqref{1.4}.
Then the spaces $H_{\fai,\,L}(\cx)$ and
$H^{M,\,q}_{\fai,\,L,\,\mathrm{at}}(\cx)$ coincide with equivalent quasi-norms.
\end{proposition}

We point out that, when $\mu(\cx)=\fz$,
Proposition \ref{p2.2} was established in \cite[Theorem 5.4]{bckyy13b}. To show Proposition \ref{p2.2}
in the case of $\mu(\cx)<\fz$, we first introduce the following notions.

For any given $\nu\in(0,\fz)$ and $x\in\cx$, let
$\Gamma_{\nu}(x):=\{(y,t)\in\cx\times(0,\fz):\ d(x,y)<\nu t\}$
be the \emph{cone of aperture $\nu$ with vertex $x\in\cx$}.
Moreover, for any closed subset $F$ of $\cx$, denote by $\mathcal{R}_{\nu}(F)$ the \emph{union of all
cones with vertices in $F$}, namely, $\mathcal{R}_{\nu}(F):=\cup_{x\in
F}\Gamma_{\nu}(x)$ and, for any open subset $O$ of $\cx$, denote the
\emph{tent over $O$} by $T_{\nu}(O)$, which is defined as
$T_{\nu}(O):=[\mathcal{R}_{\nu}(O^\complement)]^{\complement}$. It is
easy to see that
$$T_{\nu}(O)=\lf\{(x,t)\in\cx\times(0,\fz):\ d(x,O^\complement)\ge\nu t\r\}.$$
In what follows, we denote $\Gamma_1(x)$ and
$T_1(O)$ \emph{simply by $\Gamma(x)$ and $\widehat{O}$}, respectively.

For any measurable function $g$ on $\cx\times(0,\fz)$ and $x\in\cx$, let
\begin{align}\label{2.9}
\ca(g)(x):=\lf\{\int_{\bgz(x)}|g(y,t)|^2
\frac{d\mu(y)}{V(x,t)}\frac{dt}{t}\r\}^{1/2}.
\end{align}
Coifman et al. \cite{cms85} introduced the
tent space $T^p_2(\rr^{n+1}_+)$ for any $p\in(0,\fz)$,
here and hereafter, $\rr^{n+1}_+:=\rn\times(0,\fz)$. The tent space
$T^p_2(\cx\times(0,\fz))$ on spaces of homogenous type was introduced
by Russ \cite{r07}. Recall that a measurable function $g$ on $\cx\times(0,\fz)$ is said
to belong to the \emph{tent space} $T^p_2(\cx\times(0,\fz))$ with $p\in(0,\fz)$ if
$$\|g\|_{T^p_2(\cx\times(0,\fz))}:=\lf\|\ca(g)\r\|_{L^p(\cx)}<\fz.$$
Furthermore, Harboure et al. \cite{hsv07}, and Jiang and Yang
\cite{jy11}, respectively, introduced the Orlicz tent spaces
$T_{\bfai}(\rr^{n+1}_+)$ and $T_{\Phi}(\cx\times(0,\fz))$.
Moreover, the Musielak-Orlicz tent spaces $T_{\fai}(\rr^{n+1}_+)$ and $T_{\fai}(\cx\times(0,\fz))$
were introduced, respectively, in \cite{hyy} and \cite{bckyy13b,yys4}.

Let $\fai$ be as in Definition \ref{d1.2}. Recall that
the Musielak-Orlicz tent space $T_\fai(\cx\times(0,\fz))$ is defined to be the set of \emph{all measurable functions
$g$ on $\cx\times(0,\fz)$ such that $\ca(g)\in L^{\fai}(\cx)$} and,
for any $g\in T_\fai(\cx\times(0,\fz))$, the \emph{quasi-norm} $\|g\|_{T_\fai(\cx\times(0,\fz))}$
is defined by setting
$$\|g\|_{T_\fai(\cx\times(0,\fz))}:=\lf\|\ca(g)\r\|_{L^{\fai}(\cx)}=
\inf\lf\{\lz\in(0,\fz):\
\int_{\cx}\fai\lf(x,\frac{\ca(g)(x)}{\lz}\r)\,d\mu(x)\le1\r\}.
$$

Let $p\in(1,\,\fz)$. A function $A$ on $\cx\times(0,\fz)$ is called a
\emph{$(T_\fai,\,p)$-atom} if
\begin{itemize}
  \item [(i)] there exists a ball $B\subset\cx$ such that $\supp(a)\subset\widehat{B}$;
  \item [(ii)] $\|A\|_{T^p_2(\cx\times(0,\fz))}\le
[\mu(B)]^{1/p}\|\mathbf{1}_B\|_{L^\fai(\cx)}^{-1}$.
\end{itemize}

In particular, if $\mu(\cx)<\fz$, a function $A$ on $\cx\times(0,\fz)$
is called a \emph{$(T_\fai,\,p)$-single-atom} if
$$\|A\|_{T^p_2(\cx\times(0,\fz))}\le[\mu(\cx)]^{1/p}\|\mathbf{1}_\cx\|^{-1}_{L^\fai(\cx)}.$$

Moreover, if $A$ is a $(T_\fai,p)$-atom (or $(T_\fai,\,p)$-single-atom) for any $p\in (1,\fz)$,
we then call $A$ a \emph{$(T_\fai,\fz)$-atom} (or \emph{$(T_\fai,\,\fz)$-single-atom}).

For functions in $T_\fai(\cx\times(0,\fz))$, we have the following
atomic decomposition.

\begin{lemma}\label{l2.4}
Let $\fai$ be as in Definition \ref{d1.2}. Then, when $\mu(\cx)=\fz$, for any $f\in
T_\fai(\cx\times(0,\fz))$, there exist $\{\lz_j\}_{j=1}^\fz\subset\cc$ and a sequence
$\{A_j\}_{j=1}^\fz$ of $(T_\fai,\,\fz)$-atoms associated, respectively, with $\{B_j\}_{j=1}^\fz\subset\cx$ such that,
for almost every $(x,t)\in\cx\times(0,\fz)$,
\begin{equation*}
f(x,t)=\sum_{j=1}^\fz\lz_jA_j(x,t).
\end{equation*}
Moreover, there exists a positive constant $C$ such that, for any
$f\in T_\fai(\cx\times(0,\fz))$,
\begin{align*}
\blz\lf(\{\lz_j A_j\}_{j=1}^\fz\r):=\inf\lf\{\lz\in(0,\fz):\
\sum_{j=1}^\fz\fai\lf(B_j,\frac{|\lz_j|}
{\lz\|\mathbf{1}_{B_j}\|_{L^\fai(\cx)}}\r)\le1\r\}
\le C\|f\|_{T_\fai(\cx\times(0,\fz))}.\nonumber
\end{align*}

Furthermore, when $\mu(\cx)<\fz$, for any $f\in T_\fai(\cx\times(0,\fz))$,
there exist $\{\lz_0\}\cup\{\lz_j\}_{j=1}^\fz\subset\cc$,
a $(T_\fai,\,\fz)$-single-atom $A_0$ and a sequence
$\{A_j\}_{j=1}^\fz$ of $(T_\fai,\,\fz)$-atoms associated, respectively, with $\{B_j\}_{j=1}^\fz\subset\cx$
such that, for almost every $(x,t)\in\cx\times(0,\fz)$,
\begin{equation*}
f(x,t)=\lz_0 A_0(x,t)+\sum_{j=1}^\fz\lz_jA_j(x,t),
\end{equation*}
where, for any $j\in\nn$, $B_j\neq\cx$.
Moreover, there exists a positive constant $C$ such that, for any
$f\in T_\fai(\cx\times(0,\fz))$,
\begin{align}\label{2.10}
&\blz\lf(\{\lz_0A_0\}\bigcup\{\lz_j A_j\}_{j=1}^\fz\r)\\ \nonumber
&\hs:=\inf\Bigg\{\lz\in(0,\fz):\
\fai\lf(\cx,\frac{|\lz_0|}
{\lz\|\mathbf{1}_{\cx}\|_{L^\fai(\cx)}}\r)+\lf.\sum_{j=1}^\fz\fai\lf(B_j,\frac{|\lz_j|}
{\lz\|\mathbf{1}_{B_j}\|_{L^\fai(\cx)}}\r)\le1\r\}\nonumber\\
&\hs\,\,\le C\|f\|_{T_\fai(\cx\times(0,\fz))}.\nonumber
\end{align}
\end{lemma}

When $\mu(\cx)=\fz$, the proof of Lemma \ref{l2.4} is similar to that of \cite[Theorem 3.1]{yys4}.
When $\mu(\cx)<\fz$, for any $k\in\zz$, let
$O_k:=\{x\in\cx:\ \ca(f)(x)>2^k\}$. If there exists $k_0\in\zz$ such that $O_{k_0}=\cx$,
then, for any $\nu\in(0,\fz)$,  $T_{\nu}(O_{k_0})=\cx\times(0,\fz)$ (see also \cite{r07}). Via this fact and similarly to the proof of \cite[Theorem 3.1]{yys4},
we can prove Lemma \ref{l2.3} in the case of $\mu(\cx)<\fz$,
the details being omitted here.

Now we show Proposition \ref{p2.2} by using Lemma \ref{l2.4}.

\begin{proof}[Proof of Proposition \ref{p2.2}]
When $\mu(\cx)=\fz$, Proposition \ref{p2.2} was proved in \cite[Theorem 5.4]{bckyy13b}.

Assume that $\mu(\cx)<\fz$. Let $M\in\nn$ with
$M>nq(\fai)/[2i(\fai)]$.
Assume that $\Phi$ is as in Lemma \ref{l2.1}.
Let $L^2_b(\cx\times(0,\fz))$ denote the \emph{set of all functions
$f\in L^2(\cx\times(0,\fz))$ with bounded supports}, namely, there exist a ball $B\subset\cx$
and $0<\wz{c}_1<\wz{c}_2<\fz$ such that $\supp(f)\subset B\times(\wz{c}_1,\wz{c}_2)$.
Then, for any $k\in\nn$,
$f\in L^2_{b}(\cx\times(0,\fz))$ and $x\in\cx$, the \emph{operator}
 $\pi_{\Phi,\,L,\,k}$ is defined by setting
\begin{equation*}
\pi_{\Phi,\,L,\,k}(f)(x):=C_{(\Phi,\,k)}\int_0^{\fz}\lf(t^2L\r)^{k+1}
\Phi(t\sqrt{L})(f(\cdot,\,t))(x)\,\frac{dt}{t},
\end{equation*}
where $C_{(\Phi,\,k)}$ is a positive constant such that
\begin{equation*}
C_{(\Phi,\,k)}\int_0^\fz  t^{2(k+1)}\Phi(t)t^2e^{-t^2} \,\frac{dt}{t}=1.
\end{equation*}

Denote by $T^b_{\fai}(\cx\times(0,\fz))$ and $T^{p,\,b}_2(\cx\times(0,\fz))$, with
$p\in(0,\fz)$, the \emph{set of all functions
in $T_\fai(\cx\times(0,\fz))$ and $T^p_2(\cx\times(0,\fz))$ with bounded supports}, respectively.
Then $T^b_{\fai}(\cx\times(0,\fz))
\subset T^{2,\,b}_2(\cx\times(0,\fz))$ as sets (see \cite[Proposition 3.1]{bckyy13b}).
Let $f\in T^b_{\fai}(\cx\times(0,\fz))$.
From Lemma \ref{l2.4}, \cite[Corollary 3.2]{bckyy13b} and the facts that
$T^b_{\fai}(\cx\times(0,\fz))\subset T^{2,\,b}_2(\cx\times(0,\fz))$ as sets and
$\pi_{\Phi,\,L,\,M}$ is bounded from $T_2^2(\cx\times(0,\fz))$ to $L^2(\cx)$, we deduce that there
exist a $(T_\fai,\,\fz)$-single-atom $A_0$, a sequence $\{A_j\}_{j=1}^\fz$
of $(T_\fai,\,\fz)$-atoms associated, respectively, with the balls $\{B_j\}_{j=1}^\fz$ and $\{\lz_0\}\cup\{\lz_j\}_{j=1}^\fz\subset \cc$ such that
\begin{align*}
\pi_{\Phi,\,L,\,M}(f)=\lz_0\pi_{\Phi,\,L,\,M}(A_0)+
\sum_{j=1}^\fz\lz_j\pi_{\Phi,\,L,\,M}(A_j)=:\lz_0\az_0+\sum_{j=1}^\fz\lz_j \az_j
\end{align*}
in $L^2(\cx)$ and
\begin{align*}
\Lambda\lf(\{\lz_0A_0\}\bigcup\{\lz_j A_j\}_{j=1}^\fz\r)\ls\|f\|_{T_{\fai}(\cx\times(0,\fz))},
\end{align*}
where, for any $j\in\nn$, $B_j\neq\cx$, and $\Lambda(\{\lz_0A_0\}\cup\{\lz_jA_j\}_{j=1}^\fz)$ is as in \eqref{2.10}.
By the proof of \cite[Theorem 5.4]{bckyy13b}, we know that, for each $j\in\nn$,
$\az_j$ is a $(\fai,\,q,\,\,M)_L$-atom, up to a harmless constant multiple, associated with the ball $B_j$ with $q\in(1,\fz)$.

Moreover, for any $f\in L^2(\cx)$ and $x\in\cx$, we define the \emph{Lusin area function, $S^{M+1,\,\Phi}_{L}(f)$, associated with $L$} by setting
\begin{equation*}
S^{M+1,\,\Phi}_{L}(f)(x):=\lf\{\int_{\bgz(x)}\lf|(t^2 L)^{M+1}\Phi(t\sqrt{L})(f)(y)\r|^2\frac{d\mu(y)\,dt}{V(x,t)t}\r\}^{1/2}.
\end{equation*}
Following the argument same as that used in the proof of \cite[Lemma 5.3]{bckyy},
we find that, for any $p\in (1,\,\fz)$, $S^{M+1,\,\Phi}_{L}$ is bounded  on $L^{p}(\cx)$.
Let $q\in(1,\fz)$. Then, for any $g\in L^{q'}(\cx)\cap L^{2}(\cx)$,
\begin{align*}
\lf|\int_{\cx}\az_0(x)g(x)\,d\mu(x)\r|&=\lf|\int_{\cx}\pi_{\Phi,\,L,\,M}(A_0)(x)g(x)\,d\mu(x)\r|\\
&\sim\lf|\int_0^\fz\int_{\cx}A_0(x,t)\lf(t^2L\r)^{M+1}
\Phi(t\sqrt{L})g(x)\,\frac{d\mu(x)\,dt}{t}\r|\\
&\ls\int_\cx\ca(A_0)(x)S^{M+1,\,\Phi}_{L}(g)(x)\,d\mu(x)\\
&\ls\lf\|\ca(A_0)\r\|_{L^{q}(\cx)}\lf\|S^{M+1,\,\Phi}_{L}(g)\r\|_{L^{q'}(\cx)}
\ls\|A_0\|_{T^q_2(\cx\times(0,\fz))}\|g\|_{L^{q'}(\cx)},
\end{align*}
which further implies that
$$\|\az_0\|_{L^{q}(\cx)}
\ls\|A_0\|_{T^q_2(\cx\times(0,\fz))}\ls[\mu(\cx)]^{1/q}\|\mathbf{1}_\cx\|_{L^\fai(\cx)}^{-1}.
$$
Thus, $\az_0$ is a $(\fai,\,q)$-single-atom up to a harmless constant multiple. Then the remainder of the proof of Proposition \ref{p2.2} is similar to that of \cite[Theorem 5.4]{bckyy13b}, the details being omitted here. This finishes the proof of Proposition \ref{p2.2}.
\end{proof}

\begin{remark}\label{r2.2}
We point out that \cite[Theorem 3.1]{bckyy13b} and \cite[Theorem 3.1]{yys4} hold true only when $\mu(\cx)=\fz$. In the case of $\mu(\cx)<\fz$, the correct version
of \cite[Theorem 3.1]{bckyy13b} and \cite[Theorem 3.1]{yys4} is Lemma \ref{l2.4}.
Similarly, the atomic characterization of the Musielak-Orlicz-Hardy
space $H_{\fai,\,L}(\cx)$ obtained in \cite[Theorem 5.4]{bckyy13b} and \cite[Theorem 5.5]{yys4} only
hold true under the assumption that $\mu(\cx)=\fz$. Thus, under Assumptions \ref{a1}
and \ref{a2}, Proposition \ref{p2.2} improves
the atomic characterization of $H_{\fai,\,L}(\cx)$ in \cite[Theorem 5.4]{bckyy13b} and \cite[Theorem 5.5]{yys4}
by removing the restriction that $\mu(\cx)=\fz$.

We also point out that Proposition \ref{p2.2} is new even when $\fai(x,t):=t^p$, with $p\in(0,1]$, for any
$x\in\cx$ and $t\in[0,\fz)$, and $\mu(\cx)<\fz$.
\end{remark}

Now we prove the equivalence of the spaces $H^{\phi,\,\az}_{\fai,\,L,\,\mathrm{max}}(\cx)$
and $H^{\phi}_{\fai,\,L,\,\mathrm{rad}}(\cx)$.

\begin{theorem}\label{t2.1}
Assume that $\cx$  is a space of homogeneous type.
Let $L$ satisfy Assumptions \ref{a1} and \ref{a2} and $\fai$ be as in Definition \ref{d1.2}.
Assume that $\phi\in\cs(\rr)$ is even and $\phi(0)=1$.
Then there exists a positive constant $C$, depending on $n$, $\phi$
and $\fai$, such that, for any $f\in L^2(\cx)$,
\begin{align}\label{2.11}
\lf\|\phi^\ast_L(f)\r\|_{L^\fai(\cx)}\le C\lf\|\phi^+_L(f)\r\|_{L^\fai(\cx)}.
\end{align}
Moreover, for any $\az\in(0,\fz)$, the spaces $H^{\phi,\,\az}_{\fai,\,L,\,\mathrm{max}}(\cx)$
and $H^{\phi}_{\fai,\,L,\,\mathrm{rad}}(\cx)$ coincide with equivalent quasi-norms.
\end{theorem}

\begin{proof}
Assume that $f\in L^2(\cx)$. For any $N\in(0,\fz)$ and $x\in\cx$, let
$$\cm^{\ast\ast}_{L,\,\phi,\,N}(f)(x):=\sup_{y\in\cx,\,s\in(0,\fz)}
\frac{|\phi(s\sqrt{L})(f)(y)|}{[1+\frac{d(x,y)}{s}]^N}.
$$
Then it is easy to see that, for any $x\in\cx$,
\begin{align}\label{2.12}
\phi^\ast_L(f)(x)\le 2^N\cm^{\ast\ast}_{L,\,\phi,\,N}(f)(x).
\end{align}
Moreover, it was proved in \cite[(3.4)]{sy18} that, for any $\tz\in(0,1)$ satisfying $N\tz>2n$,
there exists a positive constant $C$, depending on $\phi$, $N$, $\tz$ and $n$, such that,
for almost every $x\in\cx$,
\begin{align}\label{2.13}
\cm^{\ast\ast}_{L,\,\phi,\,N}(f)(x)\le C\lf[\cm\lf([\phi^+_L(f)]^\tz\r)(x)\r]^{1/\tz}.
\end{align}
Let $q_0\in(q(\fai),\fz)$, $p_0\in(0,i(\fai))$, $N_0\in(0,\fz)$ and $\tz_0\in(0,1)$ be
such that  $N_0\tz_0>2n$ and $\tz_0 q_0<p_0$. Then we know that $\fai$ is of uniformly
lower type $p_0$ and $\fai\in\aa_{q_0}(\cx)$. For any
$\az\in(0,\fz)$ and $g\in L^{q_0}_{\loc}(\cx)$, let
$$g=g\mathbf{1}_{\{x\in\cx:\ |g(x)|\le\az\}}+g\mathbf{1}_{\{x\in\cx:\
|g(x)|>\az\}}=:g_1+g_2.$$
Then it is easy to see that $\{x\in\cx:\ \cm(g)(x)>2\az\}\subset\{x\in\cx:\ \cm(g_2)(x)>\az\}$.
From this and Lemma \ref{l2.2}(iii),
we deduce that, for any $t\in(0,\fz)$,
\begin{align*}
&\int_{\{x\in\cx:\ \cm(g)(x)>2\az\}}\fai(x,t)\,d\mu(x)\\
&\hs\le\int_{\{x\in\cx:\
\cm(g_2)(x)>\az\}}\fai(x,t)\,d\mu(x)\le\frac{1}{\az^{q_0}}\int_{\cx}
\lf[\cm(g_2)(x)\r]^{q_0}\fai(x,t)\,d\mu(x)\\
&\hs\ls\frac{1}{\az^{q_0}}\int_{\cx}
|g_2(x)|^{q_0}\fai(x,t)\,d\mu(x)\sim\frac{1}{\az^{q_0}}\int_{\{x\in\cx:\
|g(x)|>\az\}} |g(x)|^{q_0}\fai(x,t)\,d\mu(x),
\end{align*}
which further implies that, for any $\az\in(0,\fz)$,
\begin{align*}
&\int_{\{x\in\cx:\
[\cm([\phi^+_L(f)]^{\tz_0})(x)]^{1/\tz_0}>\az\}}\fai(x,t)\,d\mu(x)\\ \nonumber
&\hs\ls\frac{1}{\az^{\tz_0 q_0}}\int_{\{x\in\cx:\ [\phi^+_L(f)(x)]^{\tz_0}>\frac{\az^{\tz_0}}{2}\}}
\lf[\phi^+_L(f)(x)\r]^{\tz_0 q_0}\fai(x,t)\,d\mu(x)\\ \nonumber
&\hs\ls\sz_{\phi^+_L(f),\,t}\lf(\frac{\az}{2^{1/\tz_0}}\r)+\frac{1}{\az^{\tz_0
q_0}} \int_{\frac{\az}{2^{1/\tz_0}}}^{\fz}\tz_0 q_0s^{\tz_0 q_0-1}
\sz_{\phi^+_L(f),\,t}(s)\,ds,
\end{align*}
here and hereafter,
$$\sz_{\phi^+_L(f),\,t}(\az):=\int_{\{x\in\cx:\
\phi^+_L(f)(x)>\az\}}\fai(x,t)\,d\mu(x).$$
By this, \eqref{2.13}, the
uniformly upper type $1$ and uniformly lower type $p_0$ properties of $\fai$
and $\tz_0 q_0<p_0$, we conclude that
\begin{align*}
&\int_{\cx}\fai\lf(x,\cm^{\ast\ast}_{L,\,\phi,\,N_0}(f)(x)\r)\,d\mu(x)\\
&\hs\ls\int_{\cx} \fai\lf(x,\lf[\cm\lf([\phi^+_L(f)]^{\tz_0}\r)(x)\r]^{1/\tz_0}\r)\,d\mu(x)
\ls\int_{\cx}\int_0^{[\cm([\phi^+_L(f)]^{\tz_0})(x)]^{1/\tz_0}}\frac{\fai(x,t)}{t}\,dt\,d\mu(x)\\
&\hs\sim\int_0^{\fz}\frac{1}{t}\int_{\{x\in\cx:\ [\cm([\phi^+_L(f)]^{\tz_0})(x)]^{1/\tz_0}>t\}}\fai(x,t)\,d\mu(x)\,dt\\
&\hs\ls\int_0^{\fz}\frac{1}{t}\int_{\{x\in\cx:\
\phi^+_L(f)(x)>\frac{t}{2^{1/\tz_0}}\}}\fai(x,t)\,d\mu(x)\,dt\\
&\hs\hs+\int_0^{\fz}\frac{1}{t^{\tz_0 q_0+1}}\lf\{\int_{\frac{t}
{2^{1/\tz_0}}}^{\fz}\tz_0 q_0s^{\tz_0 q_0-1}\sz_{\phi^+_L(f),\,t}
(s)\,ds\r\}\,dt\\
&\hs\sim\mathrm{J}_{\phi^+_L(f)} +\int_0^{\fz}\tz_0 q_0s^{\tz_0
q_0-1}\lf\{\int_0^{2^{1/\tz_0}s}
\frac{1}{t^{\tz_0 q_0+1}}\sz_{\phi^+_L(f),\,t}(s)\,dt\r\}\,ds\\
&\hs\ls\mathrm{J}_{\phi^+_L(f)}+\int_0^{\fz}\tz_0 q_0s^{\tz_0 q_0-1}
\sz_{\phi^+_L(f),\,t}(s)\fai(x,2^{1/\tz_0}s)
\lf\{\int_0^{2^{1/\tz_0}s}\lf[\frac{t}{2^{1/\tz_0}s}\r]^{p_0}
\frac{1}{t^{\tz_0 q_0+1}}\,dt\r\}\,ds\\
&\hs\ls\mathrm{J}_{\phi^+_L(f)}+\int_0^{\fz}\tz_0 q_0s^{\tz_0
q_0-1}\sz_{\phi^+_L(f),\,t}(s)
\frac{\fai(x,s)}{(2^{\frac{1}{\tz_0}}s)^{p_0}}
\lf\{\int_0^{2^{1/\tz_0}s}t^{p_0-\tz_0 q_0-1}\,dt\r\}\,ds\\
&\hs\ls\mathrm{J}_{\phi^+_L(f)} +\int_0^{\fz}\int_{\{x\in\cx:\
\phi^+_L(f)(x)>s\}}\frac{\fai(x,s)}{s}\,ds\sim\int_{\cx}
\fai\lf(x,\phi^+_L(f)(x)\r)\,d\mu(x),
\end{align*}
where
$$\mathrm{J}_{\phi^+_L(f)}:=\int_0^{\fz}
\int_{\{x\in\cx:\ \phi^+_L(f)(x)>t\}}\frac{\fai(x,t)}{t}\,d\mu(x)\,dt,$$
 which further implies that
\begin{align*}
\lf\|\cm^{\ast\ast}_{L,\,\phi,\,N_0}(f)\r\|_{L^\fai(\cx)}\ls
\lf\|\phi^+_L(f)\r\|_{L^\fai(\cx)}.
\end{align*}
From this and \eqref{2.12}, we deduce that \eqref{2.11} holds true.

Moreover, by the fact that, for any $f\in L^2(\cx)$, $\phi^+_L(f)\le\phi^\ast_L(f)$,
\eqref{2.11} and \eqref{2.4}, we conclude that, for any $\az\in(0,\fz)$ and $f\in L^2(\cx)$,
\begin{align*}
\lf\|\phi^{\ast}_{L,\,\az}(f)\r\|_{L^\fai(\cx)}\sim\lf\|\phi^+_L(f)\r\|_{L^\fai(\cx)}.
\end{align*}
From this, it follows that
\begin{align*}
\lf[H^{\phi,\,\az}_{\fai,\,L,\,\mathrm{max}}(\cx)\cap L^2(\cx)\r]
=\lf[H^{\phi}_{\fai,\,L,\,\mathrm{rad}}(\cx)\cap L^2(\cx)\r]
\end{align*}
with equivalent quasi-norms, which, combined with the fact that
$$H^{\phi,\,\az}_{\fai,\,L,\,\mathrm{max}}(\cx)\cap L^2(\cx)\ \mathrm{and}\
H^{\phi}_{\fai,\,L,\,\mathrm{rad}}(\cx)\cap L^2(\cx)$$
are, respectively, dense in the spaces
$H^{\phi,\,\az}_{\fai,\,L,\,\mathrm{max}}(\cx)$ and $H^{\phi}_{\fai,\,L,\,\mathrm{rad}}(\cx)$, and a
density argument, implies that the spaces $H^{\phi,\,\az}_{\fai,\,L,\,\mathrm{max}}(\cx)$ and
$H^{\phi}_{\fai,\,L,\,\mathrm{rad}}(\cx)$
coincide with equivalent quasi-norms. This finishes the proof of Theorem \ref{t2.1}.
\end{proof}

To prove Theorem \ref{t1.1}, we need the following Whitney type covering lemma
on the space $\cx$ of homogeneous type, which was essentially established in
\cite[Chapter III, Theorem 1.3]{cw71}.

\begin{lemma}\label{l2.5}
Let $\cx$ be a space of homogeneous type and
$\boz\subset\cx$ an open set with $\mu(\boz)<\fz$. Then there exists
a sequence $\{B(x_k,\rho_k)\}_{k=1}^\fz$ of balls, with $\{x_k\}_{k=1}^\fz\subset\boz$ and $\rho_k:=\dist(x_k,\boz^\complement)$ for any $k\in\nn$, such that
\begin{itemize}
\item[\rm(i)] $\bigcup_{k=1}^\fz B(x_k,\rho_k/2)=\boz$;
\item[\rm(ii)] for any $k_1,\,k_2\in\nn$,
$B(x_{k_1},\rho_{k_1}/10)\cap
  B(x_{k_2},\rho_{k_2}/10)=\emptyset$ if $k_1\neq k_2$.
\end{itemize}
\end{lemma}

Now we show Theorem \ref{t1.1} via Propositions \ref{p2.1} and \ref{p2.2}, Theorem \ref{t2.1}
and Lemma \ref{l2.5}.

\begin{proof}[Proof of Theorem \ref{t1.1}]
By Theorem \ref{t2.1}, we know that, for any $\az\in(0,\fz)$,
\begin{align}\label{2.14}
\lf[H^{\phi,\,\az}_{\fai,\,L,\,\mathrm{max}}(\cx)\cap L^2(\cx)\r]
= \lf[H^{\phi}_{\fai,\,L,\,\mathrm{rad}}(\cx)\cap L^2(\cx)\r].
\end{align}

Now we prove that, for any $M\in\nn\cap(nq(\fai)/(2i(\fai)),\fz)$ and $q\in([r(\fai)]'I(\fai),\fz]\cap(1,\fz]$,
\begin{align}\label{2.15}
\lf[H^{M,\,q}_{\fai,\,L,\,\mathrm{at}}(\cx)\cap L^2(\cx)\r]
\subset \lf[H^{\phi}_{\fai,\,L,\,\mathrm{rad}}(\cx)\cap L^2(\cx)\r].
\end{align}

We first assume that $\mu(\cx)<\fz$. In this case, we know that $\diam(\cx)<\fz$
and $\cx=B_0=:B(x_0,R_0)$ with $R_0\in(0,\fz)$
(see, for example, \cite[Lemma 5.1]{ny97}).
For any $(\fai,q)$-single-atom $\az$,
by \eqref{1.3} and the definition of $\az^+_L$, we conclude that
$\az^+_L\ls\cm(\az)$, where $\cm$ denotes the Hardy-Littlewood maximal operator on $\cx$.
Moreover, from $q\in([r(\fai)]'I(\fai),\fz]\cap(1,\fz]$,
it follows that there exists $p_1\in (I(\fai),\,1]$
such that $\fai$ is of uniformly upper type $p_1$ and $\fai\in \rh_{(q/p_1)'}(\cx)$,
which, together with the fact that $\az^+_L\ls\cm(\az)$, the H\"older inequality,
the boundedness of $\cm$ on $L^{q}(\cx)$ and Lemma \ref{l2.2}(iv),
further implies that, for any $\lz\in\cc$,
\begin{align}\label{2.16}
\int_{\cx}\fai\lf(x,|\lz|\az^+_L(x)\r)\,d\mu(x)
&=\int_{B_0}\fai\lf(x,|\lz|\az^+_L(x)\r)\,d\mu(x)
\ls\int_{B_0}\fai\lf(x,|\lz|\cm(\az)(x)\r)\,d\mu(x)\\ \nonumber
&\ls\int_{B_0}\fai\lf(x,|\lz|\|\mathbf{1}_{B_0}\|^{-1}_{L^\fai(\cx)}\r)
\lf[1+\cm(\az)(x)\|\mathbf{1}_{B_0}\|_{L^\fai(\cx)}\r]^{p_1}\,d\mu(x)\\ \nonumber
&\ls\fai\lf(B_0,|\lz|\|\mathbf{1}_{B_0}\|_{L^\fai(\cx)}^{-1}\r)+\|\mathbf{1}_{B_0}\|_{L^\fai(\cx)}^{p_1}
\lf\|\cm(\az)\r\|_{L^q(\cx)}^{p_1}\\ \nonumber
&\hs\times\lf\{\int_{B_0} \lf[\fai\lf(x,\, |\lz|\|\mathbf{1}_{B_0}\|_{L^\fai(\cx)}^{-1}\r)\r]^{(\frac{q}{p_1})'}\,d\mu(x)\r\}^{\frac{1}{(\frac{q}{p_1})'}}\\ \nonumber
&\ls\fai\lf(B_0,|\lz|\|\mathbf{1}_{B_0}\|_{L^\fai(\cx)}^{-1}\r)\sim\fai\lf(\cx,|\lz|\|\mathbf{1}_{\cx}\|_{L^\fai(\cx)}^{-1}\r).
\end{align}
Via Proposition \ref{p2.1} and Theorem \ref{t2.1},
to finish the proof of \eqref{2.15}, it suffices to prove that, for any $\lz\in\cc$
and $(\fai,\,q,\,M)_L$-atom $\az$ associated with the ball $B:=B(x_B,r_B)$ with $x_B\in\cx$ and $r_B\in(0,\fz)$,
\begin{align}\label{2.17}
\int_{\cx}\fai\lf(x,|\lz|\az^+_L(x)\r)\,d\mu(x)\ls\fai\lf(B,|\lz|\|\mathbf{1}_{B}\|^{-1}_{L^\fai(\cx)}\r),
\end{align}
where the implicit positive constant depends on $n$ and $\fai$.
Indeed, when $\mu(\cx)<\fz$, let $f\in [H^{M,\,q}_{\fai,\,L,\,\mathrm{at}}(\cx)\cap L^2(\cx)]$. Then there exist
$\{\lz_0\}\cup\{\lz_j\}_{j=1}^\fz\subset\cc$, a $(\fai,q)$-single-atom $\az_0$ and a sequence $\{\az_j\}_{j=1}^\fz$
of $(\fai,\,q,\,M)_L$-atoms associated, respectively, with the balls $\{B_j\}_{j=1}^\fz$ such that
\begin{equation*}
f=\lz_0\az_0+\sum_{j=1}^\fz\lz_j\az_j\ \text{in}\ L^2(\cx) \ \text{and} \ \
\|f\|_{H^{M,\,q}_{\fai,\,L,\,\mathrm{at}}(\cx)}\sim\blz\lf(\{\lz_0\az_0\}\bigcup\{\lz_j\az_j\}_{j=1}^\fz\r),
\end{equation*}
which, combined with \eqref{2.16} and \eqref{2.17}, implies that, for any $\lz\in(0,\fz)$,
\begin{align*}
\int_\cx\fai\lf(x,\frac{f^+_L(x)}{\lz}\r)\,d\mu(x)
&\ls\int_\cx\fai\lf(x,\frac{|\lz_0|(\az_0)_L^+(x)}{\lz}\r)\,d\mu(x)+\sum_{j=1}^\fz\int_\cx
\fai\lf(x,\frac{|\lz_j|(\az_j)_L^+(x)}{\lz}\r)\,d\mu(x)\\
&\ls\fai\lf(\cx,\frac{|\lz_0|}{\lz\|\mathbf{1}_{\cx}\|_{L^\fai(\cx)}}\r)+
\sum_{j=1}^\fz\fai\lf(B_j,\frac{|\lz_j|}{\lz\|\mathbf{1}_{B_j}\|_{L^\fai(\cx)}}\r).
\end{align*}
From this and Theorem \ref{t2.1}, it follows that $f\in [H^{\phi}_{\fai,\,L,\,\mathrm{rad}}(\cx)\cap L^2(\cx)]$ and
$$\|f\|_{H^{\phi}_{\fai,\,L,\,\mathrm{rad}}(\cx)}\ls
\|f\|_{H^{M,\,q}_{\fai,\,L,\,\mathrm{at}}(\cx)}.$$
Thus, \eqref{2.15} holds true in the case that $\mu(\cx)<\fz$.

When $\mu(\cx)=\fz$, by \eqref{2.17} and a similar argument as above, we find that \eqref{2.15} also holds true.

Now we show \eqref{2.17}. By \eqref{1.3}, we know that, for any $x\in\cx$,
\begin{align}\label{2.18}
\az^+_L(x)\ls\cm(\az)(x).
\end{align}
Moreover, from $q\in([r(\fai)]'I(\fai),\fz]\cap(1,\fz]$,
we deduce that there exists $p_2\in (I(\fai),\,1]$ such that $\fai$ is
of uniformly upper type $p_2$ and
$\fai\in \rh_{(q/p_2)'}(\cx)$. By this, \eqref{2.18}, the H\"older inequality,
the boundedness of $\cm$ on $L^{q}(\cx)$ and Lemma \ref{l2.2}(iv),
similarly to the proof of \eqref{2.16}, we find that, for any $\lz\in\cc$,
\begin{align}\label{2.19}
&\int_{4B}\fai\lf(x,|\lz|\az^+_L(x)\r)\,d\mu(x)
\ls\fai\lf(B,|\lz|\|\mathbf{1}_{B}\|_{L^\fai(\cx)}^{-1}\r).
\end{align}
For any $x\in S_j(B)$ with $j\in\nn$ and $j\ge2$, similarly to the proof of \cite[(2.26)]{yys16},
we conclude that, for any $s\in(0,2M)$,
\begin{align}\label{2.20}
\az^+_L(x)\ls 2^{-(n+s)}\|\mathbf{1}_B\|^{-1}_{L^\fai(\cx)}.
\end{align}
Let $s\in(nq(\fai)/i(\fai),2M)$. From $s>nq(\fai)/i(\fai)$, it follows that there exist $p_0\in(0,i(\fai))$
and $\wz{q}\in(q(\fai),\fz)$ such that $s>n\wz{q}/p_0$, $\fai$ is
of uniformly lower type $p_0$ and $\fai\in\aa_{\wz{q}}(\cx)$,
which, combined with \eqref{2.20} and Lemma \ref{l2.2}(iv), implies that
\begin{align*}
\int_{\cx\backslash(4B)}\fai\lf(x,|\lz|\az^+_L(x)\r)\,d\mu(x)&
=\sum_{j=2}^\fz\int_{S_j(B)}\fai\lf(x,|\lz|\az^+_L(x)\r)\,d\mu(x)\\
&\ls\sum_{j=2}^\fz2^{-j(n+s)p_0}\fai\lf(S_j(B),|\lz|\|\mathbf{1}_B\|^{-1}_{L^\fai(\cx)}\r)\\
&\ls\sum_{j=2}^\fz2^{-j[(n+s)p_0-n\wz{q}]}\fai\lf(B,|\lz|\|\mathbf{1}_B\|^{-1}_{L^\fai(\cx)}\r)
\ls\fai\lf(B,|\lz|\|\mathbf{1}_B\|^{-1}_{L^\fai(\cx)}\r).
\end{align*}
By this and \eqref{2.19}, we find that \eqref{2.17} holds true,
which completes the proof of \eqref{2.15}.

Now we prove that, for any $q\in(1,\fz]$,
\begin{align}\label{2.21}
\lf[H^{\phi,\,\az}_{\fai,\,L,\,\mathrm{max}}(\cx)\cap L^2(\cx)\r]
\subset \lf[H^{M,\,\fz}_{\fai,\,L,\,\mathrm{at}}(\cx)\cap L^2(\cx)\r]\subset
\lf[H^{M,\,q}_{\fai,\,L,\,\mathrm{at}}(\cx)\cap L^2(\cx)\r].
\end{align}
Via Proposition \ref{p2.1}, to show \eqref{2.21}, it suffices to prove that,
for any $f\in [H_{\fai,\,L,\,\mathrm{max}}(\cx)\cap L^2(\cx)]$,
$f\in H^{M,\,\fz}_{\fai,\,L,\,\mathrm{at}}(\cx)$ and
\begin{align}\label{2.22}
\|f\|_{H^{M,\,\fz}_{\fai,\,L,\,\mathrm{at}}(\cx)}\ls\|f\|_{H_{\fai,\,L,\,\mathrm{max}}(\cx)},
\end{align}
where the implicit positive constant depends on $n$, $M$ and $\fai$.
To prove \eqref{2.22}, we borrow some ideas from \cite{bdl}.

We first assume that $\mu(\cx)<\fz$. Let $\Phi$
be as in Lemma \ref{l2.1} and $M\in\nn$ with $M>nq(\fai)/[2i(\fai)]$.
Then, from the spectral calculus, it follows that there exists a positive constant $C_{(\Phi,\,M)}$ such that
\begin{align*}
f=C_{(\Phi,\,M)}\int_0^\fz(t\sqrt{L})^{2M}\Phi(t\sqrt{L})\Phi(t\sqrt{L})(f)\frac{dt}{t}
\end{align*}
in $L^2(\cx)$. For any $x\in\rr$, let
$$\psi(x):=C_{(\Phi,\,M)}\int_1^\fz (tx)^{2M}[\Phi(tx)]^2\frac{dt}{t}=
C_{(\Phi,\,M)}\int_x^\fz t^{2M}[\Phi(t)]^2\frac{dt}{t}.
$$
Then $\psi\in\cs(\rr)$ is an even function, $\psi(0)=1$ and, for any $s\in(0,\fz)$ and $x\in\rr$,
$$\psi(sx)=C_{(\Phi,\,M)}\int_s^\fz (tx)^{2M}[\Phi(tx)]^2\frac{dt}{t},
$$
which further implies that
\begin{align*}
\psi(s\sqrt{L})(f)=C_{(\Phi,\,M)}\int_s^\fz (t\sqrt{L})^{2M}\Phi(t\sqrt{L})\Phi(t\sqrt{L})(f)\frac{dt}{t}
\end{align*}
in $L^2(\cx)$.
Furthermore, for any $f\in L^2(\cx)$ and $x\in\cx$, let
$$\cm_L(f)(x):=\sup_{d(x,y)<8t,\,t\in(0,\fz)}\lf[\lf|\psi(t\sqrt{L})(f)(y)\r|+\lf|\Phi(t\sqrt{L})(f)(y)\r|\r].
$$

Let $R_0:=\diam(X)/2$. Then
\begin{align}\label{2.23}
f&=C_{(\Phi,\,M)}\int_0^{R_0}(t\sqrt{L})^{2M}\Phi(t\sqrt{L})\Phi(t\sqrt{L})(f)\frac{dt}{t}
+C_{(\Phi,\,M)}\int_{R_0}^{\fz}\cdots\\ \nonumber
&=C_{(\Phi,\,M)}\int_0^{R_0}(t\sqrt{L})^{2M}\Phi(t\sqrt{L})\Phi(t\sqrt{L})(f)\frac{dt}{t}
+\psi(R_0\sqrt{L})(f)=:f_1+f_2.
\end{align}

By Proposition \ref{p2.1}, we find that
\begin{align}\label{2.24}
\|\cm_L(f)\|_{L^\fai(\cx)}\sim\|f\|_{H_{\fai,\,L,\,\mathrm{max}}(\cx)},
\end{align}
where the implicit positive constant depends on $n$, $M$, $\fai$ and $\Phi$.
From $R_0=\diam(\cx)/2$, it follows that, for any $x,\,y\in\cx$,
$$\lf|\psi(R_0\sqrt{L})(f)(x)\r|\le\sup_{d(z,y)<6R_0}\lf|\psi(R_0\sqrt{L})(f)(z)\r|\le\cm_L(f)(y),
$$
which, together with \eqref{2.24}, further implies that
\begin{align*}
\|f_2\|_{L^\fz(\cx)}&\le\inf_{y\in\cx}\cm_L(f)(y)=\lf[\inf_{y\in\cx}\cm_L(f)(y)\r]
\|\mathbf{1}_\cx\|_{L^\fai(\cx)}\|\mathbf{1}_\cx\|_{L^\fai(\cx)}^{-1}.
\end{align*}
Let
$$\az_{0,\,1}:=\lf[\inf_{y\in\cx}\cm_L(f)(y)\r]^{-1}\|\mathbf{1}_\cx\|_{L^\fai(\cx)}^{-1}f_2 \ \ \text{and}
\ \ \lz_{0,\,1}:=\lf[\inf_{y\in\cx}\cm_L(f)(y)\r]\|\mathbf{1}_\cx\|_{L^\fai(\cx)}.$$
Then $f_2=\lz_{0,\,1}\az_{0,\,1}$
and $\az_{0,\,1}$ is a $(\fai,\fz)$-single-atom. Furthermore, for any $\lz\in(0,\fz)$,
\begin{align}\label{2.25}
\fai\lf(\cx,\frac{\lz_{0,\,1}}{\lz\|\mathbf{1}_{\cx}\|_{L^\fai(\cx)}}\r)
\le\int_\cx\fai\lf(x,\frac{\cm_L(f)(x)}{\lz}\r)\,d\mu(x).
\end{align}

Now we deal with the term $f_1$. For any $k\in\zz$, let
$$\boz_k:=\lf\{x\in\cx:\ \cm_L(f)(x)>2^k\r\}.$$
By the facts that $\cm_L(f)$ is lower continuous and $\cx$ is bounded,
we conclude that there exists $k_0\in\zz$ such that $\boz_{k_0}=\cx$ and $\boz_{k_0+1}\neq\cx$.
Without loss of generality we may assume that $k_0=0$. Then,
for any $t\in(0,\fz)$ and $k\in\zz_+$, let
\begin{equation}\label{2.26}
\boz_{k,\,t}:=\begin{cases}\boz_0, \ \ \ &k=0,\\
\{x\in\cx:\ d(x,\boz_k^\complement)>4t\}, \ \ \ &k\in\nn,
\end{cases}
\end{equation}
and $T_{k,\,t}:=\boz_{k,\,t}\backslash\boz_{k+1,\,t}$.
It is easy to see that, for any $t\in(0,\fz)$, $\cx=\cup_{k=0}^\fz T_{k,\,t}$,
which implies that
\begin{align}\label{2.27}
f_1&=\sum_{k=0}^\fz C_{(\Phi,\,M)}\int_0^{R_0}(t\sqrt{L})^{2M}\Phi(t\sqrt{L})\lf[\Phi(t\sqrt{L})(f)
\mathbf{1}_{T_{k,\,t}}\r]\frac{dt}{t}=:\sum_{k=0}^\fz f_{1,\,k}.
\end{align}
We first estimate $f_{1,\,0}$. For any $x\in\cx$,
$$f_{1,\,0}(x)=C_{(\Phi,\,M)}\int_0^{R_0}\int_{T_{0,\,t}}K_{(t\sqrt{L})^{2M}\Phi(t\sqrt{L})}(x,y)
\Phi(t\sqrt{L})(f)(y)\,d\mu(y)\frac{dt}{t}.
$$
We deal with $f_{1,\,0}$ by considering the following two cases.

\emph{Case 1)} $x\in\boz_1^\complement$. In this case,
by Lemma \ref{l2.1}, we know that, for any $t\in(0,\fz)$,
$$\supp\lf(K_{(t\sqrt{L})^{2M}\Phi(t\sqrt{L})}(x,\cdot)\r)\subset\{z\in\cx:\ d(x,z)\le t\}\subset T_{0,\,t},
$$
which implies that
\begin{align}\label{2.28}
|f_{1,\,0}(x)|&=C_{(\Phi,\,M)}\lf|\int_0^{R_0}\int_{\cx}K_{(t\sqrt{L})^{2M}\Phi(t\sqrt{L})}(x,y)
\Phi(t\sqrt{L})(f)(y)\,d\mu(y)\frac{dt}{t}\r|\\ \nonumber
&=C_{(\Phi,\,M)}\lf|\int_0^{R_0}(t^2\sqrt{L})^M\Phi(t\sqrt{L})
\Phi(t\sqrt{L})(f)(x)\,\frac{dt}{t}\r|\\ \nonumber
&\le\lim_{\epz\to0}|\psi(\epz\sqrt{L})(f)(x)|+|\psi(R_0\sqrt{L})(f)(x)|.
\end{align}
Moreover, since $x\in\boz_1^\complement$, it follows that, for any $s\in(0,\fz)$,
$|\psi(s\sqrt{L})f(x)|\le2$, which, combined with \eqref{2.28}, implies that
\begin{align}\label{2.29}
|f_{1,\,0}(x)|\le4.
\end{align}

\emph{Case 2)} $x\in\boz_1$. In this case, we write
\begin{align}\label{2.30}
f_{1,\,0}(x)&=C_{(\Phi,\,M)}\int_0^{R_0}\int_{T_{0,\,t}}K_{(t\sqrt{L})^{2M}\Phi(t\sqrt{L})}(x,y)
\Phi(t\sqrt{L})(f)(y)\,d\mu(y)\frac{dt}{t}\\ \nonumber
&=C_{(\Phi,\,M)}\int_0^{d(x,\boz_1^\complement)/5}
\int_{T_{0,\,t}}K_{(t\sqrt{L})^{2M}\Phi(t\sqrt{L})}(x,y)
\Phi(t\sqrt{L})(f)(y)\,d\mu(y)\frac{dt}{t}\\ \nonumber
&\hs+C_{(\Phi,\,M)}\int_{d(x,\boz_1^\complement)/5}^{d(x,\boz_1^\complement)/3}\cdots
+C_{(\Phi,\,M)}\int_{d(x,\boz_1^\complement)/3}^{R_0}\cdots\\ \nonumber
&=:\mathrm{E}_1(x)+\mathrm{E}_2(x)+\mathrm{E}_3(x).
\end{align}
For any $t\in(0,d(x,\boz_1^\complement)/5)$ and $y\in T_{0,\,t}$, we know that $d(x,y)\ge t$,
which, together with Lemma \ref{l2.1}, implies that
$K_{(t\sqrt{L})^{2M}\Phi(t\sqrt{L})}(x,y)=0$
and hence $\mathrm{E}_1(x)=0$.
For the term $\mathrm{E}_2(x)$, by Lemma \ref{l2.1} again, we find that
\begin{align}\label{2.31}
|\mathrm{E}_2(x)|\ls\int_{d(x,\boz_1^\complement)/5}^{d(x,\boz_1^\complement)/3}\frac{1}{\mu(B(x,t))}
\int_{B(x,t)\cap T_{0,\,t}}\sup_{y\in T_{0,\,t}}\lf|\Phi(t\sqrt{L})f(y)\r|\,d\mu(y)\,\frac{dt}{t}.
\end{align}
From the definition of $T_{0,\,t}$, we deduce that, for any $y\in T_{0,\,t}$,
there exists $z\in\boz_1^\complement$ such that $d(y,z)<6t$, which further implies that
$$\lf|\Phi(t\sqrt{L})f(y)\r|\le\cm_L(f)(z)\le2.
$$
By this and \eqref{2.31}, we conclude that $|\mathrm{E}_2(x)|\ls1$.

Now we estimate the term $\mathrm{E}_3(x)$. From $x\in\boz_1$ and $t\in(d(x,\boz_1^\complement)/3,R_0)$,
it follows that, for any $t\in(d(x,\boz_1^\complement)/3,R_0)$,
$$\supp\lf(K_{(t\sqrt{L})^{2M}\Phi(t\sqrt{L})}(x,\cdot)\r)\subset\{z\in\cx:\ d(x,z)\le t\}\subset T_{0,\,t},
$$
which, combined with the definition of $\psi(s\sqrt{L})$, implies that
\begin{align}\label{2.32}
\mathrm{E}_3(x)=\psi(s_0\sqrt{L})(f)(x)-\psi(R_0\sqrt{L})(f)(x)
\end{align}
with $s_0=d(x,\boz_1^\complement)/3$.
By $t\in(d(x,\boz_1^\complement)/3,R_0)$, we know that there exists $z\in\boz_1^\complement$ such that
$d(x,z)<3t$, which implies that, for any $t\in(d(x,\boz_1^\complement)/3,R_0)$,
$$\lf|\psi(t\sqrt{L})f(x)\r|\le\cm_L(f)(z)\le2.
$$
From this and \eqref{2.32}, we deduce that $|\mathrm{E}_3(x)|\ls1$,
which, together with \eqref{2.30} and the fact that $\mathrm{E}_1(x)=0$ and $|\mathrm{E}_2(x)|\ls1$, implies that,
for any $x\in\boz_1$, $|f_{1,\,0}(x)|\ls1$. By this and \eqref{2.29}, we find that, for any $x\in\cx$,
$|f_{1,\,0}(x)|\ls1$, which, combined with the fact that $\cx=\boz_0=\{x\in\cx:\ \cm_L(f)(x)>1\}$,
further implies that, for any $\lz\in(0,\fz)$,
\begin{align*}
\|f_{1,\,0}\|_{L^\fz(\cx)}\ls\inf_{y\in\cx}\cm_L(f)(y).
\end{align*}
From this, we deduce that there exists a positive constant $\wz{C}_1$ such that
$$\|f_{1,\,0}\|_{L^\fz(\cx)}\le \wz{C}_1\lf[\inf_{y\in\cx}\cm_L(f)(y)\r]\|\mathbf{1}_\cx\|_{L^\fai(\cx)}
\|\mathbf{1}_\cx\|_{L^\fai(\cx)}^{-1}.$$
Let
$$\az_{1,\,0}:=\wz{C}^{-1}_1\lf[\inf_{y\in\cx}\cm_L(f)(y)\r]^{-1}\|\mathbf{1}_\cx\|_{L^\fai(\cx)}^{-1}f_{1,\,0}$$
and
$\lz_{1,\,0}:=\wz{C}_1[\inf_{y\in\cx}\cm_L(f)(y)]\|\mathbf{1}_\cx\|_{L^\fai(\cx)}$.
Then $f_{1,\,0}=\lz_{1,\,0}\az_{1,\,0}$
and $\az_{1,\,0}$ is a $(\fai,\,\fz)$-single-atom.
Moreover, for any $\lz\in(0,\fz)$,
\begin{align}\label{2.33}
\fai\lf(\cx,\frac{\lz_{1,\,0}}{\lz\|\mathbf{1}_{\cx}\|_{L^\fai(\cx)}}\r)
\le\int_\cx\fai\lf(x,\frac{\cm_L(f)(x)}{\lz}\r)\,d\mu(x).
\end{align}

Now we deal with the term $f_{1,\,k}$ with $k\in\nn$. Observe that, for each $k\in\nn$, $\boz_k$ is open and $\mu(\boz_k)<\fz$.
Then, by Lemma \ref{l2.5}, we know that, for each $k\in\nn$, there exist a sequence
$\{x_{k,\,i}\}_{i=1}^\fz\subset \boz_k$ of points and a sequence $\{B_{k,\,i}\}_{i=1}^\fz$ of balls such that

\begin{itemize}
\item[(i)] $\bigcup_{i=1}^\fz B_{k,\,i}=\boz_k$;

\item[(ii)] for any $i_1,\,i_2\in\nn$, $\frac{1}{5}B_{k,\,i_1}\cap
  \frac{1}{5}B_{k,\,i_2}=\emptyset$ if $i_1\neq i_2$, where $B_{k,\,i}:=B(x_{k,\,i},\rho_{k,\,i}/2)$
with $\rho_{k,\,i}:=\dist(x_{k,\,i},\boz_k^\complement)$.
\end{itemize}
For any $k,\,i\in\nn$ and $t\in(0,\fz)$, let $B^t_{k,\,i}:=B(x_{k,\,i}, r_{B_{k,\,i}}+2t)$. Moreover, for any $t\in(0,\fz)$ and $k,\,i\in\nn$, let
$$R^t_{k,\,i}:=\begin{cases}
T_{k,\,t}\cap B^t_{k,\,i} \ \ &\text{if}\ T_{k,\,t}\cap B_{k,\,i}\neq\emptyset,\\
0\ \ \ &\text{if}\ T_{k,\,t}\cap B_{k,\,i}=\emptyset,
\end{cases}
$$
and $E^t_{k,\,i}:=R^t_{k,\,i}\backslash(\bigcup_{\ell>k}R^t_{\ell,\,i})$.
Then it is easy to see that, for any $t\in(0,\fz)$ and $k\in\nn$,
$$T_{k,\,t}=\bigcup_{i\in\nn}E^t_{k,\,i}.
$$
Thus, for any $k\in\nn$,
$$f_{1,\,k}=\sum_{i\in\nn}C_{(\Phi,\,M)}\int_0^{R_0}(t\sqrt{L})^{2M}
\Phi(t\sqrt{L})\lf[\Phi(t\sqrt{L})(f)
\mathbf{1}_{E^t_{k,\,i}}\r]\frac{dt}{t}.$$

For any $k,\,i\in\nn$, when $E^t_{k,\,i}=\emptyset$, let $\az_{k,\,i}:=0$ and $\lz_{k,\,i}:=0$;
when $E^t_{k,\,i}\neq\emptyset$, let $\az_{k,\,i}:=L^M b_{k,\,i}$ and $\lz_{k,\,i}:=2^k\|\mathbf{1}_{B_{k,\,i}}\|_{L^\fai(\cx)}$
with
$$b_{k,\,i}=\frac{C_{(\Phi,\,M)}}{\lz_{k,\,i}}\int_0^{R_0}t^{2M}\Phi(t\sqrt{L})\lf[\Phi(t\sqrt{L})(f)
\mathbf{1}_{E^t_{k,\,i}}\r]\frac{dt}{t}.
$$
Then
$$\sum_{k\in\nn}f_{1,\,k}=\sum_{k,\,i\in\nn}\lz_{k,\,i}\az_{k,\,i}
$$
in $L^2(\cx)$.

It is obvious that, when $\az_{k,\,i}=0$, then $\az_{k,\,i}$ is a $(\fai,\fz,M)_L$-atom associated with
the ball $B_{k,\,i}$. Now we assume that $\az_{k,i}\neq0$.
If $r_{B_{k,\,i}}<t/2$, then $d(x_{B_{k,\,i}}, \boz_k^\complement)=2r_{B_{k,\,i}}<t$, which implies that
$$B^t_{k,\,i}=B(x_{B_{k,\,i}}, r_{B_{k,\,i}}+2t)\subset\lf\{x\in\cx:\ d(x,\boz_k^\complement)<4t\r\}.
$$
By this, we know that $R^t_{k,\,i}=T_{k,\,t}\cap B^t_{k,\,i}=\emptyset$. Therefore, when $\az_{k,\,i}\neq0$, then
$r_{B_{k,\,i}}\ge t/2$, which, together with the definition of $b_{k,\,i}$ and Lemma \ref{l2.1}, further implies that,
for any $j\in\{0,\,1,\,\ldots,\,M\}$, $\supp(L^j b_{k,\,i})\subset 8B_{k,\,i}$.
Now we prove that there exists a positive constant $\wz{C}_2$ such that, for any $j\in\{0,\,1,\,\ldots,\,M\}$,
\begin{align}\label{2.34}
\lf\|L^j b_{k,\,i}\r\|_{L^\fz(\cx)}\le\wz{C}_2 (8r_{B_{k,\,i}})^{2(M-j)}\lf\|\mathbf{1}_{8B_{k,\,i}}\r\|_{L^\fai(\cx)}^{-1}.
\end{align}
For any $j\in\{0,\,1,\,\ldots,\,M-1\}$, from $r_{B_{k,\,i}}\ge t/2$ and Lemma \ref{l2.1}, it follows that,
for any $x\in\cx$,
\begin{align}\label{2.35}
\lf|L^jb_{k,\,i}(x)\r|&\le\frac{C_{(\Phi,\,M)}}{\lz_{k,\,i}}\int_0^{2r_{B_{k,\,i}}}t^{2(M-j)}
\lf|(t^2L)^j\Phi(t\sqrt{L})\lf[\Phi(t\sqrt{L})(f)
\mathbf{1}_{E^t_{k,\,i}}\r](x)\r|\frac{dt}{t}\\ \nonumber
&\ls\frac{1}{\lz_{k,\,i}}\int_0^{2r_{B_{k,\,i}}}t^{2(M-j)}\frac{1}{\mu(B(x,t))}
\int_{B(x,t)\cap E^t_{k,\,i}}\lf|\Phi(t\sqrt{L})f(y)\r|\,d\mu(y)\frac{dt}{t}\\ \nonumber
&\ls\frac{1}{\lz_{k,\,i}}\int_0^{2r_{B_{k,\,i}}}t^{2(M-j)}\frac{1}{\mu(B(x,t))}
\int_{B(x,t)\cap T_{k,\,t}}\lf|\Phi(t\sqrt{L})f(y)\r|\,d\mu(y)\frac{dt}{t}.
\end{align}
By the definition of $T_{k,\,t}$, we know that, for each $y\in T_{k,\,t}$, there exists $z\in\boz_{k+1}^\complement$
such that $d(y,z)<4t$, which further implies that, for any $y\in T_{k,\,t}$,
$$\lf|\Phi(t\sqrt{L})(f)(y)\r|\le\cm_L(f)(z)\le2^{k+1}.
$$
From this and \eqref{2.35}, we deduce that, for any $j\in\{0,\,1,\,\ldots,\,M-1\}$ and $x\in\cx$,
\begin{align}\label{2.36}
\lf|L^jb_{k,\,i}(x)\r|&
\ls\frac{2^k}{\lz_{k,\,i}}\int_0^{2r_{B_{k,\,i}}}t^{2(M-j)}\frac{dt}{t}\ls r_{B_{k,\,i}}^{2(M-j)}
\lf\|\mathbf{1}_{B_{k,\,i}}\r\|^{-1}_{L^\fai(\cx)}\\ \nonumber
&\ls(8r_{B_{k,\,i}})^{2(M-j)}\lf\|\mathbf{1}_{8B_{k,\,i}}\r\|^{-1}_{L^\fai(\cx)}.
\end{align}
When $j=M$,
$$L^Mb_{k,\,i}(x)=\frac{C_{(\Phi,\,M)}}{\lz_{k,\,i}}\int_0^{2r_{B_{k,\,i}}}
(t^2L)^M\Phi(t\sqrt{L})\lf[\Phi(t\sqrt{L})(f)
\mathbf{1}_{E^t_{k,\,i}}\r](x)\frac{dt}{t}.
$$
By
$$E^t_{k,\,i}=[(T_{k,\,i}\cap B^t_{k,\,i})\backslash(T_{k,\,i}\cap F^t_{k,\,i})],
$$
where $F^t_{k,\,i}:=\cup_{\ell>i}B^t_{k,\,\ell}=\{x\in\cx:\ d(x,\cup_{\ell>i}B_{k,\ell})<2t\}$,
we conclude that
\begin{align}\label{2.37}
L^Mb_{k,\,i}(x)
&=\frac{C_{(\Phi,\,M)}}{\lz_{k,i}}\int_0^{2r_{B_{k,\,i}}}
\int_{T_{k,\,i}\cap B^t_{k,\,i}}K_{(t^2L)^M\Phi(t\sqrt{L})}(x,y)\Phi(t\sqrt{L})(f)
(y)\,d\mu(y)\frac{dt}{t}\\ \nonumber
&\hs-\frac{C_{(\Phi,\,M)}}{\lz_{k,\,i}}\int_0^{2r_{B_{k,\,i}}}
\int_{T_{k,\,i}\cap F^t_{k,\,i}\cap B^t_{k,\,i}}\cdots
=:\frac{C_{(\Phi,\,M)}}{\lz_{k,\,i}}\mathrm{I}_{k,\,i}(x)
+\frac{C_{(\Phi,\,M)}}{\lz_{k,\,i}}\mathrm{J}_{k,\,i}(x).
\end{align}
It was proved in \cite[Lemma 3.1]{bdl1} that, for any $k,\,i\in\nn$ and $x\in\cx$,
$|\mathrm{I}_{k,\,i}(x)|+|\mathrm{J}_{k,\,i}(x)|\ls2^k$, which, combined with \eqref{2.37}, implies that
$$\lf\|L^Mb_{k,\,i}\r\|_{L^\fz(\cx)}\ls\|\mathbf{1}_{B_{k,\,i}}\|^{-1}_{L^\fai(\cx)}\ls\lf\|\mathbf{1}_{8B_{k,\,i}}\r\|^{-1}_{L^\fai(\cx)}.
$$
From this and \eqref{2.36}, we deduce that there exists a positive constant $\wz{C}_2$ such that,
for any $k,\,i\in\nn$ and $j\in\{0,\,1,\,\ldots,\,M\}$,
$$\|L^jb_{k,\,i}\|_{L^\fz(\cx)}
\le\wz{C}_2(8r_{B_{k,\,i}})^{2(M-j)}\lf\|\mathbf{1}_{8B_{k,\,i}}\r\|^{-1}_{L^\fai(\cx)}.$$
Thus, \eqref{2.34} holds true.
For any $k,\,i\in\nn$, let
$$\wz{b}_{k,\,i}:=b_{k,\,i}/\wz{C}_2,\ \ \wz{\az}_{k,\,i}:=\az_{k,\,i}/\wz{C}_2\ \
\text{and}\ \ \wz{\lz}_{k,\,i}:=\wz{C}_2\lz_{k,\,i}=\wz{C}_2 2^k\|\mathbf{1}_{B_{k,\,i}}\|_{L^\fai(\cx)}.$$
Then, for any $k,\,i\in\nn$, $\wz{\az}_{k,\,i}$ is a $(\fai,\,\fz,\,M)_L$-atom associated with the ball $8B_{k,\,i}$ and
\begin{align}\label{2.38}
\sum_{k\in\nn}f_{1,\,k}=\sum_{k,\,i\in\nn}\wz{\lz}_{k,\,i}\wz{\az}_{k,\,i}.
\end{align}
Moreover, by Lemma \ref{l2.2}(iv), we find that, for any $\lz\in(0,\fz)$,
\begin{align}\label{2.39}
\sum_{k,\,i\in\nn}\fai\lf(8B_{k,\,i},\frac{\wz{\lz}_{k,\,i}}{\lz\|\mathbf{1}_{8B_{k,\,i}}\|_{L^\fai(\cx)}}\r)
&\ls\sum_{k,\,i\in\nn}\fai\lf(\frac{B_{k,\,i}}{5},\frac{2^k}{\lz}\r)\sim
\sum_{k\in\nn}\fai\lf(\boz_{k},\frac{2^k}{\lz}\r).
\end{align}
Furthermore, similarly to the proof of \cite[Lemma 3.4]{yys4} (see also \cite[Lemma 5.4]{k14}), we conclude that
$$
\sum_{k\in\nn}\fai\lf(\boz_{k},\frac{2^k}{\lz}\r)
\ls\int_{\cx}\fai\lf(x,\frac{\cm_L(f)(x)}{\lz}\r)\,d\mu(x),$$
which, together with \eqref{2.39}, \eqref{2.25} and \eqref{2.33}, implies that, for any $\lz\in(0,\fz)$,
\begin{align}\label{2.40}
&\fai\lf(\cx,\frac{\lz_{0,\,1}}{\lz\|\mathbf{1}_{\cx}\|_{L^\fai(\cx)}}\r)+
\fai\lf(\cx,\frac{\lz_{1,\,0}}{\lz\|\mathbf{1}_{\cx}\|_{L^\fai(\cx)}}\r)
+\sum_{k,\,i\in\nn}\fai\lf(8B_{k,\,i},\frac{\wz{\lz}_{k,\,i}}{\lz\|\mathbf{1}_{8B_{k,\,i}}\|_{L^\fai(\cx)}}\r)\\ \nonumber
&\hs\ls\int_{\cx}\fai\lf(x,\frac{\cm_L(f)(x)}{\lz}\r)\,d\mu(x).
\end{align}
Moreover, from \eqref{2.23}, \eqref{2.27} and \eqref{2.38}, we deduce that
$$f=f_1+f_2=\lz_{0,\,1}\az_{0,\,1}+\lz_{1,\,0}\az_{1,\,0}+\sum_{k,\,i\in\nn}\wz{\lz}_{k,\,i}\wz{\az}_{k,\,i}.
$$
By this and \eqref{2.40},
we further conclude that $f\in H^{M,\,\fz}_{\fai,\,L,\,\mathrm{at}}(\cx)\subset H^{M,\,q}_{\fai,\,L,\,\mathrm{at}}(\cx)$
with $q\in(1,\fz]$ and
$$\|f\|_{H^{M,\,q}_{\fai,\,L,\,\mathrm{at}}(\cx)}\le\|f\|_{H^{M,\,\fz}_{\fai,\,L,\,\mathrm{at}}(\cx)}\ls \|f\|_{H_{\fai,\,L,\,\mathrm{max}}(\cx)},
$$
namely, \eqref{2.22} and hence \eqref{2.21} hold true when $\mu(\cx)<\fz$.

Now we assume that $\mu(\cx)=\fz$. In this case, we have
$$
f=C_{(\Phi,\,M)}\int_0^{\fz}(t\sqrt{L})^{2M}\Phi(t\sqrt{L})\Phi(t\sqrt{L})(f)\frac{dt}{t}.
$$
Similarly to \eqref{2.38} and \eqref{2.40}, we know that there exist a sequence $\{\az_{k,\,i}\}_{k\in\zz,\,i\in\nn}$ of $(\fai,\fz,M)_L$-atoms and $\{\lz_{k,\,i}\}_{k\in\zz,\,i\in\nn}\subset\cc$ such that
$$f=\sum_{k\in\zz,\,i\in\nn}\lz_{k,\,i}\az_{k,\,i}
$$
and, for any $q\in(1,\fz]$,
$$\|f\|_{H^{M,\,q}_{\fai,\,L,\,\mathrm{at}}(\cx)}\ls \|f\|_{H_{\fai,\,L,\,\mathrm{max}}(\cx)}.$$
Thus, \eqref{2.22} and hence \eqref{2.21} hold true when $\mu(\cx)=\fz$.

Therefore, from \eqref{2.14}, \eqref{2.15} and \eqref{2.21}, it follows that,
for any $M\in\nn$ with $M>nq(\fai)/(2i(\fai))$ and $q\in([r(\fai)]'I(\fai),\fz]\cap(1,\fz]$,
$$\lf[H^{\phi,\,\az}_{\fai,\,L,\,\mathrm{max}}(\cx)\cap L^2(\cx)\r]
= \lf[H^{\phi}_{\fai,\,L,\,\mathrm{rad}}(\cx)\cap L^2(\cx)\r]=
\lf[H^{M,\,q}_{\fai,\,L,\,\mathrm{at}}(\cx)\cap L^2(\cx)\r]
$$
with equivalent quasi-norms. By this, the fact that the spaces
$H^{M,\,q}_{\fai,\,L,\,\mathrm{at}}(\cx)\cap L^2(\cx)$,
$H^{\phi,\,\az}_{\fai,\,L,\,\mathrm{max}}(\cx)\cap L^2(\cx)$ and
$H^{\phi}_{\fai,\,L,\,\mathrm{rad}}(\cx)\cap L^2(\cx)$ are
dense, respectively, in the spaces
$H^{M,\,q}_{\fai,\,L,\,\mathrm{at}}(\cx)$,
$H^{\phi,\,\az}_{\fai,\,L,\,\mathrm{max}}(\cx)$ and $H^{\phi}_{\fai,\,L,\,\mathrm{rad}}(\cx)$, and a
density argument, we conclude that the spaces $H^{M,\,q}_{\fai,\,L,\,\mathrm{at}}(\cx)$,
$H^{\phi,\,\az}_{\fai,\,L,\,\mathrm{max}}(\cx)$ and $H^{\phi}_{\fai,\,L,\,\mathrm{rad}}(\cx)$
coincide with equivalent quasi-norms, which, together with Proposition \ref{p2.2}, further implies that,
for any $M\in\nn$ with $M>nq(\fai)/(2i(\fai))$ and $q\in([r(\fai)]'I(\fai),\fz]\cap(1,\fz]$,
the spaces $H_{\fai,\,L}(\cx)$, $H^{M,\,q}_{\fai,\,L,\,\mathrm{at}}(\cx)$,
$H^{\phi,\,\az}_{\fai,\,L,\,\mathrm{max}}(\cx)$ and $H^{\phi}_{\fai,\,L,\,\mathrm{rad}}(\cx)$
coincide with equivalent quasi-norms.
This finishes the proof of Theorem \ref{t1.1}.
\end{proof}

\section{Proof of Theorem \ref{t1.3}\label{s3}}

In this section, we give the proof of Theorem \ref{t1.3}.

\begin{proof}[Proof of Theorem \ref{t1.3}]
Assume that $nq(\fai)/i(\fai)<n+\dz_0$, where $\dz_0\in(0,1]$ is as in Assumption \ref{a3}.
Let $q\in(q(\fai)[r(\fai)]',\fz]\cap(1,\fz]$ and $M\in\nn$ satisfy $M>nq(\fai)/[2i(\fai)]$.
Using Theorems \ref{t1.1} and \ref{t1.2}, to show Theorem \ref{t1.3},
we only need to show that
the spaces $H^{M,\,q}_{\fai,\,L,\,\mathrm{at}}(\cx)=H^{q}_{\fai,\,\mathrm{at}}(\cx)$
coincide with the equivalent quasi-norm.

We first show that
\begin{equation}\label{3.1}
\lf[H^{M,\,q}_{\fai,\,L,\,\mathrm{at}}(\cx)\cap L^2(\cx)\r]\subset\lf[H^{q}_{\fai,\,\mathrm{at}}(\cx)\cap L^2(\cx)\r].
\end{equation}
Let $f\in H^{M,\,q}_{\fai,\,L,\,\mathrm{at}}(\cx)\cap L^2(\cx)$.
If $\mu(\cx)<\fz$, by Definition \ref{d1.4}, we know that there exist $\{\lz_0\}\cup\{\lz_j\}_{j=1}^\fz\subset\cc$,
a $(\fai,\,q)$-single-atom and a sequence $\{\az_j\}_{j=1}^\fz$
of $(\fai,\,q,\,M)_L$-atoms associated, respectively, with the balls $\{B_j\}_{j=1}^\fz$ such that
\begin{equation}\label{3.2}
f=\lz_0\az_0+\sum_{j=1}^\fz\lz_j\az_j\ \text{in}\ L^2(\cx) \ \text{and} \ \
\|f\|_{H^{M,\,q}_{\fai,\,L,\,\mathrm{at}}(\cx)}\sim\blz\lf(\{\lz_0\az_0\}\bigcup\{\lz_j\az_j\}_{j=1}^\fz\r).
\end{equation}
Let $\az$ be a $(\fai,\,q,\,M)_L$-atom  associated with the ball $B$.
From the definition of $(\fai,\,q,\,M)_L$-atoms, we deduce that
there exists $b\in D(L)$ such that $\az=Lb$,
which, combined with Assumption \ref{a4} and an argument similar to that used in the proof of \cite[Theorem 9.1]{hlmmy},
further implies that
\begin{equation*}
\int_\cx\az(x)\,d\mu(x)=0.
\end{equation*}
By this, we know that $\az$ is a $(\fai,\,q,\,0)$-atom,
which, together with \eqref{3.2}, implies that $f\in
H^{q}_{\fai,\,\mathrm{at}}(\cx)\cap L^2(\cx)$ and
\begin{equation}\label{3.3}
\|f\|_{H^{q}_{\fai,\,\mathrm{at}}(\cx)}\ls\|f\|_{H^{M,\,q}_{\fai,\,L,\,\mathrm{at}}(\cx)}.
\end{equation}
When $\mu(\cx)=\fz$, similarly to \eqref{3.3}, we conclude that \eqref{3.3} also holds true in this case.
Thus, \eqref{3.1} holds true.

Now we prove that
\begin{equation}\label{3.4}
\lf[H^{q}_{\fai,\,\mathrm{at}}(\cx)\cap L^2(\cx)\r]\subset\lf[H^{M,\,q}_{\fai,\,L,\,\mathrm{at}}(\cx)\cap L^2(\cx)\r].
\end{equation}
Let $f\in H^{q}_{\fai,\,\mathrm{at}}(\cx)\cap L^2(\cx)$.
If $\mu(\cx)<\fz$, by Definition \ref{d1.7}, we know that there exist $\{\lz_0\}\cup\{\lz_j\}_{j=1}^\fz\subset\cc$,
a $(\fai,\,q)$-single-atom and a sequence $\{\az_j\}_{j=1}^\fz$
of $(\fai,\,q,\,0)$-atoms associated, respectively, with the balls $\{B_j\}_{j=1}^\fz$ such that
\begin{equation}\label{3.5}
f=\lz_0\az_0+\sum_{j=1}^\fz\lz_j\az_j\ \text{in}\ L^2(\cx) \ \text{and} \ \
\|f\|_{H^{q}_{\fai,\,\mathrm{at}}(\cx)}\sim\blz\lf(\{\lz_0\az_0\}\bigcup\{\lz_j\az_j\}_{j=1}^\fz\r).
\end{equation}
Similarly to \eqref{2.16}, we know that, for any $\lz\in\cc$,
\begin{align}\label{3.6}
\int_{\cx}\fai\lf(x,|\lz|(\az_0)^+_{L,\,\loc}(x)\r)\,d\mu(x)\ls\fai\lf(\cx,|\lz|\|\mathbf{1}_{\cx}\|_{L^\fai(\cx)}^{-1}\r).
\end{align}
Let $\az$ be a $(\fai,\,q,\,0)$-atom associated with the ball $B:=B(x_0,r_0)$, where $x_0\in\cx$ and $r_0\in(0,\fz)$.
By using Assumption \ref{a3} and $\int_\cx\az(x)\,d\mu(x)=0$ and then similarly to the proof of
\eqref{2.17}, we conclude that, for any $\lz\in\cc$,
\begin{align*}
\int_{\cx}\fai\lf(x,|\lz|\az^+_{L,\,\loc}(x)\r)\,d\mu(x)\ls\fai\lf(B,|\lz|\|\mathbf{1}_{B}\|^{-1}_{L^\fai(\cx)}\r),
\end{align*}
where the implicit positive constant depends on $n$ and $\fai$,
which, combined with \eqref{3.5} and \eqref{3.6}, implies that, for any $\lz\in(0,\fz)$,
\begin{align*}
\int_\cx\fai\lf(x,\frac{f^+_{L,\,\loc}(x)}{\lz}\r)\,d\mu(x)
&\ls\int_\cx\fai\lf(x,\frac{|\lz_0|(\az_0)_{L,\,\loc}^+(x)}{\lz}\r)\,d\mu(x)+\sum_{j=1}^\fz\int_\cx
\fai\lf(x,\frac{|\lz_j|(\az_j)_{L,\,\loc}^+(x)}{\lz}\r)\,d\mu(x)\\
&\ls\fai\lf(\cx,\frac{|\lz_0|}{\lz\|\mathbf{1}_{\cx}\|_{L^\fai(\cx)}}\r)+
\sum_{j=1}^\fz\fai\lf(B_j,\frac{|\lz_j|}{\lz\|\mathbf{1}_{B_j}\|_{L^\fai(\cx)}}\r).
\end{align*}
From this, it follows that $f\in [h_{\fai,\,L,\,\mathrm{rad}}(\cx)\cap L^2(\cx)]$ and
$$\|f\|_{h_{\fai,\,L,\,\mathrm{rad}}(\cx)}\ls\|f\|_{H^{q}_{\fai,\,\mathrm{at}}(\cx)},$$
which, together with Theorem \ref{t1.2}, further implies that
$f\in [H^{M,\,q}_{\fai,\,L,\,\mathrm{at}}(\cx)\cap L^2(\cx)]$ and
\begin{align}\label{3.7}
\|f\|_{H^{M,\,q}_{\fai,\,L,\,\mathrm{at}}(\cx)}\ls\|f\|_{H^{q}_{\fai,\,\mathrm{at}}(\cx)}.
\end{align}
By this, we know that \eqref{3.4} holds true in the case that $\mu(\cx)<\fz$.
When $\mu(\cx)=\fz$, via replacing $f^+_{L,\,\loc}$ by $f^+_{L}$ and repeating the proof of \eqref{3.7}, we conclude that \eqref{3.7} also holds true in this case.
Thus, \eqref{3.4} holds true.

From \eqref{3.1} and \eqref{3.4}, it follows that
$$\lf[H^{q}_{\fai,\,\mathrm{at}}(\cx)\cap L^2(\cx)\r]=\lf[H^{M,\,q}_{\fai,\,L,\,\mathrm{at}}(\cx)\cap L^2(\cx)\r],$$
which, combined with the fact that the spaces
$H^{q}_{\fai,\,\mathrm{at}}(\cx)\cap L^2(\cx)$ and
$H^{M,\,q}_{\fai,\,L,\,\mathrm{at}}(\cx)\cap L^2(\cx)$ are
dense, respectively, in the spaces
$H^{q}_{\fai,\,\mathrm{at}}(\cx)$ and $H^{M,\,q}_{\fai,\,L,\,\mathrm{at}}(\cx)$, and a
density argument, implies that the spaces $H^{q}_{\fai,\,\mathrm{at}}(\cx)$
and $H^{M,\,q}_{\fai,\,L,\,\mathrm{at}}(\cx)$ coincide with equivalent quasi-norms.
This finishes the proof of Theorem \ref{t1.3}.
\end{proof}

\section{\!\!\!\!Applications to Musielak-Orlicz-Hardy spaces on Lipschitz domains\label{s4}}

 In this section, we always assume that
$\boz$ is a \emph{strongly Lipschitz domain} of $\rn$,
namely, $\boz$ is a proper open connected set of
$\rn$ whose boundary is a finite union of parts of rotated graphs of
Lipschitz maps, at most one of these parts possibly unbounded. It is
well known that strongly Lipschitz domains include special
Lipschitz domains, bounded Lipschitz domains and exterior domains; see, for example,
\cite{ar03,at01a} for their definitions and properties.
Moreover, it is worth pointing out that strongly Lipschitz domains in $\rn$ are spaces of
homogeneous type (see, for example, \cite{ar03,yys2,yys3}).

Now we describe the divergence form elliptic operators considered in
this section. The most typical example is the Laplace operator on bounded Lipschitz domains of
$\rn$ with the Dirichlet or the Neumann boundary condition.
If $\boz$ is either $\rn$ or a strongly Lipschitz domain of $\rn$, we denote by
$W^{1,\,2}(\boz)$ the usual \emph{Sobolev space on $\boz$} equipped
with the norm $(\|f\|^2_{L^2 (\boz)}+\|\nabla f\|^2_{L^2
(\boz)})^{1/2}$, where $\nabla f$ denotes the \emph{distributional
gradient} of $f$. In what follows, $W^{1,\,2}_0 (\boz)$ stands for
the \emph{closure} of $C^{\fz}_c (\boz)$ in $W^{1,\,2}(\boz)$, where
$C^{\fz}_c (\boz)$ denotes the set of all \emph{$C^\fz$ functions on
$\boz$ with compact supports contained in $\boz$}.

If $A:\ \rn\to M_n (\cc)$ is a measurable function, define
$$\|A\|_{\fz}:=\esup_{x\in\rn,\,|\xi|=|\eta|=1}|A(x)\xi
\cdot\ol{\eta}|,$$
where $M_n (\cc)$ denotes the \emph{set of all $n\times
n$ complex-valued matrixes}, $\xi,\,\eta\in\cc^n$ and $\ol{\eta}$
denotes the \emph{conjugate vector} of $\eta$. For any $\dz\in(0,1]$,
denote by the \emph{symbol $\ca(\dz)$} the class of all measurable
functions $A:\ \rn\to M_n (\cc)$ satisfying the ellipticity condition, namely, for
any $x\in\rn$ and $\xi\in\cc^n$,
\begin{equation*}
\|A\|_{\fz}\le \dz^{-1} \ \text{and}\ \Re
(A(x)\xi\cdot\xi)\ge\dz|\xi|^2,
\end{equation*}
here and hereafter, $\Re (A(x)\xi\cdot\xi)$ denotes the \emph{real
part} of $ A(x)\xi\cdot\xi$. Denote by $\ca$
the \emph{union of all $\ca(\dz)$ for any $\dz\in(0,1]$}.

When $A\in\ca$ and $V$ is a closed subspace of $W^{1,\,2}(\boz)$
containing $W^{1,\,2}_0 (\boz)$, denote by $L$ the {\it
maximal-accretive operator} (see \cite[p.\,23, Definition 1.46]{o04}
for its definition) on $L^2 (\boz)$ with largest domain
$D(L)\subset V$ such that, for any $f\in D(L)$ and $g\in V$,
\begin{equation}\label{4.1}
\langle Lf,g\rangle=\int_{\boz}A(x)\nabla f(x)\cdot\ol{\nabla
g(x)}\,dx,
\end{equation}
where $\langle \cdot,\cdot\rangle$ denotes the \emph{interior product} in
$L^2 (\boz)$. In this sense, for any $f\in D(L)$, we write
\begin{equation}\label{4.2}
Lf:=-\divz(A\nabla f).
\end{equation}

We recall the following Dirichlet and Neumann boundary conditions
of $L$ from \cite[p.\,152]{ar03}.

\begin{definition}\label{d4.1}
Let $\boz$ be either $\rn$ or a strongly Lipschitz domain
of $\rn$, and $L$ as in \eqref{4.2}. The operator $L$ is said
to satisfy the \emph{Dirichlet boundary condition} (for simplicity,
DBC) if $V:= W^{1,\,2}_0(\boz)$ in \eqref{4.1} and the \emph{Neumann boundary condition}
(for simplicity, NBC) if $V:= W^{1,\,2}(\boz)$ in \eqref{4.1}.
Denote by $L_D$ and $L_N$ the second-order elliptic operator with the Dirichlet and the Neumann
boundary condition on the strongly Lipschitz domain $\boz$, respectively.
\end{definition}

We point out that, when $\boz:=\rn$, $W^{1,\,2}_0
(\boz)=W^{1,\,2}(\boz)$. Thus, in this case, DBC and NBC are
identical.

Denote by $\cs(\rn)$ the space of all Schwartz functions equipped
with the well-known topology determined by a countable family of
seminorms and by $\cs'(\rn)$ its \emph{dual space} equipped
with the weak-$\ast$ topology (namely, the space of all
\emph{tempered distributions}).
For any $m\in\nn$, define
$$\cs_m(\rn):=\lf\{\phi\in\cs(\rn):\ \sup_{x\in\rn}\sup_{
\bz\in\zz^n_+,\,|\bz|\le m+1}(1+|x|)^{(m+2)(n+1)}|\partial^\bz_x\phi(x)|\le1\r\}.
$$
Then, for any $f\in\cs'(\rn)$, the \emph{non-tangential grand maximal function} $f^\ast_m$
of $f$ is defined by setting, for any $x\in\rn$,
$$f^\ast_m(x):=\sup_{\phi\in\cs_m(\rn)}\sup_{|y-x|<t,\,t\in(0,\fz)}|f\ast\phi_t(y)|,
$$
where, for any $t\in(0,\fz)$, $\phi_t(\cdot):=t^{-n}\phi(\frac{\cdot}{t})$.
When $m(\fai):=\lfz n[q(\fai)/i(\fai)-1]\rfz$, where $q(\fai)$ and $i(\fai)$ are,
respectively, as in \eqref{1.6} and \eqref{1.5}, we denote $f^\ast_{m(\fai)}$ simply by $f^\ast$.

Let $\cd(\boz)$ denote the \emph{space of all
infinitely differentiable functions with compact supports in $\boz$}
equipped with the inductive topology and $\cd'(\boz)$ its
\emph{topological dual} equipped with the weak-$\ast$
topology, which is called the \emph{space of distributions on $\boz$}.

Now we recall the notion of the Musielak-Orlicz-Hardy space $H_\fai(\rn)$ introduced by Ky \cite{k14} and then introduce the ``geometric" Musielak-Orlicz-Hardy spaces $H_{\fai,\,r}(\boz)$
and $H_{\fai,\,z}(\boz)$ on domains, respectively, as follows.

\begin{definition}\label{d4.2}
Let $\fai$ be as in Definition \ref{d1.2} and $\boz$ a subdomain in $\rn$.

\begin{itemize}
\item[(i)] The \emph{Musielak-Orlicz Hardy space
$H_\fai(\rn)$} is defined as the set of all $f\in\cs'(\rn)$ such that
$f^\ast\in L^\fai(\rn)$
equipped with the \emph{quasi-norm} $\|f\|_{H_\fai(\rn)}:=\|f^\ast\|_{L^\fai(\rn)}$.
\item[(ii)] The \emph{Musielak-Orlicz-Hardy space
$H_{\fai,\,z}(\boz)$} is defined by setting
\begin{equation*}
H_{\fai,\,z}(\boz):=\lf\{f\in H_\fai(\rn):\ f=0\ \text{on}\
(\ol{\boz})^\complement\r\}/\{f\in H_\fai(\rn):\ f=0\ \text{on}\
\boz\},
\end{equation*}
where $\ol{\boz}$ denotes the closure of $\boz$ in $\rn$. Moreover, the \emph{quasi-norm} of the element in $H_{\fai,\,z}(\boz)$ is defined to be the quotient norm.
\item[(iii)] A distribution $f$ on $\boz$ is said to belong to the
\emph{Musielak-Orlicz-Hardy space $H_{\fai,\,r}(\boz)$} if $f$ is the restriction to $\boz$ of a distribution $F$ in $H_{\fai}(\rn)$, namely,
\begin{align*}
H_{\fai,\,r}(\boz):=&\,\{f\in \mathcal{D}'(\boz):\ \text{there exists an}
\ F\in H_{\fai}(\rn)\ \text{such that}\ F|_{\boz}=f\}\\
=&\,H_{\fai}(\rn)/\{F\in H_{\fai}(\rn):\ F=0\ \text{on}\ \boz\}.
\end{align*}
Moreover, for any $f\in H_{\fai,\,r}(\boz)$, the \emph{quasi-norm}
$\|f\|_{H_{\fai,\,r}(\boz)}$ of $f$ in $H_{\fai,\,r}(\boz)$
is defined by setting
$$\|f\|_{H_{\fai,\,r}(\boz)}:=\inf\lf\{\|F\|_{H_{\fai}(\rn)}:\
F\in H_{\fai}(\rn)\ \text{and}\ F|_{\boz}=f\r\},$$
where the infimum is taken over all $F\in H_{\fai}(\rn)$ satisfying
$F=f$ on $\boz$.
\end{itemize}
\end{definition}

When $\fai(x,t):=t^p$, with $p\in(0,1]$,
for any $x\in\rn$ and $t\in[0,\fz)$, the Hardy spaces $H^p_r (\boz)$ and $H^p_z (\boz)$
on the domain $\boz$ were introduced and studied by Chang et al.
\cite{cds99} and Chang et al. \cite{cks92,cks93}. We point out that the Hardy spaces $H^p_r (\boz)$ and $H^p_z (\boz)$ naturally appear in the study of the regularity of the Green operators
for the Dirichlet boundary problem, respectively, for the Neumann
boundary problem (see, for example, \cite{cks93,cds99,dhmmy}).
Moreover, when $\fai$ is as in \eqref{1.8}, the Orlicz-Hardy spaces
$H_{\Phi,\,r}(\boz)$ and $H_{\Phi,\,z}(\boz)$ and the weighted local Orlicz-Hardy spaces
$h^\Phi_{\omega,\,r}(\boz)$ and $h^\Phi_{\omega,\,z}(\boz)$
were studied in \cite{ccyy13,yys2,yys3}. Furthermore, the ``geometric" Musielak-Orlicz-Hardy
space $H_{\fai,\,L,\,r}(\boz)$ on strongly Lipschitz domains associated with the Schr\"odinger
operator was studied in \cite{cfyy16}.

Let $\boz$ be a strongly Lipschitz domain in $\rn$ and $L_D$ a second-order elliptic operator with the real bounded coefficients and the Dirichlet boundary condition on $\boz$.
Denote by $\{K_t\}_{t>0}$ the kernels of the semigroup $\{e^{-tL_D}\}_{t>0}$.
By \cite[Corollary 3.2.8]{d89} (see also \cite{at01a}), we know that there exist
positive constants $C$ and $c$ such that, for any $t\in(0,\fz)$ and $x,\,y\in\boz$,
\begin{equation}\label{4.3}
|K_t(x,y)|\ls\frac{1}{t^{n/2}}\exp\lf\{-\frac{|x-y|^2}{ct}\r\}\le
\frac{C}{|B_\boz(x,\sqrt{t})|}\exp\lf\{-\frac{|x-y|^2}{ct}\r\},
\end{equation}
where $B_\boz(x,\sqrt{t}):=B(x,\sqrt{t})\cap\boz$, which implies that $L_D$ satisfies Assumption \ref{a2}.
Moreover, from \cite[Proposition 5]{at01a}, it follows that
there exist positive constants $C$, $c$ and $\dz_1\in(0,1]$ such that,
for any $t\in(0,\fz)$ and $x,\,y_1,\,y_2\in\boz$ with $|y_1-y_2|<\sqrt{t}/2$,
\begin{equation}\label{4.4}
|K_t(x,y_1)-K_t(x,y_2)|\le\frac{C}{|B_\boz(x,\sqrt{t})|}
\lf[\frac{|y_1-y_2|}{\sqrt{t}}\r]^{\dz_1}
\exp\lf\{-\frac{|x-y|^2}{ct}\r\}
\end{equation}
(see also \cite{ar03}).

Assume further that $L_D$ has symmetrical coefficients. Then $L_D$ satisfies Assumptions \ref{a1} and \ref{a2}.

For $H_{\fai,\,r}(\boz)$, we have the following equivalent characterizations via the (local)
Musielak-Orlicz-Hardy spaces on $\boz$ associated with $L_D$.

\begin{theorem}\label{t4.1}
Let $\boz$ be a strongly Lipschitz domain in $\rn$ such that $\boz^\complement$
is unbounded, and $\fai$ as in Definition \ref{d1.2}.
Assume that $L_D$ is a second-order self-adjoint elliptic operator on $\boz$, with the Dirichlet boundary condition,
satisfying \eqref{4.3} and \eqref{4.4}. Let $\dz_1\in(0,1]$, $r(\fai)$, $I(\fai)$, $q(\fai)$
and $i(\fai)$ be, respectively, as in \eqref{4.4}, \eqref{1.7}, \eqref{1.4}, \eqref{1.6} and \eqref{1.5}.
If $nq(\fai)/i(\fai)<n+\dz_1$, then, for any $q\in([r(\fai)]'I(\fai),\fz]\cap(1,\fz]$ and
$M\in\nn\cap(\frac{nq(\fai)}{2i(\fai)},\fz)$,
the Musielak-Orlicz-Hardy spaces $H_{\fai,\,L_D}(\boz)$, $H^{M,\,q}_{\fai,\,L_D,\,\mathrm{at}}(\boz)$,
$H_{\fai,\,L_D,\,\mathrm{max}}(\boz)$, $H_{\fai,\,L_D,\,\mathrm{rad}}(\boz)$
and $H_{\fai,\,r}(\boz)$ coincide with equivalent quasi-norms. In particular, if $\boz$
is bounded, then the spaces
$$H_{\fai,\,L_D}(\boz),\ H^{M,\,q}_{\fai,\,L_D,\,\mathrm{at}}(\boz),\
H_{\fai,\,L_D,\,\mathrm{max}}(\boz),\ H_{\fai,\,L_D,\,\mathrm{rad}}(\boz),\ h_{\fai,\,L_D,\,\mathrm{max}}(\boz),\ h_{\fai,\,L_D,\,\mathrm{rad}}(\boz)$$
and $H_{\fai,\,r}(\boz)$ coincide with equivalent quasi-norms.
\end{theorem}

When $\fai(x,t):=t^p$, with $p\in(0,1]$, for any $x\in\rn$ and $t\in[0,\fz)$,
we denote the spaces $H_{\fai,\,L_D,\,\mathrm{max}}(\boz)$,
$$H_{\fai,\,L_D,\,\mathrm{rad}}(\boz),\ H_{\fai,\,r}(\boz),\ h_{\fai,\,L_D,\,\mathrm{max}}(\boz)\
\mathrm{and}\ h_{\fai,\,L_D,\,\mathrm{rad}}(\boz)$$
simply, respectively, by $H^p_{L_D,\,\mathrm{max}}(\boz)$,
$H^p_{L_D,\,\mathrm{rad}}(\boz)$, $H^p_{r}(\boz)$, $h^p_{L_D,\,\mathrm{max}}(\boz)$ and
$h^p_{L_D,\,\mathrm{rad}}(\boz)$. Then, as the corollary of Theorem \ref{t4.1}, we have the following conclusion.

\begin{corollary}\label{c4.1}
Let $\boz$ be a strongly Lipschitz domain in $\rn$ such that $\boz^\complement$
is unbounded.
Assume that $L_D$ is a second-order self-adjoint elliptic operator on $\boz$, with the Dirichlet boundary condition,
satisfying \eqref{4.3} and \eqref{4.4}. Let $\dz_1\in(0,1]$ be as in \eqref{4.4}
and $p\in(\frac{n}{n+\dz_1},1]$.
Then, for any $q\in(1,\fz]$ and $M\in\nn\cap(\frac{n}{2p},\fz)$, the Hardy spaces $H^p_{L_D}(\boz)$, $H^{p,\,M,\,q}_{L_D,\,\mathrm{at}}(\boz)$, $H^p_{L_D,\,\mathrm{max}}(\boz)$, $H^p_{L_D,\,\mathrm{rad}}(\boz)$
and $H^p_{r}(\boz)$ coincide with equivalent quasi-norms. In particular, if $\boz$ is bounded, then the spaces
$$H^p_{L_D}(\boz),\ H^{p,\,M,\,q}_{L_D,\,\mathrm{at}}(\boz),\
H^p_{L_D,\,\mathrm{max}}(\boz),\ H^p_{L_D,\,\mathrm{rad}}(\boz),\ h^p_{L_D,\,\mathrm{max}}(\boz),\ h^p_{L_D,\,\mathrm{rad}}(\boz)$$
and $H^p_{r}(\boz)$ coincide with equivalent quasi-norms.
\end{corollary}

\begin{remark}\label{r4.1}
\begin{itemize}
\item[(i)] The equivalences of the spaces $H^p_{L_D}(\boz)$,
$H^p_{L_D,\,\mathrm{max}}(\boz)$, $H^p_{L_D,\,\mathrm{rad}}(\boz)$ and $H^p_{r}(\boz)$ in Corollary \ref{c4.1} were obtained in \cite[Theorem 1, Proposition 5 and Theorem 20]{ar03} (which require $p=1$) and \cite[Theorem 4.4 and Remark 4.5(c)]{bdl1} (which require $p\in(\frac{n}{n+\dz_1},1]$). In particular, when $\boz$ is bounded, the equivalences of the spaces $h^p_{L_D,\,\mathrm{max}}(\boz)$, $h^p_{L_D,\,\mathrm{rad}}(\boz)$ and $H^p_{r}(\boz)$ were obtained in \cite[Remark 3.2, Theorem 4.4 and Remark 4.5]{bdl1}.

\item[(ii)] Let $\Phi$ be an Orlicz function satisfying that $i(\Phi)\in(\frac{n}{n+\dz_1},1]$
and $\Phi$ is of upper type 1. When $\fai(x,t):=\Phi(t)$ for any $x\in\rn$ and $t\in[0,\fz)$,
the equivalences of the spaces $H_{\fai,\,L_D}(\boz)$, $H_{\fai,\,L_D,\,\mathrm{max}}(\boz)$,
and $H_{\fai,\,r}(\boz)$ in Theorem \ref{t4.1} were obtained in
\cite[Theorem 1.9]{yys3}. The equivalence of the spaces
$H_{\fai,\,L_D,\,\mathrm{rad}}(\boz)$ and $H_{\fai,\,r}(\boz)$ is new in this case. Moreover, when $\boz$ is bounded, the equivalences of the spaces $h_{\fai,\,L_D,\,\mathrm{max}}(\boz)$, $h_{\fai,\,L_D,\,\mathrm{rad}}(\boz)$
and $H_{\fai,\,r}(\boz)$ are new even when $\fai(x,t):=\Phi(t)$
for any $x\in\rn$ and $t\in[0,\fz)$.

\item[(iii)] Theorem \ref{t4.1} is new even when $\fai$ is as in \eqref{1.8}
or \eqref{1.9}.
\end{itemize}
\end{remark}

To show Theorem \ref{t4.1}, we need the following atomic characterization of the space
$H_\fai(\rn)$ established by Ky \cite{k14}.

\begin{definition}\label{d4.3}
Let $\fai$ be as in Definition \ref{d1.2} with $\cx$ being replaced by $\rn$.
\begin{itemize}
  \item[(I)] For any ball $B\subset\rn$,
  the \emph{space $L^q_\fai(B)$} with
$q\in[1,\fz]$ is defined to be the set of all measurable functions
$f$ on $\rn$ supported in $B$ such that
\begin{equation*}
\|f\|_{L^q_{\fai}(B)}:=
\begin{cases}\dsup_{t\in (0,\fz)}
\lf[\frac{1}
{\fai(B,t)}\dint_{\rn}|f(x)|^q\fai(x,t)\,dx\r]^{1/q}<\fz,& q\in [1,\fz),\\
\|f\|_{L^{\fz}(B)}<\fz,&q=\fz.
\end{cases}
\end{equation*}
\item[(II)] A triplet $(\fai,\,q,\,s)$ is said to be
\emph{admissible} if $q\in(q(\fai),\fz]$ and $s\in\zz_+$ satisfies that $s\ge\lfz n[\frac{q(\fai)}{i(\fai)}-1]\rfz$. A measurable function $a$ on $\rn$ is called a \emph{$(\fai,\,q,\,s)$-atom} if there exists a ball
$B\subset\rn$ such that
\begin{enumerate}
  \item[(i)] $\supp (a)\subset B$;
  \item[(ii)] $\|a\|_{L^q_{\fai}(B)}\le\|\mathbf{1}_B\|_{L^\fai(\rn)}^{-1}$;
  \item[(iii)] $\int_{\rn}a(x)x^{\az}\,dx=0$ for any
$\az\in\zz_+^n$ with $|\az|\le s$.
\end{enumerate}

  \item[(III)] The \emph{atomic Musielak-Orlicz Hardy space}
$H^{\fai,\,q,\,s}(\rn)$ is defined to be the set of all
$f\in\cs'(\rn)$ satisfying that $f=\sum_jb_j$ in $\cs'(\rn)$,
where $\{b_j\}_j$ is a sequence of multiples of
$(\fai,\,q,\,s)$-atoms with $\supp (b_j)\subset B_j$ and
$\sum_j\fai(B_j,\|b_j\|_{L^q_{\fai}(B_j)})<\fz.$
Moreover, letting
\begin{align*}
&\blz_q(\{b_j\}_j):= \inf\lf\{\lz\in(0,\fz):\ \ \sum_j\fai\lf(B_j,
\frac{\|b_j\|_{L^q_{\fai}(B_j)}}{\lz}\r)\le1\r\},
\end{align*}
the \emph{quasi-norm} $\|f\|_{H^{\fai,\,q,\,s}(\rn)}$ of $f\in H^{\fai,\,q,\,s}(\rn)$ in
$H^{\fai,\,q,\,s}(\rn)$ is defined
by setting
$$\|f\|_{H^{\fai,\,q,\,s}(\rn)}:=\inf\lf\{\blz_q(\{b_j\}_j)\r\},$$
where the infimum is taken over all the
decompositions of $f$ as above.
\end{itemize}
\end{definition}

The following lemma is just \cite[Theorem 3.1]{k14}.

\begin{lemma}\label{l4.1}
Let $\fai$ be as in Definition \ref{d1.2} with $\cx$ being replaced by $\rn$ and $(\fai,\,q,\,s)$
admissible. Then $H_\fai(\rn)=H^{\fai,\,q,\,s}(\rn)$ with
equivalent norms.
\end{lemma}

Now we prove Theorem \ref{t4.1} by using Theorems \ref{t1.1} and \ref{t1.2} and Lemma \ref{l4.1}.

\begin{proof}[Proof of Theorem \ref{t4.1}]
Via taking $\cx:=\boz$ and $L:=L_D$ in Theorem \ref{t1.1}, we know that
the spaces $H_{\fai,\,L_D}(\boz)$, $H^{M,\,q}_{\fai,\,L_D,\,\mathrm{at}}(\boz)$, $H_{\fai,\,L_D,\,\mathrm{max}}(\boz)$
and $H_{\fai,\,L_D,\,\mathrm{rad}}(\boz)$ coincide with equivalent quasi-norms.
Moreover, when $\boz$ is bounded, by Theorem \ref{t1.2}, we find that the spaces
$H_{\fai,\,L_D}(\boz)$, $H^{M,\,q}_{\fai,\,L_D,\,\mathrm{at}}(\boz)$, $H_{\fai,\,L_D,\,\mathrm{max}}(\boz)$,
$H_{\fai,\,L_D,\,\mathrm{rad}}(\boz)$, $h_{\fai,\,L_D,\,\mathrm{max}}(\boz)$ and
$h_{\fai,\,L_D,\,\mathrm{rad}}(\boz)$ coincide with equivalent quasi-norms.
To finish the proof of Theorem \ref{t4.1}, we only need to show that
$H_{\fai,\,L_D,\,\mathrm{rad}}(\boz)$ and $H_{\fai,\,r}(\boz)$ coincide with equivalent quasi-norms.

We first prove that
\begin{equation}\label{4.5}
\lf[H_{\fai,\,r}(\boz)\cap L^2 (\boz)\r]\subset\lf[H_{\fai,\,L_D,\,\mathrm{rad}}(\boz)\cap L^2(\boz)\r].
\end{equation}
Let $f\in [H_{\fai,\,r}(\boz)\cap L^2(\boz)]$. From the definition of
$H_{\fai,\,r}(\boz)$, we deduce that there exists $\wz{f}\in
H_{\fai}(\rn)$ such that $\wz{f}\big|_{\boz}=f$ and
\begin{equation}\label{4.6}
\lf\|\wz{f}\r\|_{H_{\fai}(\rn)}\ls\|f\|_{H_{\fai,\,r}(\boz)}.
\end{equation}
Using the assumptions that $nq(\fai)/i(\fai)<n+\dz_1$ and \eqref{4.4},
together with an argument similar to the proof of \cite[(3.2)]{yys3}, we conclude that, for any
constant multiple of a $(\fai,\,\fz,\,0)$-atom, $b$, supported in a
ball $B_0:=B(x_0,r_0)$, with $x_0\in\boz$ and $r_0\in(0,\fz)$, and $\lz\in(0,\fz)$,
\begin{equation}\label{4.7}
\int_{\boz}\fai\lf(x,
\frac{(b)^+_{L_D}(x)}{\lz}\r)\,dx\ls\fai\lf(B_0,
\frac{\|b\|_{L^\fz(B_0)}}{\lz}\r).
\end{equation}
For any $\wz{f}\in H_{\fai}(\rn)$, from Lemma \ref{l4.1},
it follows that there exists a sequence $\{b_i\}_{i=1}^\fz$ of constant multiples of $(\fai,\,\fz,\,0)$-atoms, with the constants depending on $i$, such
that $\wz{f}=\sum_{i=1}^\fz b_i$ in $\cs'(\rn)$ and
\begin{equation}\label{4.8}
\Lambda_\fz(\{b_i\}_{i=1}^\fz)\sim\lf\|\wz f\r\|_{H_{\fai}(\rn)}.
\end{equation}
Moreover, by the proof of \cite[Theorem 3.1]{k14},
we conclude that the supports of $\{b_i\}_{i=1}^\fz$ are of
finite intersection property. From this, $f\in L^2 (\boz)$,
$\wz{f}=\sum_{i=1}^\fz b_i$ in $\cs'(\rn)$ and $\wz{f}\big|_{\boz}=f$, we
deduce that $f=\sum_{i=1}^\fz b_i$ almost everywhere on $\boz$, which
further implies that
\begin{align*}
\int_{\boz}K_{t^2} (x,y)f(y)\,dy=\sum_{i=1}^\fz
\int_{\boz}K_{t^2} (x,y)b_i (y)\,dy.
\end{align*}
By this, we find that, for any $x\in\boz$,
$(f)^+_{L_D}(x)\le\sum_{i=1}^\fz(b_i)^+_{L_D}(x),$ which, combined with \eqref{4.7}, implies that,
for any $\lz\in(0,\fz)$,
\begin{align*}
\int_{\boz}\fai\lf(x,\frac{(f)^+_{L_D}(x)}{\lz}\r)\,dx\ls\sum_{i=1}^\fz
\int_{\boz}\fai\lf(x,\frac{(b_i)^+_{L_D}(x)}{\lz}\r)\,dx\ls\sum_{i=1}^\fz
\fai\lf(B_i,\frac{\|b_i\|_{L^\fz(B_i)}}{\lz}\r),
\end{align*}
where, for each $i$, $\supp(b_i)\subset B_i$.
From this, \eqref{4.6} and \eqref{4.8}, it follows that
$$\|f\|_{H_{\fai,\,L_D,\,\mathrm{rad}}(\boz)}\ls\Lambda_\fz(\{b_i\}_{i=1}^\fz)\sim\lf\|\wz
f\r\|_{H_{\fai}(\rn)}\ls\|f\|_{H_{\fai,\,r}(\boz)},$$
which implies that \eqref{4.5} holds true.

Now we show that
\begin{equation}\label{4.9}
\lf[H_{\fai,\,L_D,\,\mathrm{rad}}(\boz)\cap L^2(\boz)\r]\subset\lf[H_{\fai,\,r}(\boz)\cap L^2(\boz)\r].
\end{equation}
Let $f\in [H_{\fai,\,L_D,\,\mathrm{rad}}(\boz)\cap L^2 (\boz)]$.  We first assume that
$\boz$ is bounded. Then there exist
$\{\lz_0\}\cup\{\lz_j\}_{j=1}^\fz\subset\cc$, a $(\fai,\fz)$-single-atom $\az_0$ and
a sequence $\{\az_j\}_{j=1}^\fz$
of $(\fai,\,\fz,\,M)_L$-atoms associated, respectively, with the balls $\{B_{\boz,\,j}\}_{j=1}^\fz$ such that
\begin{equation}\label{4.10}
f=\lz_0\az_0+\sum_{j=1}^\fz\lz_j\az_j\ \text{in}\ L^2(\boz) \ \text{and} \ \
\|f\|_{H_{\fai,\,L_D,\,\mathrm{rad}}(\boz)}\sim\blz\lf(\{\lz_0\az_0\}\bigcup\{\lz_j\az_j\}_{j=1}^\fz\r),
\end{equation}
where, for any $j\in\nn$, $B_{\boz,\,j}:=B_j\cap\boz$ and
$B_j:=B(x_j,r_j)$ with $x_j\in\rn$ and $r_j\in(0,\fz)$ is
a ball in $\rn$.

Now we show that there exists $\wz{f}\in H_\fai(\rn)$ such that $\wz{f}\big|_{\boz}=f$ and
\begin{equation*}
\lf\|\wz{f}\r\|_{H_{\fai}(\rn)}\ls\blz\lf(\{\lz_0\az_0\}\bigcup\{\lz_j\az_j\}_{j=1}^\fz\r)\sim
\|f\|_{H_{\fai,\,L_D,\,\mathrm{rad}}(\boz)},
\end{equation*}
which further implies that $f\in H_{\fai,\,r}(\boz)$ and
$$\|f\|_{H_{\fai,\,r}(\boz)}\ls\|f\|_{H_{\fai,\,L_D,\,\mathrm{rad}}(\boz)},$$
and hence \eqref{4.9} holds true.

For the $(\fai,\fz)$-single-atom $\az_0$, by the Whitney decomposition of $\boz$ (see, for example, \cite[p.\,15, Lemma 2]{st93}), we know that there exists a sequence $\{B_{0,\,k}\}_{k=1}^\fz$ of balls
such that $\cup_{k=1}^\fz B_{0,\,k}=\boz$ and, for any $k\in\nn$, $2B_{0,\,k}\subset\boz$,
$4B_{0,\,k}\cap\partial\boz\neq\emptyset$ and $\sum_{k=1}^\fz\mathbf{1}_{B_{0,\,k}}\ls1$.
For any $k\in\nn$, let
$$\lz_{0,\,k}:=\frac{\|\mathbf{1}_{B_{0,\,k}}\|_{L^\fai(\rn)}}{\|\mathbf{1}_{\boz}\|_{L^\fai(\rn)}}\ \ \ \text{and} \ \ \
\az_{0,\,k}:=\frac{1}{\lz_{0,\,k}}\frac{\az_0\mathbf{1}_{B_{0,\,k}}}{\sum_{j=1}^\fz\mathbf{1}_{B_{0,\,j}}}.
$$
Then $\az_0=\sum_{k=1}^\fz\lz_{0,\,k}\az_{0,\,k}$ and,
for any $k\in\nn$,
$\|\az_{0,\,k}\|_{L^\fz(\rn)}\le\|\mathbf{1}_{B_{0,\,k}}\|_{L^\fai(\rn)}^{-1}$.
From the assumption that $\boz^\complement$ is unbounded and
a geometric property of strongly Lipschitz domains proved in \cite[p.\,183]{ar03}, it follows that,
for any $k\in\nn$, there exists a ball $\wz{B}_{0,\,k}\subset\boz^\complement$ such that
$$r_{\wz{B}_{0,\,k}}\sim r_{B_{0,\,k}}\ \ \text{and} \ \ \lf|x_{B_{0,\,k}}-x_{\wz{B}_{0,\,k}}\r|\ls r_{B_{0,\,k}},
$$
which further implies that there exists a positive constant $\wz{C}_3$
such that $(B_{0,\,k}\cup \wz{B}_{0,\,k})\subset\wz{C}_3B_{0,\,k}$.
For any $k\in\nn$, let
$$\wz{\lz}_{0,\,k}:=\lz_{0,\,k}\ \ \ \text{and}\ \ \ \wz{\az}_{0,\,k}:=\az_{0,\,k}-
\lf[\int_{B_{0,\,k}}\az_{0,\,k}(x)\,dx\r]\frac{\mathbf{1}_{\wz{B}_{0,\,k}}}{|\wz{B}_{0,\,k}|}.
$$
Then $\int_{\rn}\wz{\az}_{0,\,k}(x)\,dx=0$, $\supp(\wz{\az}_{0,\,k})\subset\wz{C}_3B_{0,\,k}$ and
\begin{align*}
\lf\|\wz{\az}_{0,\,k}\r\|_{L^\fz(\rn)}\le\|\az_{0,\,k}\|_{L^\fz(\rn)}\lf[1+\frac{|B_{0,\,k}|}{|\wz{B}_{0,\,k}|}\r]
\ls\|\az_{0,\,k}\|_{L^\fz(\rn)}\ls\lf\|\mathbf{1}_{B_{0,\,k}}\r\|_{L^\fai(\rn)}^{-1}\ls
\lf\|\mathbf{1}_{\wz{C}_3B_{0,\,k}}\r\|_{L^\fai(\rn)}^{-1}.
\end{align*}
Thus, $\wz{\az}_{0,\,k}$ is a harmless constant multiple of a $(\fai,\fz,0)$-atom and $\wz{\az}_{0,\,k}\big|_\boz=\az_{0,\,k}$.
Let $\wz{\az}_0:=\sum_{k=1}^\fz\wz{\lz}_{0,\,k}\wz{\az}_{0,\,k}$. Then $\wz{\az}_0\in H^{\fai,\,\fz,\,0}(\rn)$,
$\wz{\az}_{0}\big|_\boz=\az_{0}$ and
\begin{align}\label{4.11}
\sum_{k=1}^\fz\fai\lf(\wz{C}_3B_{0,\,k},\lf\|\lz_0\wz{\lz}_{0,\,k}\wz{\az}_{0,\,k}\r\|_{L^\fz(\rn)}\r)
\ls\sum_{k=1}^\fz\fai\lf(B_{0,\,k},\frac{|\lz_0|}{\|\mathbf{1}_\boz\|_{L^\fai(\rn)}}\r)
\ls\fai\lf(\boz,\frac{|\lz_0|}{\|\mathbf{1}_\boz\|_{L^\fai(\rn)}}\r).
\end{align}

Now we assume that $\az$ is a $(\fai,\,\fz,\,M)_L$-atom associated with the ball $B_\boz:=(B\cap\boz)$.
We deal with $\az$ by considering the following three cases.

\emph{Case 1)} $B\subset\boz$ but $4B\cap\partial\boz\neq\emptyset$. In this case, similarly to the proof
for $(\fai,\fz)$-single-atoms,
we know that there exist $\wz{\az}\in H^{\fai,\,\fz,\,0}(\rn)$, $\{\wz{\lz}_k\}_{k=1}^\fz\subset\cc$ and
a sequence $\{\wz{\az}_{k}\}_{k=1}^\fz$ of $(\fai,\,\fz,\,0)$-atoms
supported, respectively, in the balls $\{\wz{B}_{k}\}_{k=1}^\fz$ such that
$\wz{\az}=\sum_{k=1}^\fz\wz{\lz}_{k}\wz{\az}_{k}$,
$\wz{\az}\big|_\boz=\az$ and, for any $\lz\in(0,\fz)$,
\begin{align*}
\sum_{k=1}^\fz\fai\lf(\wz{B}_{k},\lf\|\lz\wz{\lz}_{k}\wz{\az}_{k}\r\|_{L^\fz(\rn)}\r)
\ls\fai\lf(B\cap\boz,\frac{\lz}{\|\mathbf{1}_B\|_{L^\fai(\rn)}}\r).
\end{align*}

\emph{Case 2)}  $4B\subset\boz$. In this case, by the definition of $(\fai,\,\fz,\,M)_L$-atoms, we conclude that
there exists $b\in \cd(L_D)$ such that $\az=L_D b$ and $\supp(b)\subset B$. Take $\psi\in C^\fz_c(\boz)$ satisfying
$\psi\equiv1$ on $2B$. Then, from the condition $b\in \cd(L_D)$ and the definition of $L_D$, we deduce that
$$\int_{B}\az(x)\,dx=\int_{\boz}\az(x)\psi(x)\,dx=-\int_{\boz}A(x)\nabla b(x)\cdot\ol{\nabla
\psi(x)}\,dx=0,
$$
which, combined with $\supp(\az)\subset B$ and $\|\az\|_{L^\fz(B)}\le\|\mathbf{1}_B\|^{-1}_{L^\fai(\rn)}$, implies that
$\az$ is a $(\fai,\,\fz,\,0)$-atom.

\emph{Case 3)} $B\cap\boz\neq\emptyset$. In this case, for any $x\in B\cap\boz$, let $B_x:=B(x,\rho(x)/2)$,
where $\rho(x):=\dist(x,B\cap\partial\boz)$. Then, by the Besicovitch covering theorem (see, for example,
\cite[p.\,44]{st93}), we find that there exists
a sequence $\{B_{x_i}\}_{i=1}^\fz$ of balls such that $(B\cap\boz)\subset\cup_{i=1}^\fz B_{x_i}$ and
$\sum_{i=1}^\fz\mathbf{1}_{B_{x_i}}\ls1$.

For any $i\in\nn$, let
$$\lz_{i}:=\frac{\|\mathbf{1}_{B_{x_i}\cap\boz}\|_{L^\fai(\rn)}}{\|\mathbf{1}_{B\cap\boz}\|_{L^\fai(\rn)}}\ \ \ \text{and} \ \ \
\az_{i}:=\frac{1}{\lz_{i}}\frac{\az\mathbf{1}_{B_{x_i}}}{\sum_{j=1}^\fz\mathbf{1}_{B_{x_j}}}.
$$
Then $\az=\sum_{i=1}^\fz\lz_{i}\az_{i}$ and, for any $i\in\nn$,
$\|\az_{i}\|_{L^\fz(\rn)}\le\|\mathbf{1}_{B_{x_i}\cap\boz}\|_{L^\fai(\rn)}^{-1}$.
By an argument similar to that used in the case of $(\fai,\fz)$-single-atoms, we know that,
for any $i\in\nn$, there exists a ball $\wz{B}_{x_i}\subset\boz^\complement$ such that
$$r_{\wz{B}_{x_i}}\sim r_{B_{x_i}}\ \ \text{and} \ \ \lf|x_{{B}_{x_i}}-x_{\wz{B}_{x_i}}\r|\ls r_{B_{x_i}},
$$
which further implies that there exists a positive constant $\wz{C}_4$
such that $B_{x_i}\cup \wz{B}_{x_i}\subset\wz{C}_4B_{x_i}$.
For any $i\in\nn$, let
$$\wz{\lz}_{i}:=\lz_{i}\ \ \ \text{and}\ \ \ \wz{\az}_{i}:=\az_{i}-
\lf[\int_{B_{x_i}\cap\boz}\az_{i}(x)\,dx\r]\frac{\mathbf{1}_{\wz{B}_{x_i}}}{|\wz{B}_{x_i}|}.
$$
Then $\int_{\rn}\wz{\az}_{i}(x)\,dx=0$, $\supp(\wz{\az}_{i})\subset\wz{C}_4B_{x_i}$ and
\begin{align*}
\lf\|\wz{\az}_{i}\r\|_{L^\fz(\rn)}\le\|\az_{i}\|_{L^\fz(\rn)}\lf[1+\frac{|B_{x_i}|}{|\wz{B}_{x_i}|}\r]
\ls\|\az_{i}\|_{L^\fz(\rn)}
\ls\|\mathbf{1}_{B_{x_i}\cap\boz}\|_{L^\fai(\rn)}^{-1}\ls
\lf\|\mathbf{1}_{\wz{C}_4B_{x_i}}\r\|_{L^\fai(\rn)}^{-1}.
\end{align*}
Thus, $\wz{\az}_{i}$ is a harmless constant multiple of
a $(\fai,\fz,0)$-atom and $\wz{\az}_{i}\big|_\boz=\az_{i}$.
Let $\wz{\az}:=\sum_{i=1}^\fz\wz{\lz}_{i}\wz{\az}_{i}$. Then $\wz{\az}\in H^{\fai,\,\fz,\,0}(\rn)$,
$\wz{\az}\big|_\boz=\az$ and, for any $\lz\in(0,\fz)$,
\begin{align*}
\sum_{i=1}^\fz\fai\lf(\wz{C}_4B_{x_i},\lf\|\lz\wz{\lz}_{i}\wz{\az}_{i}\r\|_{L^\fz(\rn)}\r)
\ls\sum_{i=1}^\fz\fai\lf(B_{x_i},\frac{\lz}{\|\mathbf{1}_{B\cap\boz}\|_{L^\fai(\rn)}}\r)
\ls\fai\lf(B\cap\boz,\frac{\lz}{\|\mathbf{1}_{B\cap\boz}\|_{L^\fai(\rn)}}\r).
\end{align*}

By the proofs in Cases 1) to 3), we know that, for any $j\in\nn$ and $(\fai,\,\fz,\,M)_L$-atom $\az_j$
associated with the ball $B_{\boz,\,j}=(B_j\cap\boz)$, there exist $\{\wz{\lz}_{j,\,k}\}_{k}\subset\cc$ and
a sequence $\{\wz{\az}_{j,\,k}\}_k$ of harmless constant multiples
of $(\fai,\,\fz,\,0)$-atoms supported, respectively, in the balls
$\{\wz{B}_{j,\,k}\}_k$ such that $\wz{\az}_j:=\sum_k\wz{\lz}_{j,\,k}\wz{\az}_{j,\,k}\big|_\boz=\az_j$
and, for any $\lz\in(0,\fz)$,
\begin{align}\label{4.12}
\sum_{k}\fai\lf(\wz{B}_{j,\,k},\lf\|\lz\wz{\lz}_{j,\,k}\wz{\az}_{j,\,k}\r\|_{L^\fz(\rn)}\r)
\ls\fai\lf(B_j\cap\boz,\frac{\lz}{\|\mathbf{1}_{B_j\cap\boz}\|_{L^\fai(\rn)}}\r).
\end{align}

Let
$$\wz{f}:=\sum_{k=1}^\fz\lz_0\wz{\lz}_{0,\,k}\wz{\az}_{0,\,k}+\sum_{j=1}^\fz\sum_k\lz_j\wz{\lz}_{j,\,k}
\wz{\az}_{j,\,k}.$$
Then $\wz{f}\big|_\boz=f$. Moreover, from \eqref{4.11} and \eqref{4.12}, we deduce that,
for any $\lz\in(0,\fz)$,
\begin{align*}
&\sum_{k=1}^\fz\fai\lf(\wz{C}_3B_{0,\,k},\frac{\|\lz_0\wz{\lz}_{0,\,k}\wz{\az}_{0,\,k}\|_{L^\fz(\rn)}}{\lz}\r)
+\sum_{j=1}^\fz\sum_{k}\fai\lf(\wz{B}_{j,\,k},\frac{\|\lz_j\wz{\lz}_{j,\,k}\wz{\az}_{j,\,k}\|_{L^\fz(\rn)}}{\lz}\r)\\
&\hs\ls\fai\lf(\boz,\frac{|\lz_0|}{\lz\|\mathbf{1}_\boz\|_{L^\fai(\rn)}}\r)+
\sum_{j=1}^\fz\fai\lf(B_j\cap\boz,\frac{|\lz_j|}{\lz\|\mathbf{1}_{B_j\cap\boz}\|_{L^\fai(\rn)}}\r),
\end{align*}
which, combined with \eqref{4.10} and Lemma \ref{l4.1}, further implies that
\begin{align}\label{4.13}
\|f\|_{H_{\fai,\,r}(\boz)}\ls\lf\|\wz{f}\r\|_{H_{\fai}(\rn)}\ls\|f\|_{H_{\fai,\,L_D,\,\mathrm{rad}}(\boz)}.
\end{align}
Thus, \eqref{4.9} holds true. In the case that $\boz$ is unbounded, repeating the proofs of \eqref{4.12}
and \eqref{4.13}, we know that \eqref{4.9} also holds true in this case.

By \eqref{4.5} and \eqref{4.9}, we conclude that
\begin{equation*}
\lf[H_{\fai,\,L_D,\,\mathrm{rad}}(\boz)\cap L^2(\boz)\r]=\lf[H_{\fai,\,r}(\boz)\cap L^2(\boz)\r],
\end{equation*}
which, together with the fact that the spaces
$H_{\fai,\,L_D,\,\mathrm{rad}}(\boz)\cap L^2(\boz)$ and
$H_{\fai,\,r}(\boz)\cap L^2(\boz)$ are dense, respectively, in the spaces
$H_{\fai,\,L_D,\,\mathrm{rad}}(\boz)$ and $H_{\fai,\,r}(\boz)$, and a
density argument, implies that the spaces $H_{\fai,\,L_D,\,\mathrm{rad}}(\boz)$,
and $H_{\fai,\,r}(\boz)$ coincide with equivalent quasi-norms.
This finishes the proof of Theorem \ref{t4.1}.
\end{proof}

Let $\boz$ be a strongly Lipschitz domain in $\rn$ and $L_N$ a second-order
elliptic operator with the real bounded coefficients and the Neumann boundary condition on $\boz$.
Denote by $\{K_t\}_{t>0}$ the kernels of the semigroup $\{e^{-tL_N}\}_{t>0}$.
From \cite[Theorem 3.2.9]{d89} (see also \cite{at01a}), it follows that there exist  positive constants $C$ and $c$ such that, for any $t\in(0,\fz)$ and $x,\,y\in\boz$,
\begin{equation}\label{4.14}
|K_t(x,y)|\le\frac{C}{|B_\boz(x,\sqrt{t})|}\exp\lf\{-\frac{|x-y|^2}{ct}\r\},
\end{equation}
where $B_\boz(x,\sqrt{t}):=B(x,\sqrt{t})\cap\boz$. This implies that $L_N$ satisfies Assumption \ref{a2}. Moreover, by \cite[Propsition 5]{at01a}, we know that
there exist positive constants $C$, $c$ and $\dz_2\in(0,1]$
such that, for any $t\in(0,[\diam(\boz)]^2)$
and $x,\,y_1,\,y_2\in\boz$ with $|y_1-y_2|<\sqrt{t}/2$,
\begin{equation}\label{4.15}
|K_t(x,y_1)-K_t(x,y_2)|\le\frac{C}{|B_\boz(x,\sqrt{t})|}
\lf[\frac{|y_1-y_2|}{\sqrt{t}}\r]^{\dz_2}
\exp\lf\{-\frac{|x-y|^2}{ct}\r\}
\end{equation}
(see also \cite{ar03}). Furthermore, it is well known that $L_N$ satisfies Assumption \ref{a4}
(see, for example, \cite{ar03,at01a,o04}).

Assume further that $L_N$ has symmetrical coefficients. Then $L_N$ satisfies Assumptions \ref{a1} and \ref{a2}.

Let
$$\wz{H}_{\fai,\,z}(\boz):=\begin{cases}H_{\fai,\,z}(\boz)\ \ &\text{if}\ \boz\ \text{is unbounded},\\
H_{\fai,\,z}(\boz)+L^\fz(\boz)\ \ &\text{if}\ \boz\ \text{is bounded},
\end{cases}
$$
where, for any $f\in H_{\fai,\,z}(\boz)+L^\fz(\boz)$,
$$\|f\|_{H_{\fai,\,z}(\boz)+L^\fz(\boz)}:=\inf\lf\{\|f_1\|_{H_{\fai,\,z}(\boz)}+\|f_2\|_{L^\fz(\boz)}:\ f=f_1+f_2\r\}.$$

For $\wz{H}_{\fai,\,z}(\boz)$, we have the following equivalent characterizations via the (local)
Musielak-Orlicz-Hardy spaces on $\boz$ associated with $L_N$.

\begin{theorem}\label{t4.2}
Let $\boz$ be a strongly Lipschitz domain in $\rn$ and $\fai$ as in Definition \ref{d1.2}.
Assume that $L_N$ is a second-order self-adjoint elliptic operator on $\boz$, with
the Neumann boundary condition, satisfying \eqref{4.14} and \eqref{4.15}. Let $\dz_2\in(0,1]$, $r(\fai)$, $I(\fai)$, $q(\fai)$
and $i(\fai)$ be, respectively, as in \eqref{4.15}, \eqref{1.7}, \eqref{1.4}, \eqref{1.6} and \eqref{1.5}.
If $nq(\fai)/i(\fai)<n+\dz_2$, then, for any $q\in([r(\fai)]'I(\fai),\fz]\cap(1,\fz]$ and
$M\in\nn\cap(\frac{nq(\fai)}{2i(\fai)},\fz)$,
the Musielak-Orlicz-Hardy spaces $H_{\fai,\,L_N}(\boz)$, $H^{M,\,q}_{\fai,\,L_N,\,\mathrm{at}}(\boz)$,
$H_{\fai,\,L_N,\,\mathrm{max}}(\boz)$, $H_{\fai,\,L_N,\,\mathrm{rad}}(\boz)$, $H^q_{\fai,\,\mathrm{at}}(\boz)$
and $\wz{H}_{\fai,\,z}(\boz)$ coincide with equivalent quasi-norms. In particular, if $\boz$
is bounded, then the spaces
$$H_{\fai,\,L_N}(\boz),\ H^{M,\,q}_{\fai,\,L_N,\,\mathrm{at}}(\boz),\
H_{\fai,\,L_N,\,\mathrm{max}}(\boz),\ H_{\fai,\,L_N,\,\mathrm{rad}}(\boz),\ H^q_{\fai,\,\mathrm{at}}(\boz),\ h_{\fai,\,L_N,\,\mathrm{max}}(\boz),\
h_{\fai,\,L_N,\,\mathrm{rad}}(\boz)$$
and $\wz{H}_{\fai,\,z}(\boz)$
coincide with equivalent quasi-norms.
\end{theorem}

When $\fai(x,t):=t^p$, with $p\in(0,1]$, for any $x\in\rn$ and $t\in[0,\fz)$,
we denote the spaces $$H_{\fai,\,L_N,\,\mathrm{max}}(\boz),\ H_{\fai,\,L_N,\,\mathrm{rad}}(\boz),\
\wz{H}_{\fai,\,z}(\boz),\ h_{\fai,\,L_N,\,\mathrm{max}}(\boz)\ \mathrm{and}\
h_{\fai,\,L_N,\,\mathrm{rad}}(\boz)$$
simply, respectively,  by $H^p_{L_N,\,\mathrm{max}}(\boz)$,
$H^p_{L_N,\,\mathrm{rad}}(\boz)$, $\wz{H}^p_{z}(\boz)$, $h^p_{L_N,\,\mathrm{max}}(\boz)$ and
$h^p_{L_N,\,\mathrm{rad}}(\boz)$.

As a simple corollary of Theorem \ref{t4.2}, we have the following conclusion.

\begin{corollary}\label{c4.2}
Let $\boz$ be a strongly Lipschitz domain in $\rn$.
Assume that $L_N$ is a second-order self-adjoint elliptic operator on $\boz$, with
the Neumann boundary condition, satisfying \eqref{4.14} and \eqref{4.15}. Let $\dz_2\in(0,1]$
and $p\in(\frac{n}{n+\dz_2},1]$. Then, for any $q\in(1,\fz]$ and $M\in\nn\cap(\frac{n}{2p},\fz)$,
the Hardy spaces $H^p_{L_N}(\boz)$, $H^{p,\,M,\,q}_{L_N,\,\mathrm{at}}(\boz)$,
$H^p_{L_N,\,\mathrm{max}}(\boz)$, $H^p_{L_N,\,\mathrm{rad}}(\boz)$, $H^{p,\,q}_{\mathrm{at}}(\boz)$
and $\wz{H}^p_{z}(\boz)$ coincide with equivalent quasi-norms. In particular, if $\boz$
is bounded, then the spaces $H^p_{L_N}(\boz)$, $H^{p,\,M,\,q}_{L_N,\,\mathrm{at}}(\boz)$,
$H^p_{L_N,\,\mathrm{max}}(\boz)$, $H^p_{L_N,\,\mathrm{rad}}(\boz)$, $H^{p,\,q}_{\mathrm{at}}(\boz)$,
$h^p_{L_N,\,\mathrm{max}}(\boz)$, $h^p_{L_N,\,\mathrm{rad}}(\boz)$
and $\wz{H}^p_{z}(\boz)$ coincide with equivalent quasi-norms.
\end{corollary}

\begin{remark}\label{r4.2}
\begin{itemize}
  \item[(i)] When $p=1$ and $\boz$ is unbounded,
  the equivalences of the spaces $H^p_{L_N,\,\mathrm{max}}(\boz)$,
  $H^p_{L_N,\,\mathrm{rad}}(\boz)$, $H^p_{L_N}(\boz)$, $H^{p,\,q}_{\mathrm{at}}(\boz)$ and
  $\wz{H}^p_{z}(\boz)$ in Corollary \ref{c4.2} were established in
  \cite[Theorem 1, Proposition 5 and Theorem 20]{ar03}. The other cases of
  Corollary \ref{c4.2} even in this case are new.
  Thus, Corollary \ref{c4.2} improves the results obtained
  in \cite{ar03} by removing the assumption that $\boz$ is unbounded even when $p=1$.

  \item[(ii)] When $p\in(\frac{n}{n+\dz_2},1]$, the equivalences of the spaces $H^p_{L_N,\,\mathrm{max}}(\boz)$,
  $H^p_{L_N,\,\mathrm{rad}}(\boz)$, $H^p_{L_N}(\boz)$ and $H^q_{\fai,\,\mathrm{at}}(\boz)$ in
  Corollary \ref{c4.2} were obtained in \cite[Theorems 1.4 and 4.1]{bdl1}. It is worth pointing out
  that the equivalence of the spaces $H^p_{L_N}(\boz)$ and $\wz{H}^p_{z}(\boz)$ when $\boz$ is unbounded
  and the equivalences of the spaces $h^p_{L_N,\,\mathrm{max}}(\boz)$, $h^p_{L_N,\,\mathrm{rad}}(\boz)$
and $\wz{H}^p_{z}(\boz)$ when $\boz$ is bounded are new in this case. Thus, Corollary \ref{c4.2} improves
\cite[Theorems 1.4 and 4.1]{bdl1} in this case.
\end{itemize}
\end{remark}

\begin{remark}\label{r4.3}
\begin{itemize}
\item[(i)] Let $\Phi$ be an Orlicz function satisfying that $i(\Phi)\in(\frac{n}{n+\dz_2},1]$
and $\Phi$ is of upper type 1. When $\fai(x,t):=\Phi(t)$ for any $x\in\rn$ and $t\in[0,\fz)$ and $\boz$ is unbounded,
the equivalences of the spaces $H_{\fai,\,L_N}(\boz)$, $H_{\fai,\,L_N,\,\mathrm{max}}(\boz)$, $H_{\fai,\,L_N,\,\mathrm{rad}}(\boz)$,
and $H_{\fai,\,r}(\boz)$ in Theorem \ref{t4.1} were obtained in
  \cite[Theorems 1.10 and 1.12]{yys2}. Theorem \ref{t4.2} improves the results in
  \cite[Theorems 1.10 and 1.12]{yys2} by removing the assumption that $\boz$ is unbounded.
  Moreover, when $\boz$ is bounded, the  equivalences of the spaces $h_{\fai,\,L_N,\,\mathrm{max}}(\boz)$, $h_{\fai,\,L_N,\,\mathrm{rad}}(\boz)$ and $\wz{H}_{\fai,\,z}(\boz)$ are new even in this case.

\item[(ii)] Theorem \ref{t4.2} is new even when $\fai$ is as in \eqref{1.8} or \eqref{1.9}.
\end{itemize}
\end{remark}

Now we prove Theorem \ref{t4.2} by using Theorems \ref{t1.1}, \ref{t1.2} and \ref{t1.3} and Lemma \ref{l4.1}.

\begin{proof}[Proof of Theorem \ref{t4.2}]
Applying Theorems \ref{t1.1} and \ref{t1.3} to $\cx:=\boz$ and $L:=L_N$, we conclude that
the spaces $H_{\fai,\,L_N}(\boz)$, $H^{M,\,q}_{\fai,\,L_N,\,\mathrm{at}}(\boz)$, $H_{\fai,\,L_N,\,\mathrm{max}}(\boz)$,
$H_{\fai,\,L_N,\,\mathrm{rad}}(\boz)$ and $H^q_{\fai,\,\mathrm{at}}(\boz)$ coincide with equivalent quasi-norms.
Moreover, when $\boz$ is bounded, by Theorem \ref{t1.2}, we know that the spaces
$H_{\fai,\,L_N}(\boz)$, $H^{M,\,q}_{\fai,\,L_N,\,\mathrm{at}}(\boz)$, $H_{\fai,\,L_N,\,\mathrm{max}}(\boz)$,
$H_{\fai,\,L_N,\,\mathrm{rad}}(\boz)$, $H^q_{\fai,\,\mathrm{at}}(\boz)$, $h_{\fai,\,L_N,\,\mathrm{max}}(\boz)$ and
$h_{\fai,\,L_N,\,\mathrm{rad}}(\boz)$ coincide with equivalent quasi-norms.
To finish the proof of Theorem \ref{t4.2}, we only need to show that the spaces
$H_{\fai,\,L_N,\,\mathrm{rad}}(\boz)$ and $\wz{H}_{\fai,\,z}(\boz)$ coincide with equivalent quasi-norms.

We first prove that
\begin{equation}\label{4.16}
\lf[H_{\fai,\,L_N,\,\mathrm{rad}}(\boz)\cap L^2(\boz)\r]\subset\lf[\wz{H}_{\fai,\,z}(\boz)\cap L^2(\boz)\r].
\end{equation}
Let $f\in [H_{\fai,\,L_N,\,\mathrm{rad}}(\boz)\cap L^2(\boz)]$. We first assume that $\boz$ is bounded.
Then there exist $\{\lz_0\}\cup\{\lz_j\}_{j=1}^\fz\subset\cc$, a $(\fai,\,\fz)$-single-atom $\az_0$
and a sequence $\{\az_j\}_{j=1}^\fz$ of $(\fai,\,\fz,\,M)_L$-atoms
such that
$$
f=\lz_0\az_0+\sum_{j=1}^\fz\lz_j\az_j\ \ \text{in}\ \  L^2(\boz)\ \ \text{and}\ \
\blz\lf(\{\lz_0\az_0\}\bigcup\lf\{\lz_j\az_j\r\}_{j=1}^\fz\r)\ls\|f\|_{H_{\fai,\,L_N,\,\mathrm{rad}}(\boz)}.
$$

Obviously, $f_1:=\lz_0\az_0\in L^\fz(\boz)$ and
\begin{equation}\label{4.17}
\|f_1\|_{L^\fz(\boz)}\le
\blz\lf(\{\lz_0\az_0\}\bigcup\lf\{\lz_j\az_j\r\}_{j=1}^\fz\r)\ls\|f\|_{H_{\fai,\,L_N,\,\mathrm{rad}}(\boz)}.
\end{equation}
Let $\az$ be a $(\fai,\,\fz,\,M)_L$-atom associated with the ball $B_\boz:=B\cap\boz$.
By using the fact that, for any $x\in\boz$ and $t\in(0,\fz)$,
$$\int_\boz K_t(x,y)\,dy=1
$$
and an argument similar to that used in the proof of \cite[Lemma 9.1]{hlmmy}, we know that
$$\int_{\boz}\az(x)\,dx=0.
$$
Denote by $\wz{\az}$ the zero extension out of $\boz$ of $\az$. Then
$\int_{\rn}\wz{\az}(x)\,dx=0$, $\supp(\wz{\az})\subset B$ and
$$\|\wz{\az}\|_{L^\fz(\rn)}=\|\az\|_{L^\fz(\boz)}\le\|\mathbf{1}_{B\cap\boz}\|^{-1}_{L^\fai(\boz)}
\ls\|\mathbf{1}_{B}\|^{-1}_{L^\fai(\rn)}.$$
Thus, $\wz{\az}$ is a harmless constant multiple of a $(\fai,\,\fz,\,0)$-atom.

For any $j$, denote by $\wz{\az}_j$ the zero extension out of $\boz$ of $\az_j$.
Then, for any $j$, $\wz{\az}_j$ is a harmless constant multiple of a $(\fai,\,\fz,\,0)$-atom.
Let $f_2:=\sum_{j=1}^\fz\lz_j\az_j$ and $\wz{f}_2:=\sum_{j=1}^\fz\lz_j\wz{\az}_j$.
Then $\wz{f}_2$ is the zero extension out of $\boz$ of $f_2$,
$\wz{f}_2\in H_\fai(\rn)$ and
\begin{equation}\label{4.18}
\lf\|\wz{f}_2\r\|_{H_\fai(\rn)}\ls\blz\lf(\lf\{\lz_j\az_j\r\}_{j=1}^\fz\r)
\ls\|f\|_{H_{\fai,\,L_N,\,\mathrm{rad}}(\boz)},
\end{equation}
which further implies that $f_2\in H_{\fai,\,z}(\boz)$ and $\|f_2\|_{H_{\fai,\,z}(\boz)}\ls\|f\|_{H_{\fai,\,L_N,\,\mathrm{rad}}(\boz)}$.
From this, $f=f_1+f_2$, \eqref{4.17} and the definition of $\wz{H}_{\fai,\,z}(\boz)$, we deduce that
$$\|f\|_{\wz{H}_{\fai,\,z}(\boz)}\le\|f_1\|_{L^\fz(\boz)}+\|f_2\|_{H_{\fai,\,z}(\boz)}
\ls\|f\|_{H_{\fai,\,L_N,\,\mathrm{rad}}(\boz)}.
$$
This finishes the proof of \eqref{4.16} in the case that $\boz$ is bounded.
When $\boz$ is unbounded, similarly to the proof of \eqref{4.18}, we know that \eqref{4.16} also
holds true in this case.

We now show that
\begin{equation}\label{4.19}
\lf[\wz{H}_{\fai,\,z}(\boz)\cap L^2(\boz)\r]\subset\lf[H_{\fai,\,L_N,\,\mathrm{rad}}(\boz)\cap L^2(\boz)\r].
\end{equation}
Let $f\in[\wz{H}_{\fai,\,z}(\boz)\cap L^2(\boz)]$.
We first assume that $\boz$ is bounded. Then there exist $f_1\in L^\fz(\boz)$
and $f_2\in H_{\fai,\,z}(\boz)$ such that $f=f_1+f_2$ and
$$\|f_1\|_{L^\fz(\boz)}+\|f_2\|_{H_{\fai,\,z}(\boz)}\ls\|f\|_{\wz{H}_{\fai,\,z}(\boz)}.
$$
Let $\lz_0:=\|f_1\|_{L^\fz(\boz)}\|\mathbf{1}_\boz\|_{L^\fai(\boz)}$ and
$\az_0:=f_1/[\|f_1\|_{L^\fz(\boz)}\|\mathbf{1}_\boz\|_{L^\fai(\boz)}]$. Then $f_1=\lz_0\az_0$
and $\az_0$ is a $(\fai,\fz)$-single-atom. Thus, by Theorem \ref{t1.1}, we know that
$f_1\in H_{\fai,\,L_N,\,\mathrm{rad}}(\boz)$ and
\begin{equation}\label{4.20}
\|f_1\|_{H_{\fai,\,L_N,\,\mathrm{rad}}(\boz)}\ls\|f_1\|_{L^\fz(\boz)}.
\end{equation}

Let $\wz{f}_2$ be the zero extension out of $\boz$ of $f_2$.
Then $\wz{f}_2\in [H_{\fai}(\rn)\cap L^2(\rn)]$. For any $t\in(0,[\diam(\boz)]^2)$,
$x\in\boz$ and $y\in\boz$, let
$$F_{x,\,t}(y):= |B(x,\sqrt{t})\cap\boz|\exp\lf\{\frac{|x-y|^2}{ct}\r\}K_t(x,y).$$
Then $F_{x,\,t}$ is a bounded H\"older continuous function on $\boz$.
An argument similar to that used in \cite[p.\,156]{ar03} shows that $F_{x,\,t}$ can be
extended to a bounded H\"older continuous function on $\rn$ with the
H\"older index $\wz\dz\in (0,\dz_2)$. Denote this extension by
$\wz{F}_{x,\,t}$. For any $t\in(0,[\diam(\boz)]^2)$, $x\in\boz$ and $y\in\boz$,
let
$$\wz{K_t}(x,y):=|B(x,\sqrt{t})\cap\boz|^{-1}\exp\lf\{-\frac{|x-y|^2}{ct}\r\}\wz{F}_{x,\,t}(y).$$
Obviously, for any
$t\in(0,[\diam(\boz)]^2)$ and $x,\,y\in\boz$, $K_t (x,y)=\wz{K_t}(x,y)$.

By Lemma \ref{l4.1}, we find that there exists a sequence $\{b_j\}_j$ of
constant multiples of $(\fai,\,\fz,\,0)$-atoms such that
$\wz{f}_2=\sum_j b_j$ in $L^2(\rn)$, which implies that $\wz{f}=\sum_j
b_j$ almost everywhere. Thus, for any $t\in(0,\diam(\boz))$ and
$x\in\boz$,
$$e^{-t^2L_N}(f_2)(x)=\int_{\boz}K_{t^2}(x,y)f_2(y)\,dy=\int_{\rn}\wz{K}_{t^2}
(x,y)\wz{f}_2(y)\,dy=\sum_j \int_{\rn}\wz{K}_{t^2}(x,y)b_j(y)\,dy,
$$
which implies that $(f_2)^+_{L_N,\,\loc}\le\sum_j (b_j)^+_{L_N,\,\loc}$. From
this and an argument similar to \eqref{2.17}, we deduce that
$f_2\in [h_{\fai,\,L_N,\,\mathrm{rad}}(\boz)\cap L^2(\boz)]$
and
\begin{equation}\label{4.21}
\|f_2\|_{h_{\fai,\,L_N,\,\mathrm{rad}}(\boz)}\ls\|f_2\|_{H_{\fai,\,z}(\boz)},
\end{equation}
which, combined with \eqref{4.20} and $f=f_1+f_2$, further implies that
\begin{align*}
\|f\|_{h_{\fai,\,L_N,\,\mathrm{rad}}(\boz)}&\ls
\|f_1\|_{h_{\fai,\,L_N,\,\mathrm{rad}}(\boz)}+\|f_2\|_{h_{\fai,\,L_N,\,\mathrm{rad}}(\boz)}
\ls\|f_1\|_{H_{\fai,\,L_N,\,\mathrm{rad}}(\boz)}+\|f_2\|_{H_{\fai,\,z}(\boz)}\\
&\ls\|f_1\|_{L^\fz(\boz)}+\|f_2\|_{H_{\fai,\,z}(\boz)}\ls\|f\|_{\wz{H}_{\fai,\,z}(\boz)}.
\end{align*}
This finishes the proof of \eqref{4.19} in the case that $\boz$ is bounded.
When $\boz$ is unbounded, similarly to the proof of \eqref{4.21}, we know that
\eqref{4.19} also holds true in this case.

Thus,  by \eqref{4.16} and \eqref{4.19} and a
density argument, we conclude that the spaces $H_{\fai,\,L_N,\,\mathrm{rad}}(\boz)$,
and $\wz{H}_{\fai,\,z}(\boz)$ coincide with equivalent quasi-norms,
which completes the proof of Theorem \ref{t4.2}.
\end{proof}

Now, let $\fai$ be as in Definition \ref{d1.2} and $\boz$ a subdomain in $\rn$.
Then, for any $f\in\cs'(\rn)$ and $m\in\nn$,
the \emph{local non-tangential grand maximal function} $f^\ast_{m,\,\loc}$
of $f$ is defined by setting, for any $x\in\rn$,
$$f^\ast_{m,\,\loc}(x):=\sup_{\phi\in\cs_m(\rn)}\sup_{|y-x|<t,\,t\in(0,1)}|f\ast\phi_t(y)|.
$$
When $m(\fai):=\lfz n[q(\fai)/i(\fai)-1]\rfz$, where $q(\fai)$ and $i(\fai)$ are,
respectively, as in \eqref{1.6} and \eqref{1.5}, we denote $f^\ast_{m(\fai),\,\loc}$ simply
by $f^{\ast}_{\loc}$. The \emph{local Musielak-Orlicz-Hardy space} $h_{\fai}(\rn)$ is defined via replacing
$f^\ast$ by $f^{\ast}_{\loc}$ in Definition \ref{d4.2}. Moreover, the \emph{local Musielak-Orlicz-Hardy spaces}
$h_{\fai,\,z}(\boz)$ and $h_{\fai,\,r}(\boz)$ are defined via replacing $H_\fai(\rn)$ by $h_\fai(\rn)$
in the definitions of $H_{\fai,\,z}(\boz)$ and $H_{\fai,\,r}(\boz)$, respectively.

We then have the following conclusions.
\begin{corollary}\label{c4.3}
Let $\boz$ be a bounded Lipschitz domain in $\rn$.
\begin{itemize}
\item[{\rm(i)}] If $nq(\fai)/i(\fai)<n+\dz_1$, then
$H_{\fai,\,r}(\boz)=h_{\fai,\,r}(\boz)$
with equivalent quasi-norms, where $\dz_1\in(0,1]$ is as in \eqref{4.4}.

\item[{\rm(ii)}] If $nq(\fai)/i(\fai)<n+\dz_2$, then
$\wz{H}_{\fai,\,z}(\boz)=h_{\fai,\,z}(\boz)$ with equivalent quasi-norms,
where $\dz_2\in(0,1]$ is as in \eqref{4.15}.
\end{itemize}
\end{corollary}

\begin{proof}
To show (i), by using the atomic characterization of $h_\fai(\rn)$ established in \cite[Theorem 5.7]{yys1} and similarly to the proof of \eqref{4.5}, we know that
\begin{equation}\label{4.22}
\lf[h_{\fai,\,r}(\boz)\cap L^2 (\boz)\r]\subset\lf[h_{\fai,\,L_D,\,\mathrm{rad}}(\boz)\cap L^2(\boz)\r].
\end{equation}
Moreover, from \eqref{4.9} and the fact that $H_\fai(\rn)\subset h_\fai(\rn)$, it follows that
\begin{equation}\label{4.23}
\lf[H_{\fai,\,L_D,\,\mathrm{rad}}(\boz)\cap L^2(\boz)\r]\subset\lf[H_{\fai,\,r}(\boz)\cap L^2 (\boz)\r]\subset\lf[h_{\fai,\,r}(\boz)\cap L^2 (\boz)\r].
\end{equation}
By Theorem \ref{t1.2}, we conclude that
$$\lf[H_{\fai,\,L_D,\,\mathrm{rad}}(\boz)\cap L^2(\boz)\r]=\lf[h_{\fai,\,L_D,\,\mathrm{rad}}(\boz)\cap L^2(\boz)\r]
$$
with equivalent quasi-norms, which, combined with \eqref{4.22} and \eqref{4.23}, further implies that
$$\lf[H_{\fai,\,r}(\boz)\cap L^2(\boz)\r]=\lf[h_{\fai,\,r}(\boz)\cap L^2(\boz)\r].
$$
From this and the facts that $H_{\fai,\,r}(\boz)\cap L^2(\boz)$ and
$h_{\fai,\,r}(\boz)\cap L^2(\boz)$ are dense, respectively, in the spaces
$H_{\fai,\,r}(\boz)$ and $h_{\fai,\,r}(\boz)$, and a density argument,
we deduce that the spaces $H_{\fai,\,r}(\boz)=h_{\fai,\,r}(\boz)$
with equivalent quasi-norms, which completes the proof of (i).

To show (ii), from the definitions of $H_{\fai,\,z}(\boz)$ and
$h_{\fai,\,z}(\boz)$ and the atomic characterizations of $H_{\fai}(\rn)$ and $h_{\fai}(\rn)$ (see, for example, \cite[Theorem 5.7]{yys1}), we deduce that
\begin{equation}\label{4.24}
\lf[\wz{H}_{\fai,\,z}(\boz)\cap L^2(\boz)\r]\subset\lf[h_{\fai,\,z}(\boz)\cap L^2(\boz)\r].
\end{equation}
By using the atomic characterization of $h_\fai(\rn)$ and similarly to the proof of \eqref{4.19}, we find that
\begin{equation}\label{4.25}
\lf[h_{\fai,\,z}(\boz)\cap L^2 (\boz)\r]\subset\lf[h_{\fai,\,L_N,\,\mathrm{rad}}(\boz)\cap L^2(\boz)\r].
\end{equation}
From Theorem \ref{t1.2}, it follows that
$$\lf[H_{\fai,\,L_N,\,\mathrm{rad}}(\boz)\cap L^2(\boz)\r]=\lf[h_{\fai,\,L_N,\,\mathrm{rad}}(\boz)\cap L^2(\boz)\r]
$$
with equivalent quasi-norms, which, together with \eqref{4.16}, \eqref{4.24} and \eqref{4.25}, further implies that
$$\lf[\wz{H}_{\fai,\,z}(\boz)\cap L^2(\boz)\r]=\lf[h_{\fai,\,z}(\boz)\cap L^2(\boz)\r].
$$
By this, we further know that $\wz{H}_{\fai,\,z}(\boz)=h_{\fai,\,z}(\boz)$
with equivalent quasi-norms.
This finishes the proof of (ii) and hence of Corollary \ref{c4.3}.
\end{proof}

When $\fai(x,t):=t^p$, with $p\in(0,1]$, for any $x\in\boz$ and $t\in[0,\fz)$,
we denote the spaces $h_{\fai,\,r}(\boz)$ and $h_{\fai,\,z}(\boz)$, respectively,
simply by $h^p_{r}(\boz)$ and $h^p_{z}(\boz)$. Then we have the following remark.

\begin{remark}\label{r4.4}
The equivalences $H^1_{r}(\boz)=h^1_{r}(\boz)$ and $\wz{H}^1_{z}(\boz)=h^1_{z}(\boz)$
were obtained in \cite[Remark 17]{ar03}. When $p\in(\frac{n}{n+\dz_1},1)$ or $p\in(\frac{n}{n+\dz_2},1)$,
the corresponding equivalences $H^p_{r}(\boz)=h^p_{r}(\boz)$ and $\wz{H}^p_{z}(\boz)=h^p_{z}(\boz)$
obtained in Corollary \ref{c4.3} are new.
\end{remark}

\bigskip

\noindent Sibei Yang

\medskip

\noindent School of Mathematics and Statistics, Gansu Key Laboratory of Applied Mathematics and
Complex Systems, Lanzhou University, Lanzhou 730000, People's Republic of China

\smallskip

\noindent{\it E-mail:} \texttt{yangsb@lzu.edu.cn}

\bigskip

\noindent Dachun Yang (Corresponding author)

\medskip

\noindent Laboratory of Mathematics and Complex Systems (Ministry of Education of China),
School of Mathematical Sciences, Beijing Normal University, Beijing 100875, People's Republic of China

\smallskip

\noindent{\it E-mail:} \texttt{dcyang@bnu.edu.cn}


\begin{thebibliography}{99}

\bibitem{abdr17} V. Almeida, J. J. Betancor, E. Dalmasso and L. Rodr\'iguez-Mesa,
Local Hardy spaces with variable exponents associated to non-negative
self-adjoint operators satisfying Gaussian estimates, arXiv: 1712.06710.

\vspace{-0.3cm}

\bibitem{al11} B. T. Anh and J. Li, Orlicz-Hardy spaces associated to operators
satisfying bounded $H_\fz$ functional calculus and Davies-Gaffney estimates,
J. Math. Anal. Appl. 373 (2011), 485-501.

\vspace{-0.3cm}

\bibitem{adm} P. Auscher, X. T. Duong and A. McIntosh,
Boundedness of Banach space valued singular integral operators and
Hardy spaces, Unpublished Manuscript, 2005.

\vspace{-0.3cm}

\bibitem{ar03} P. Auscher and E. Russ, Hardy spaces and divergence operators on
strongly Lipschitz domains of $R^n$, J. Funct. Anal. 201 (2003), 148-184.

\vspace{-0.3cm}

\bibitem{at01a} P. Auscher and Ph. Tchamitchian, Gaussian estimates
for second order elliptic divergence operators on Lipschitz and
$C^1$ domains, in: Evolution equations and their applications in
Physical and Life Sciences (Bad Herrenalb, 1998), 15-32, Lecture
Notes in Pure and Applied Math., 215, Dekker, New York, 2001.

\vspace{-0.3cm}

\bibitem{ap16}
M. Avci and A. Pankov, Multivalued elliptic operators with nonstandard growth,
Adv. Nonlinear Anal. 7 (2018), 35-48.

\vspace{-0.3cm}

\bibitem{bcg15na}
P. Baroni, M. Colombo and G. Mingione,
Regularity for general functionals with double phase,
Calc. Var. Partial Differential Equations 57 (2018), no. 2, Art. 62, 48 pp.

\vspace{-0.3cm}

\bibitem{bfg10} A. Bonami, J. Feuto and S. Grellier, Endpoint for the
DIV-CURL lemma in Hardy spaces, Publ. Mat. 54 (2010), 341-358.

\vspace{-0.3cm}

\bibitem{bgk12} A. Bonami, S. Grellier and L. D. Ky, Paraproducts and products
of functions in $BMO(\rn)$ and $H^1(\rn)$ through wavelets,
J. Math. Pures Appl. (9) 97 (2012), 230-241.

\vspace{-0.3cm}

\bibitem{bijz07} A. Bonami, T. Iwaniec, P. Jones  and
M. Zinsmeister, On the product of functions in $BMO$ and $H^1$, Ann
Inst Fourier (Grenoble) 57 (2007), 1405-1439.

\vspace{-0.3cm}

\bibitem{bckyy} T. A. Bui, J. Cao, L. D. Ky, D. Yang and
S. Yang, Weighted Hardy spaces associated with operators satisfying reinforced
off-diagonal estimates, Taiwanese J. Math. 17 (2013), 1127-1166.

\vspace{-0.3cm}

\bibitem{bckyy13b}
T. A. Bui, J. Cao, L. D. Ky, D. Yang and S. Yang,
Musielak-Orlicz-Hardy spaces associated with operators satisfying reinforced
off-diagonal estimates, Anal. Geom. Metr. Spaces 1 (2013), 69-129.

\vspace{-0.3cm}

\bibitem{bd18} T. A. Bui and X. T. Duong, Regularity estimates for Green
operators of Dirichlet and Neumann problems on weighted Hardy spaces,
arXiv: 1808.09639.

\vspace{-0.3cm}

\bibitem{bdl1} T. A. Bui, X. T. Duong and F. K. Ly,
Maximal function characterizations for Hardy spaces on spaces of homogeneous type
with finite measure and applications,
arXiv: 1804.01347.

\vspace{-0.3cm}

\bibitem{bdl}
T. A. Bui, X. T. Duong and F. K. Ly, Maximal function characterizations for
new local Hardy type spaces on spaces of homogeneous type,
Trans. Amer. Math. Soc. (2018), doi: 10.1090/tran/7289.

\vspace{-0.3cm}

\bibitem{c77} A. Calder\'on, An atomic decomposition of distributions in
parabolic $H^p$ spaces, Adv. Math. 25 (1977), 216-225.

\vspace{-0.3cm}

\bibitem{ccyy13}  J. Cao, D.-C. Chang, D. Yang and S. Yang,
Weighted local Orlicz-Hardy spaces on domains and their applications
in inhomogeneous Dirichlet and Neumann problems, Trans. Amer. Math. Soc.
365 (2013), 4729-4809.

\vspace{-0.3cm}

\bibitem{CCYY16tams}
J. Cao, D.-C. Chang, D. Yang and S. Yang,
Riesz transform characterizations of Musielak-Orlicz-Hardy spaces,
Trans. Amer. Math. Soc. 368 (2016), 6979-7018.

\vspace{-0.3cm}

\bibitem{cds99} D.-C. Chang, G. Dafni and E. M. Stein,  Hardy spaces,
$\mathop\mathrm{BMO}$ and boundary value problems for the Laplacian
on a smooth domain in $\rn$, Trans. Amer. Math. Soc. 351 (1999),
1605-1661.

\vspace{-0.3cm}

\bibitem{cfyy16} D.-C. Chang, Z. Fu, D. Yang and S. Yang, Real-variable
characterizations of Musielak-Orlicz-Hardy spaces associated with
Schr\"odinger operators on domains,
Math. Methods Appl. Sci. 39 (2016), 533-569.

\vspace{-0.3cm}

\bibitem{cks92} D.-C. Chang, S. G. Krantz and E. M. Stein,
Hardy spaces and elliptic boundary value problems, Contem.
Math. 137 (1992), 119-131.

\vspace{-0.3cm}

\bibitem{cks93} D.-C. Chang, S. G. Krantz and E. M. Stein,
$H^p$ theory on a smooth domain in ${\mathbb R}^N$ and elliptic
boundary value problems, J. Funct. Anal. 114 (1993), 286-347.

\vspace{-0.3cm}

\bibitem{cn95} D. Cruz-Uribe and C. J. Neugebauer, The structure
of the reverse H\"older classes, Trans. Amer. Math. Soc. 347
(1995), 2941-2960.

\vspace{-0.3cm}

\bibitem{cms85} R. R. Coifman, Y. Meyer and E. M. Stein,
Some new function spaces and their applications  to harmonic
analysis, J. Funct. Anal. 62 (1985), 304-335.

\vspace{-0.3cm}

\bibitem{cw71} R. R. Coifman and G. Weiss, Analyse Harmonique Non-Commutative sur Certains
Espaces Homog\`enes, (French) \'Etude de Certaines Int\'egrales Singuli\`eres,
Lecture Notes in Mathematics 242, Springer-Verlag, Berlin-New York, 1971.

\vspace{-0.3cm}

\bibitem{cw77} R. R. Coifman and G. Weiss, Extensions of Hardy spaces and
their use in analysis, Bull. Amer. Math. Soc. 83 (1977), 569-645.

\vspace{-0.3cm}

\bibitem{cm15arma}
M. Colombo and G. Mingione,
Bounded minimisers of double phase variational integrals,
Arch. Ration. Mech. Anal. 218 (2015), 219-273.

\vspace{-0.3cm}

\bibitem{d89} E. B. Davies, Heat Kernels and Spectral Theory, Cambridge Tracts
in Mathematics 92, Cambridge University Press, Cambridge, 1989.

\vspace{-0.3cm}

\bibitem{dkkp17} S. Dekel, G. Kerkyacharian,  G. Kyriazis and P. Petrushev,
Hardy spaces associated with non-negative self-adjoint operators,
Studia Math. 239 (2017), 17-54.

\vspace{-0.3cm}

\bibitem{dhmmy} X. T. Duong, S. Hofmann, D. Mitrea, M. Mitrea
and L. Yan, Hardy spaces and regularity for  the inhomogeneous
Dirichlet and Neumann problems, Rev. Mat. Iberoam. 29 (2013), 183-236.

\vspace{-0.3cm}

\bibitem{dl13} X. T. Duong and J. Li, Hardy spaces associated to operators
satisfying Davies-Gaffney estimates and bounded holomorphic functional calculus,
J. Funct. Anal. 264 (2013), 1409-1437.

\vspace{-0.3cm}

\bibitem{dy04} X. T. Duong and L. Yan, On the atomic decomposition for Hardy
spaces on Lipschitz domains of ${\mathbb R}^n$, J. Funct. Anal. 215
(2004), 476-486.

\vspace{-0.3cm}

\bibitem{d05} J. Dziuba\'nski, Note on $H^1$ spaces related to degenerate
Schr\"odinger operators, Illinois J. Math. 49 (2005), 1271-1297.

\vspace{-0.3cm}

\bibitem{dz02} J. Dziuba\'nski and J. Zienkiewicz, $H^p$ spaces for
Schr\"odinger operators, in: Fourier Analysis and Related Topics, 45-53, Banach Center Publ., 56, Polish Acad. Sci., Warsaw, 2002.

\vspace{-0.3cm}

\bibitem{fs72} C. Fefferman and E. M. Stein, $H^p$ spaces of several
variables, Acta Math. 129 (1972), 137-193.

\vspace{-0.3cm}

\bibitem{fcy17} X. Fu, D.-C. Chang and D. Yang, Recent progress in
bilinear decompositions, Appl. Anal. Optim. 1 (2017), 153-210.


\vspace{-0.3cm}

\bibitem{fy17} X. Fu and D. Yang, Products of functions in $H^1_\rho(\cx)$ and
$\mathop\mathrm{BMO}_\rho(\cx)$ over RD-spaces and applications to Schr\"odinger
operators, J. Geom. Anal. 27 (2017), 2938-2976.

\vspace{-0.3cm}

\bibitem{fyl17} X. Fu, D. Yang and Y. Liang, Products of functions
in $\mathop\mathrm{BMO}(\cx)$ and $H^1_{\mathrm{at}}(\cx)$ via wavelets
over spaces of homogeneous type, J. Fourier Anal. Appl. 23 (2017), 919-990.

\vspace{-0.3cm}

\bibitem{gra1} L. Grafakos, Modern Fourier Analysis, Third edition,
Graduate Texts in Mathematics 250, Springer, New York, 2014.

\vspace{-0.3cm}

\bibitem{GSZ}
P. Gwiazda, I. Skrzypczak and A. Zatorska-Goldstein,
Existence of renormalized solutions to elliptic equation in Musielak-Orlicz space,
J. Differential Equations 264 (2018), 341-377.

\vspace{-0.3cm}

\bibitem{hhlt18} Y. Han, Y. Han, J. Li and C. Tan, Hardy and Carleson
measure spaces associated with operators on spaces of homogeneous type,
Potential Anal. 49 (2018), 247-265.

\vspace{-0.3cm}

\bibitem{hsv07} E. Harboure, O. Salinas and B. Viviani, A look at
${\rm BMO}_\phi(\omega)$ through Carleson measures, J. Fourier
Anal. Appl. 13 (2007), 267-284.

\vspace{-0.3cm}

\bibitem{hhlt13}
P. Harjulehto, P. H\"ast\"o, V. Latvala and O. Toivanen,
Critical variable exponent functionals in image restoration,
Appl. Math. Lett. 26 (2013), 56-60.

\vspace{-0.3cm}

\bibitem{hhk16}
P. Harjulehto, P. H\"ast\"o and R. Kl\`en,
Generalized Orlicz spaces and related PDE,
Nonlinear Anal. 143 (2016), 155-173.

\vspace{-0.3cm}

\bibitem{hlmmy} S. Hofmann, G. Lu, D. Mitrea, M. Mitrea and L. Yan,
Hardy spaces associated to non-negative self-adjoint operators
satisfying Davies-Gaffney estimates, Mem. Amer. Math. Soc. 214
(2011), no. 1007, vi+78 pp.

\vspace{-0.3cm}

\bibitem{hm09} S. Hofmann and S. Mayboroda,
Hardy and BMO spaces associated to divergence form elliptic
operators, Math. Ann. 344 (2009), 37-116.

\vspace{-0.3cm}

\bibitem{hmm11} S. Hofmann, S. Mayboroda and A. McIntosh,
Second order elliptic operators with complex bounded
measurable coefficients in $L^p$, Sobolev and Hardy spaces,
Ann. Sci. \'Ec. Norm. Super. (4) 44 (2011), 723-800.


\vspace{-0.3cm}

\bibitem{hyy} S. Hou, D. Yang and S. Yang, Lusin area function and molecular
characterizations of Musielak-Orlicz Hardy spaces and their applications,
Commun. Contemp. Math. 15 (2013),  no. 6, 1350029, 37 pp.

\vspace{-0.3cm}

\bibitem{hll18} J. Huang, P. Li and Y. Liu, Regularity properties of the heat
kernel and area integral characterization of Hardy
space $H^1_{\mathcal{L}}$ related to degenerate Schr\"odinger operators,
J. Math. Anal. Appl.  466 (2018), 447-470.

\vspace{-0.3cm}

\bibitem{jy10} R. Jiang and D. Yang, New Orlicz-Hardy spaces associated
 with divergence form elliptic operators, J. Funct.
 Anal. 258 (2010), 1167-1224.

\vspace{-0.3cm}

\bibitem{jy11} R. Jiang and D. Yang, Orlicz-Hardy spaces associated with operators
satisfying Davies-Gaffney estimates, Commun. Contemp. Math. 13 (2011), 331-373.

\vspace{-0.3cm}

\bibitem{jyy12} R. Jiang, Da. Yang and Do. Yang,
Maximal function characterizations of Hardy spaces associated with
magnetic Schr\"odinger operators, Forum Math. 24 (2012), 471-494.

\vspace{-0.3cm}

\bibitem{jn87} R. Johnson and C. J. Neugebauer, Homeomorphisms preserving $A_p$,
Rev. Mat. Iberoam. 3 (1987), 249-273.

\vspace{-0.3cm}

\bibitem{ana14}
M. Kbiri Alaoui, T. Nabil and M. Altanji,
On some new non-linear diffusion models for the image filtering,
Appl. Anal. 93 (2014), 269-280.

\vspace{-0.3cm}

\bibitem{k13} L. D. Ky, Bilinear decompositions and commutators of
singular integral operators, Trans. Amer. Math. Soc. 365 (2013), 2931-2958.

\vspace{-0.3cm}

\bibitem{k14} L. D. Ky, New Hardy spaces of Musielak-Orlicz type and
boundedness of sublinear operators, Integral Equations Operator Theory 78 (2014),
115-150.

\vspace{-0.3cm}

\bibitem{LY15} Y. Liang and D. Yang, Intrinsic square function
characterizations of Musielak-Orlicz Hardy spaces,
Trans. Amer. Math. Soc. 367 (2015), 3225-3256.

\vspace{-0.3cm}

\bibitem{lcfy17} L. Liu, D.-C. Chang, X. Fu and D. Yang, Endpoint boundedness of
commutators on spaces of homogeneous type, Appl. Anal. 96 (2017), 2408-2433.

\vspace{-0.3cm}

\bibitem{lcfy18} L. Liu, D.-C. Chang, X. Fu and D. Yang, Endpoint estimates
of linear commutators on Hardy spaces over spaces of homogeneous type,
Math. Meth. Appl. Sci. (2018), doi: 10.1002/mma.5112.

\vspace{-0.3cm}

\bibitem{lyy} L. Liu, D. Yang and W. Yuan, Bilinear decompositions for products
of Hardy and Lipschitz spaces on spaces of homogeneous type,
Dissertationes Math. (Rozprawy Mat.) (to appear).

\vspace{-0.3cm}

\bibitem{ls13} S. Liu and L. Song, An atomic decomposition of weighted Hardy spaces
associated to self-adjoint operators, J. Funct. Anal. 265 (2013), 2709-2723.

\vspace{-0.3cm}

\bibitem{mw}
B. Matejczyk and A. Wr\'{o}blewska-Kami\'{n}ska,
Unsteady flows of heat-conducting non-Newtonian fluids in Musielak-Orlicz spaces,
Nonlinearity 31 (2018), 701-727.

\vspace{-0.3cm}

\bibitem{m94} S. M\"uller, Hardy space methods for nonlinear partial
differential equations, Tatra Mt. Math. Publ. 4 (1994), 159-168.

\vspace{-0.3cm}

\bibitem{m83} J. Musielak, Orlicz Spaces and Modular Spaces, Lecture
Notes in Mathematics 1034, Springer-Verlag, Berlin, 1983.

\vspace{-0.3cm}

\bibitem{mo59}
J. Musielak and W. Orlicz,
On modular spaces, Studia Math. 18 (1959), 49-65.

\vspace{-0.3cm}

\bibitem{ny97} E. Nakai and K. Yabuta, Pointwise multipliers for functions
of weighted bounded mean oscillation on spaces of homogeneous type,
Math. Japon. 46 (1997), 15-28.

\vspace{-0.3cm}

\bibitem{n50}
H. Nakano, Modulared Semi-Ordered Linear Spaces,
Maruzen Co., Ltd., Tokyo, 1950.

\vspace{-0.3cm}

\bibitem{o04}  E. M. Ouhabaz, Analysis of Heat Equations on Domains,
Princeton University Press, Princeton, N. J., 2005.

\vspace{-0.3cm}

\bibitem{rr91} M. Rao and Z. Ren, Theory of Orlicz Spaces,
Marcel Dekker, Inc., New York, 1991.


\vspace{-0.3cm}

\bibitem{r07} E. Russ, The atomic decomposition for tent spaces on
spaces of homogeneous type, in:
CMA/AMSI Research Symposium ``Asymptotic Geometric Analysis, Harmonic Analysis, and Related Topics'', 125-135, Proc. Centre Math. Appl. Austral.
Nat. Univ., 42, Austral. Nat. Univ., Canberra, 2007.

\vspace{-0.3cm}

\bibitem{s95} L. Saloff-Coste, Parabolic Harnack inequality for
divergence-form second-order differential operators,
Potential Anal. 4 (1995), 429-467.

\vspace{-0.3cm}

\bibitem{sy10} L. Song and L. Yan, Riesz transforms associated to
Schr\"odinger operators on weighted Hardy spaces,
J. Funct. Anal. 259 (2010), 1466-1490.

\vspace{-0.3cm}

\bibitem{sy16} L. Song and L. Yan, A maximal function characterization for
Hardy spaces associated to nonnegative self-adjoint operators satisfying
Gaussian estimates, Adv. Math. 287 (2016), 463-484.

\vspace{-0.3cm}

\bibitem{sy18} L. Song and L. Yan, Maximal function characterizations for
Hardy spaces associated with nonnegative self-adjoint operators on
spaces of homogeneous type, J. Evol. Equ. 18 (2018), 221-243.

\vspace{-0.3cm}

\bibitem{st93} E. M. Stein, Harmonic Analysis: Real-variable Methods,
Orthogonality, and Oscillatory integrals, Princeton Univ. Press,
Princeton, NJ, 1993.

\vspace{-0.3cm}

\bibitem{sw60} E. M. Stein and G. Weiss, On the theory of
harmonic functions of several variables. I. The theory of
$H^p$-spaces, Acta Math. 103 (1960), 25-62.

\vspace{-0.3cm}

\bibitem{sg14}
A. \`Swierczewska-Gwiazda,
Nonlinear parabolic problems in Musielak-Orlicz spaces,
Nonlinear Anal. 98 (2014), 48-65.

\vspace{-0.3cm}

\bibitem{y08} L. Yan, Classes of Hardy spaces associated
with operators, duality theorem and applications, Trans. Amer. Math.
Soc. 360 (2008), 4383-4408.

\vspace{-0.3cm}

\bibitem{ylk17} D. Yang, Y. Liang and L. D. Ky, Real-variable Theory of
Musielak-Orlicz Hardy Spaces, Lecture Notes in Mathematics 2182, Springer, Cham, 2017.

\vspace{-0.3cm}

\bibitem{yy14} Da. Yang and Do. Yang, Maximal function characterizations of
Musielak-Orlicz-Hardy spaces associated with magnetic Schr\"odinger operators,
Front. Math. China 10 (2015), 1203-1232.

\vspace{-0.3cm}

\bibitem{yys1} D. Yang and S. Yang, Local Hardy spaces of Musielak-Orlicz
type and their applications, Sci. China Math. 55 (2012), 1677-1720.

\vspace{-0.3cm}

\bibitem{yys2} D. Yang and S. Yang, Orlicz-Hardy spaces associated
with divergence operators on unbounded strongly Lipschitz domains of
$\mathbb{R}^n$, Indiana Univ. Math. J. 61 (2012), 81-129.

\vspace{-0.3cm}

\bibitem{yys3} D. Yang and S. Yang, Real-variable characterizations
of Orlicz-Hardy spaces on strongly Lipschitz domains of $\rn$, Rev.
Mat. Iberoam. 29 (2013), 237-292.

\vspace{-0.3cm}

\bibitem{yys4} D. Yang and S. Yang, Musielak-Orlicz-Hardy spaces associated with
operators and their applications, J. Geom. Anal. 24 (2014), 495-570.

\vspace{-0.3cm}

\bibitem{yys16} D. Yang and S. Yang, Maximal function characterizations of
Musielak-Orlicz-Hardy spaces associated to non-negative self-adjoint
operators satisfying Gaussian estimates, Commun. Pure Appl. Anal. 15 (2016), 2135-2160.

\vspace{-0.3cm}

\bibitem{yyz14} D. Yang, W. Yuan and C. Zhuo, Musielak-Orlicz Besov-type
and Triebel-Lizorkin-type spaces,
Rev. Mat. Complut. 27 (2014), 93-157.

\vspace{-0.3cm}

\bibitem{yz18} D. Yang and J. Zhang, Variable Hardy spaces associated with
operators satisfying Davies-Gaffney estimates on metric measure spaces
of homogeneous type, Ann. Acad. Sci. Fenn. Math. 43 (2018), 47-87.

\vspace{-0.3cm}

\bibitem{yzz18} D. Yang, J. Zhang and C. Zhuo, Variable Hardy spaces associated
with operators satisfying Davies-Gaffney estimates, Proc. Edinb. Math.
Soc. (2) 61 (2018), 759-810.

\vspace{-0.3cm}

\bibitem{yz16} D. Yang and C. Zhuo, Molecular characterizations and dualities of
variable exponent Hardy spaces associated with operators,
Ann. Acad. Sci. Fenn. Math. 41 (2016), 357-398.

\vspace{-0.3cm}

\bibitem{zy16} D. Yang and C. Zhuo, Maximal function characterizations of variable
Hardy spaces associated with non-negative self-adjoint operators satisfying Gaussian estimates, Nonlinear Anal. 141 (2016), 16-42.

\end{thebibliography}
\end{document}